\documentclass[a4paper, reqno, 11pt]{amsart}

\usepackage[colorlinks=true,pagebackref,hyperindex]{hyperref}
\hypersetup{colorlinks = true, linkcolor = red, anchorcolor =red, citecolor = green, filecolor = red, urlcolor = blue, pdfauthor=author}

\usepackage[normalem]{ulem}
\usepackage{graphics,epic}
\usepackage{amsmath,amssymb, amsthm,mathrsfs}
\usepackage[all,2cell]{xy}
\usepackage{multicol}
\usepackage{multirow}
\usepackage{float}
\usepackage{caption}
\usepackage{enumitem}
\usepackage{subcaption}
\usepackage{adjustbox}
\usepackage{shorttoc}
\usepackage{url}
\usepackage{tikz}
\usepackage{tikz-cd}
\usetikzlibrary{positioning, arrows.meta, shapes.geometric, decorations.markings, patterns}
\usetikzlibrary{matrix,positioning,decorations.markings,arrows,decorations.pathmorphing, arrows.meta, calc,
	backgrounds,fit,positioning,shapes.symbols,chains,shadings,fadings,calc}
\tikzset{->-/.style={decoration={  markings,  mark=at position #1 with
			{\arrow{>}}},postaction={decorate}}}
\tikzset{-<-/.style={decoration={  markings,  mark=at position #1 with
			{\arrow{<}}},postaction={decorate}}}

\newtheorem{theorem}{Theorem}[section]
\newtheorem*{theorem*}{Theorem}

\newtheorem{lemma}[theorem]{Lemma}

\newtheorem{proposition}[theorem]{Proposition}
\newtheorem{corollary}[theorem]{Corollary}

\newtheorem{prop}[theorem]{Proposition}

\theoremstyle{definition}
\newtheorem{definition}[theorem]{Definition}
\newtheorem{remark}[theorem]{Remark}
\newtheorem*{notation}{Notations}
\newtheorem{example}[theorem]{Example}

\newcommand{\opname}[1]{\operatorname{\mathsf{#1}}}

\renewcommand{\ker}{\opname{ker}\nolimits}

\newcommand{\Irr}{\mathrm{Irr}}
\newcommand{\Id}{\operatorname{id}}

\newcommand{\Hom}{\opname{Hom}}

\newcommand{\Sym}{\opname{Sym}}

\makeatletter
\newcommand*\bigcdot{\mathpalette\bigcdot@{.5}}
\newcommand*\bigcdot@[2]{\mathbin{\vcenter{\hbox{\scalebox{#2}{$\m@th#1\bullet$}}}}}
\makeatother

\makeatletter
\renewcommand{\subsubsection}{%
	\@startsection{subsubsection}{3}%
	\z@{.5\linespacing\@plus.7\linespacing}{.5\linespacing}%
	{\normalfont\normalsize\bfseries}
}
\makeatother

\begin{document}
	\title[Symmetric Square]{Hochschild Cohomology of the Symmetric Square of an Annulus with Stops}
	
	\author{Xingyuan Lu}
	\address{School of Mathematics, Nanjing University, Nanjing 210093, Jiangsu, China}
	\curraddr{}
	\email{xingyuanlu@smail.nju.edu.cn}
	\author{Zhengfang Wang}
	\address{School of Mathematics, Nanjing University, Nanjing 210093, Jiangsu, China}
	\email{zhengfangw@nju.edu.cn}
	

	\thanks{}
	\date{}
	\dedicatory{}
	\begin{abstract}
		We compute the Hochschild cohomology of the partially wrapped Fukaya category of the symmetric square of an annulus with stops. Using an explicit generating set in this category, we give a description of its dg endomorphism algebra $\widetilde{\mathcal{A}}_{n_1,n_2}$ via a quiver with relations. We show that when one boundary component has a single stop, the dg algebra is formal; however, when both boundaries contain at least two stops, it is not formal, and its minimal $A_\infty$-model carries a nontrivial operation $m_3$. This allows us to compute its Hochschild cohomology via reduction systems and spectral sequences, and to construct a family of dg deformations associated to the resulting Hochschild cocycles. 
	\end{abstract}
	\maketitle

	\section{Introduction}
	
	The Fukaya category of a symplectic manifold, particularly its partially wrapped variant for Liouville manifolds with stops, has become a central object in symplectic geometry and homological mirror symmetry \cite{Fukaya-Seidel-Smith2009, Sylvan2019}. When the symplectic manifold arises as a symmetric product of a Riemann surface with stops, the resulting Fukaya category exhibits deep connections with representation theory and finite-dimensional algebras \cite{Auroux2010a, Auroux2010b}. In their seminal work, Auroux introduced the partially wrapped Fukaya category $\mathcal{W}(\operatorname{Sym}^k(\Sigma), \Lambda^{(k)})$ of the symmetric product $\operatorname{Sym}^k(\Sigma)$ with stops $\Lambda^{(k)}$ for a Riemann surface $\Sigma$ with boundary stops $\Lambda$, and gave an explicit description of its generating sets \cite{Auroux2010a}.
	
	In this setting, a striking phenomenon is that the algebraic structures governing these geometric categories often coincide with well-known objects in representation theory. For instance, when $(\Sigma, \Lambda)$ is a disk with stops, Dyckerhoff, Jasso, and Lekili showed in \cite{Dyckerhoff-Jasso-Lekili2021} that the dg endomorphism algebra of a specific generating set of $\mathcal{W}(\operatorname{Sym}^k(\Sigma), \Lambda^{(k)})$ is quasi-isomorphic to the higher Auslander algebras of type $\mathbb{A}$ studied in \cite{Iyama2007}. In \cite{Lekili-Polishchuk2020b}, Lekili and Polishchuk studied the complement of $n+2$ generic hyperplanes in $\mathbb{CP}^n$ and proved that its partially wrapped Fukaya category is equivalent to a certain categorical resolution of the derived category of the singular affine variety $x_1x_2\cdots x_{n+1}=0$. These families of algebras have been extensively studied in the context of tilting theory, cluster algebras, and singularity theory.
	
	Hochschild cohomology is a fundamental derived invariant that encodes essential information about deformation theory and higher structures \cite{Gerstenhaber1964, Keller1998}. For Fukaya categories, it is closely related to the symplectic cohomology of the underlying manifold \cite{Seidel2008, Ganatra2012, Smith2026}. Computing the Hochschild cohomology thus provides a rigorous algebraic counterpart to these intricate symplectic invariants. Seidel in \cite{Seidel2002} proposed a relationship between $A_\infty$ deformations of Fukaya categories and partial compactifications of the underlying symplectic manifolds. A complete answer to this proposal in the case of graded surface with stops is given in \cite{Barmeier-Schroll-Wang2025}.
	
	In this paper, we systematically compute the Hochschild cohomology of the partially wrapped Fukaya categories $\mathcal{W}(\operatorname{Sym}^2(\Sigma), \Lambda^{(2)})$ for an annulus with stops of type
	\(\widetilde{\mathbb A}_{n_1,n_2}\). This should be viewed as a first step toward understanding the $A_\infty$ deformations of partially wrapped Fukaya categories of symmetric products of general surfaces with stops. While the methods developed here are expected to extend to the general case, the corresponding computations become significantly more involved due to the rapid growth in the complexity of wrapped morphism spaces in higher symmetric products. We restrict to the annulus in order to present the construction and computations in a controlled setting where the relevant structures can be described explicitly.
	
	Our primary focus is the annulus case. Let $\Sigma$ be an annulus equipped with a set of stops $\Lambda$ of type $\widetilde{\mathbb{A}}_{n_1,n_2}$, meaning there are $n_1$ and $n_2$ stops on its two boundary components, respectively. We denote by $\widetilde{\mathcal{A}}_{n_1,n_2}$ the \emph{symmetric square dg algebra}, defined as the dg endomorphism algebra of a specific generator associated to a full formal arc system on $(\Sigma, \Lambda)$ within $\mathcal{W}(\operatorname{Sym}^2(\Sigma), \Lambda^{(2)})$. By explicitly describing the corresponding quivers with relations, we establish the following structural dichotomy regarding its formality.
	
	\begin{theorem}[see Theorem~\ref{thm:A_infty_qiso} and Corollary~\ref{Massey product}]
		Let $\Sigma$ be an annulus with stops $\Lambda$ of type $\widetilde{\mathbb{A}}_{n_1,n_2}$ with $n_2\geq n_1 \geq 1$. Then the symmetric square dg algebra $\widetilde{\mathcal{A}}_{n_1,n_2}$ exhibits the following properties:
		\begin{enumerate}
			\item If $n_1 = 1$, then the dg algebra $\widetilde{\mathcal{A}}_{1,n_2}$ is formal.
			\item If $n_1 >1$, then $\widetilde{\mathcal{A}}_{n_1,n_2}$ is not formal. Instead, it admits a minimal $A_\infty$-model $$(\mathcal{A}_{n_1,n_2}, m_2, m_3, \dots)$$ with $m_i = 0$ for $i>3$. Here, $\mathcal{A}_{n_1,n_2} = \mathrm{H}^*(\widetilde{\mathcal{A}}_{n_1,n_2})$ is the cohomology algebra. 
		\end{enumerate}
	\end{theorem}
	
	When $n_1=1$, this formality implies that the Hochschild cohomology of $\widetilde{\mathcal{A}}_{1,n_2}$ coincides with that of its cohomology algebra $\mathcal{A}_{1,n_2} = \mathrm{H}^*(\widetilde{\mathcal{A}}_{1,n_2})$. However, for $n_1 \ge 2$, the presence of the higher operation $m_3$ plays a critical role in the computation, since we have the following isomorphisms 
	$$\operatorname{HH}^*(\mathcal{W}(\operatorname{Sym}^2(\Sigma), \Lambda^{(2)})) \cong \operatorname{HH}^*(\widetilde{\mathcal{A}}_{n_1,n_2}) \cong \operatorname{HH}^*(\mathcal A_{n_1,n_2}, m_2, m_3).$$
	To systematically compute the Hochschild cohomology of the symmetric square dg algebra $\widetilde{\mathcal{A}}_{n_1,n_2}$, we proceed in two steps.
	
	First, using the reduction system method developed in \cite{Chouhy-Solotar2015, Barmeier-Wang2020}, we compute the Hochschild cohomology of the cohomology algebra $\mathcal{A}_{n_1,n_2}$ (viewed as an ungraded algebra, namely after forgetting its internal
	grading). Furthermore, since the partially wrapped Fukaya category naturally endows $\mathcal{A}_{n_1,n_2}$ with an internal $\mathbb{Z}$-grading determined by a single parameter $m$ up to derived equivalence (see Subsection~\ref{subsec:grading}), we refine our calculations to the bigraded Hochschild cohomology groups.
	
	\begin{theorem}[see Theorems~\ref{thm: main HH result} and \ref{thm:graded_HH}]
		Let $n_2 \geq n_1 \geq 1$. The dimensions of the  bigraded  Hochschild cohomology of the graded symmetric square algebra $\mathcal{A}_{n_1,n_2}$   are completely classified as follows:
		\begin{center}
			\renewcommand{\arraystretch}{1.2}
			\begin{tabular}{c|ccccc}
				\hline
				$\mathcal{A}_{n_1,n_2}$& $\operatorname{HH}^{0,0}$ & $\operatorname{HH}^{1,0}$ & $\operatorname{HH}^{3,-1}$ & $\operatorname{HH}^{1,-m}$ & $\operatorname{HH}^{2,m+1}$ \\
				\hline
				$n_1=1,\; n_2=2$ & $1$ & $1$ & $0$ & $1$ & $0$  \\
				$n_1=1,\; n_2\ge 3$ & $1$ & $1$ & $0$ & $0$ & $0$ \\
				$n_1=2,\; n_2=2$ & $1$ & $2$ & $1$ & $1$ & $1$ \\
				$n_1=2,\; n_2\ge 3$ & $1$ & $2$ & $1$ & $0$ & $1$ \\
				$n_2\ge n_1 \ge 3$ & $1$ & $2$ & $1$ & $0$ & $0$ \\
				\hline
			\end{tabular}
		\end{center}
		All bigraded Hochschild cohomology not listed in the table vanish.
	\end{theorem} 
	
	The bigraded computation is useful because it locates the possible
	non-formal contribution to the minimal model. For \(n_1\ge 2\), the ternary operation \(m_3\) has internal degree \(-1\).
	Under the identification of Hochschild cochains with the reduction system complex, it is represented by the generator of \(\operatorname{HH}^{3,-1}(\mathcal A_{n_1,n_2})\). This gives the first obstruction to replacing the minimal model by the ordinary graded algebra \((\mathcal A_{n_1,n_2},m_2)\).
	
	Second, we use a spectral sequence argument to bridge the gap between the graded algebra $(\mathcal A_{n_1,n_2}, m_2)$ and the $A_\infty$-minimal model $(\mathcal A_{n_1,n_2}, m_2, m_3)$ for $n_1 \geq 2$. By employing the spectral sequence associated with the internal degree filtration on the minimal $A_\infty$-model, we explicitly compute the Hochschild cohomology of the symmetric square dg algebra $\widetilde{\mathcal{A}}_{n_1,n_2}$. The first nontrivial differential in this spectral sequence is precisely induced by the Gerstenhaber bracket with $m_3$. By tracking the survival of the cocycles, we obtain our main computational result.
	
	\begin{theorem}[see Theorem~\ref{thm:dg_HH_dimensions}]
		The dimensions of the Hochschild cohomology of the symmetric square dg algebra $\widetilde{\mathcal{A}}_{n_1,n_2}$ are given as follows:
		\begin{center}
			\renewcommand{\arraystretch}{1.2}
			\begin{tabular}{c|cccc}
				\hline
				$\widetilde{\mathcal A}_{n_1,n_2}$ & $\operatorname{HH}^{0}$ & $\operatorname{HH}^{1}$ & $\operatorname{HH}^{1-m}$ & $\operatorname{HH}^{m+3}$ \\
				\hline
				$n_1=1,\; n_2=2$ & $1$ & $1$ & $1$ & $0$ \\
				$n_1=1,\; n_2\ge 3$ & $1$ & $1$ & $0$ & $0$ \\
				$n_1=2,\; n_2=2$ & $1$ & $1$ & $1$ & $1$ \\
				$n_1=2,\; n_2\ge 3$ & $1$ & $1$ & $0$ & $1$ \\
				$n_2\ge n_1\ge 3$ & $1$ & $1$ & $0$ & $0$ \\
				\hline
			\end{tabular}
		\end{center}
	\end{theorem}
	
	This allows us to construct a family of dg deformations of $\widetilde{\mathcal{A}}_{n_1,n_2}$, see Subsection \ref{deformed-alg}.  A natural question, inspired by Seidel's principle \cite{Seidel2002}, is to understand the precise geometric meaning of these explicit deformations in terms of partial compactifications of the symmetric products; we will explore this question in future work.
	

	\begin{notation}
		Throughout this paper, we fix an algebraically closed field $\Bbbk$ of characteristic zero. Unless otherwise specified, all unadorned tensor products $\otimes$ are taken over $\Bbbk$, and we adopt the cohomological grading convention for \(A_\infty\)-algebras. Thus, for an \(A_\infty\)-algebra \(A\), the structure map $m_r\colon A^{\otimes r}\longrightarrow A$
		is of degree \(2-r\). Similarly, for an \(A_\infty\)-morphism \(F=(F_r)_{r\geq 1}\), the component \(F_r\) is of degree \(1-r\).
	\end{notation}

	\section{Preliminaries}
	\subsection{Partially wrapped Fukaya categories of symmetric products of Riemann surfaces}
	Let $\Sigma$ be a Riemann surface with nonempty boundary and let $\Lambda$ be a finite set of points (called {\it stops}) on $\partial \Sigma$, such that each boundary component contains at least one stop. For an integer $k\geq 1$, we consider the smooth $k$-dimensional complex algebraic variety
	$$\Sym^k (\Sigma):= \underbrace{\Sigma \times \dots \times \Sigma}_{k\ \text{times}}/\mathfrak{S}_k ,$$ 
	where $\mathfrak{S}_k$ denotes the permutation group acting by permuting the factors of the product. Then we equip $\Sym^k (\Sigma)$ with stops $\Lambda^{(k)}$ defined by
	\[
	\Lambda^{(k)} = \{ [x_1, \dots, x_k] \in \Sym^k(\Sigma) \mid x_i \in \Lambda \text{ for some } i \}.
	\]
	
	The \textit{partially wrapped Fukaya category} $\mathcal{W}(\Sym^k(\Sigma), \Lambda^{(k)})$ was introduced by Auroux in \cite{Auroux2010a, Auroux2010b}. He also proved the following: given a set of Lagrangians $L_1$, $L_2$, $\dots$, $L_n$ on $(\Sigma, \Lambda)$ such that $\Sigma \backslash \{L_1, \dots, L_n\}$ is a disjoint union of disks each containing exactly one stop on its boundary (i.e., a full formal arc system on $(\Sigma, \Lambda)$ in the sense of \cite{Haiden-Katzarkov-Kontsevich2017} and also an admissible dissection in \cite{Amiot-Plamondon-Schroll2023}), then for $1\leq k \leq n$, the category $\mathcal{W}(\Sym^k(\Sigma), \Lambda^{(k)})$ is generated by $\binom{n}{k}$ Lagrangian submanifolds $L_{i_1}\times \cdots \times L_{i_k}$, where $\{i_1, i_2,\dots, i_k\}$ ranges over all $k$-element subsets of $[n] = \{1, 2, \dots, n\}.$ Given two such Lagrangian submanifolds $L_{i_1}\times \cdots \times L_{i_k}$ and $L_{j_1}\times \cdots \times L_{j_k}$, the morphism space  between them is defined as the following vector space (with no higher products)
	\begin{align*}
		&\Hom_{\mathcal{W}(\Sym^k(\Sigma), \Lambda^{(k)})}(L_{i_1}\times \cdots \times L_{i_k},L_{j_1}\times \cdots \times L_{j_k})\\
		=&\bigoplus_{\sigma \in \mathfrak{S}_k} \Hom_{\mathcal{W}(\Sigma, \Lambda)}(L_{i_1},L_{j_{\sigma(1)}})\otimes \Hom_{\mathcal{W}(\Sigma, \Lambda)}(L_{i_2},L_{j_{\sigma(2)}})\otimes \cdots \otimes \Hom_{\mathcal{W}(\Sigma, \Lambda)}(L_{i_k},L_{j_{\sigma(k)}}).
	\end{align*}
	
	
	Following \cite{Lipshitz-Ozsvath-Thurston2018}, we describe compositions of morphisms using strand diagrams, where morphisms are represented graphically and composition is induced by concatenation of diagrams. To obtain such a diagram from a morphism $f_1\otimes f_2\otimes \cdots \otimes f_k$ in $$\Hom_{\mathcal{W}(\Sigma, \Lambda)}(L_{i_1},L_{j_{\sigma(1)}})\otimes \Hom_{\mathcal{W}(\Sigma, \Lambda)}(L_{i_2},L_{j_{\sigma(2)}})\otimes \cdots \otimes \Hom_{\mathcal{W}(\Sigma, \Lambda)}(L_{i_k},L_{j_{\sigma(k)}}),$$ we first group the factors according to the boundary component in which they are situated: 
	$$f_{i_1,1}\otimes \cdots \otimes f_{i_{a_1},1}\otimes f_{i_1,2}\otimes \cdots \otimes f_{i_{a_{2}},2}\otimes \cdots\otimes  f_{i_{1},l}\otimes \cdots \otimes f_{i_{a_l},l},$$ 
	where $f_{i_1,r}, f_{i_2,r}\dots f_{i_{a_r},r}$ are Reeb chords (i.e., boundary segments) along the same  boundary component $\partial \Sigma_{r}$. Here the boundary components $\partial \Sigma_{1}, \dots, \partial \Sigma_{\ell}$ are distinct.  For each $1\leq r \leq l$,  one then constructs a strand diagram associated to  $f_{i_1,r}, f_{i_2,r}, \dots,  f_{i_{a_r},r}$ on the $r$-th boundary as follows:
	\begin{itemize}[leftmargin=*, label=--]
		\item  Choose a basepoint (a stop) $S$ on $\partial \Sigma_{r}$. 
		The complement $\partial \Sigma_{r} \setminus \{S\}$ is a $1$-dimensional manifold, and 
		the orientation of the boundary induces a linear order on it (with $S$ as the cut point).
		Denote by $M$ the set of all endpoints of the Lagrangians $\{L_1, \dots, L_n\}$ lying on 
		$\partial \Sigma_r \setminus \{S\}$; thus $M$ inherits this linear order.
		The endpoints of the Reeb chords are divided into two subsets, 
		the starting points  $P^{-}=\{f_{i_1,r}^{-}, f_{i_2,r}^{-}, \dots,  f_{i_{a_r},r}^{-}\}$ 
		and the ending points $P^{+}=\{f_{i_1,r}^{+}, f_{i_2,r}^{+}, \dots,  f_{i_{a_r},r}^{+}\}$.
		Note that no two Reeb chords share a common starting point or a common ending point, 
		and the two sets $P^{+}$ and $P^{-}$ have the same cardinality. 
		Both $P^{-}$ and $P^{+}$ are contained in $M$.
		\item Then the collection of Reeb chords $f_{i_1,r}, f_{i_2,r}, \dots,  f_{i_{a_r},r}$ can be represented inside the rectangle $M\times [0,1]$ by drawing upward-veering strands from $M\times \{0\}$ to $M\times \{1\}$, while a strand stays horizontal if it is an idempotent. These strands define a bijection from $P^{-}$ to $P^{+}$ sending each starting point $f_{i_s,r}^{-}$ to its corresponding endpoint $f_{i_s,r}^{+}$. In other words, the corresponding strand diagram corresponds to a partial permutation of the set $M$, which satisfies that each element is mapped to an element that is not earlier than itself 
		in the given linear order on the boundary.
	\end{itemize}
	
	Therefore, the diagram corresponding to the morphism $f_1 \otimes f_2 \otimes \dots \otimes f_k$ consists of (several) strand diagrams constructed as above. 
	
	Compositions of morphisms in $\mathcal{W}(\Sym^k(\Sigma), \Lambda^{(k)})$ are induced by those of the underlying Reeb chords in $\mathcal{W}(\Sigma, \Lambda)$:
	$$(g_1\otimes g_2\otimes \cdots \otimes g_k)\circ(f_1\otimes f_2\otimes \cdots \otimes f_k)=g_{1}f_{\sigma(1)}\otimes g_{2}f_{\sigma(2)}\otimes \cdots \otimes g_{k}f_{\sigma(k)}$$
	provided there exists a permutation $\sigma \in \mathfrak{S}_k $ such that each $g_{i}f_{\sigma(i)}$ is nonzero in $\mathcal{W}(\Sigma, \Lambda)$. An essential vanishing condition also holds: if the composition of the corresponding strand diagrams produces any pair of strands that cross each other more than once, then the resulting morphism is zero.
	
	Finally, the differential on the dg endomorphism algebra is defined as the sum of the resulting resolved diagrams over all possible ways of resolving a single crossing in the strand diagram.

	Let us explain the above construction via the following explicit example. 
	\begin{example} \label{example:strand}
		Let $\Sigma$ be an annulus with 6 stops (red points on the boundary) as shown in Figure \ref{fig:annulus-2-4}. 
		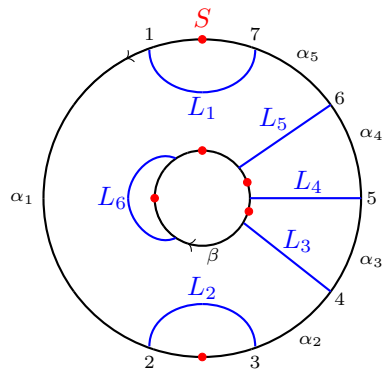
\begin{figure}[htbp]
			\centering
			\begin{tikzpicture}[scale=0.7]
				
				\tikzset{
					with arrows/.style={
						decoration={ 
							markings,
							mark=at position #1 with {\arrow{>}}
						}, 
						postaction={decorate}
					}, 
					with arrows/.default=0.5,
					vertex/.style = {
						circle, draw, fill,  
						minimum size = 2pt, 
						inner sep=1pt
					}
				}
				
				\def \radius {1.5cm}
				\def \numSectors {10}
				
				\foreach \s in {1,...,\numSectors} {
					\draw[thick] ([shift=({360/\numSectors*(\s)}:\radius-0.6cm)]0,0) 
					arc ({360/\numSectors*(\s)}:{360/\numSectors*(\s+1)}:\radius-0.6cm);
					
					
					\draw[thick] ([shift=({360/\numSectors*(\s)}:\radius+1.5cm)]0,0) 
					arc ({360/\numSectors*(\s)}:{360/\numSectors*(\s+1)}:\radius+1.5cm);
				}
				
				\draw[with arrows=0] ([shift=({360/\numSectors*(7)}:\radius-0.6cm)]0,0) 
				arc ({360/\numSectors*(7)}:{360/\numSectors*(5)}:\radius-0.6cm);
				\draw[with arrows=0.3]({360/\numSectors*(3)}:\radius+1.5cm) 
				arc ({360/\numSectors*(3)}:{360/\numSectors*(4)}:\radius+1.5cm);
				
				\draw [blue, thick] ([shift=({360/\numSectors*(8)}:\radius-0.9cm)]0.6,0.1) 
				-- ([shift=({360/\numSectors*(9)}:\radius+1.5cm)]0,0);  
				\draw [blue, thick] ([shift=({360/\numSectors*(10)}:\radius-1.0cm)]0.4,0) 
				-- ([shift=({360/\numSectors*(10)}:\radius+1.5cm)]0,0);  
				\draw [blue, thick] ([shift=({360/\numSectors*(12)}:\radius-0.9cm)]0.5,0) 
				-- ([shift=({360/\numSectors*(11)}:\radius+1.5cm)]0,0);  
				
				\node[blue] at (0,1.7) {\small $L_{1}$};
				\node[blue] at (0,-1.7) {\small $L_{2}$};
				\node[blue] at (2,0.3) {\small $L_{4}$};
				\node[blue] at (1.8,-0.8) {\small $L_{3}$};
				\node[blue] at (1.35,1.5) {\small $L_{5}$};
				\node[blue] at (-1.7,0) {\small $L_{6}$};
				\node[red] at (0,3.4) {\small $S$};

				\node[vertex,red] at ([shift=({360/\numSectors*(7.5)}:\radius+1.5cm)]0,0) {};
				\node[vertex,red] at ([shift=({360/\numSectors*(2.5)}:\radius+1.5cm)]0,0) {};
				\node[vertex,red] at ([shift=({360/\numSectors*(2.5)}:\radius+1.5cm)]-0.9,-3) {};
				\node[vertex,red] at ([shift=({360/\numSectors*(2.5)}:\radius+1.5cm)]0,-2.1) {};
				\node[vertex,red] at ([shift=({360/\numSectors*(2.5)}:\radius+1.5cm)]0.85,-2.7) {};
				\node[vertex,red] at ([shift=({360/\numSectors*(2.5)}:\radius+1.5cm)]0.88,-3.25) {};
				
				\node at (-3.4,0){\tiny $\alpha_1$};
				\node at (2.05,-2.7){\tiny $\alpha_2$};
				\node at (3.2,-1.2){\tiny $\alpha_3$};
				\node at (3.2,1.2){\tiny $\alpha_4$};
				\node at (2,2.7){\tiny $\alpha_5$};
				\node at (0.2,-1.1){\tiny $\beta$};
				\node at (-1,3.1){\tiny $1$};
				\node at (-1,-3.1){\tiny $2$};
				\node at (1,-3.1){\tiny $3$};
				\node at (2.6,-1.9){\tiny $4$};
				\node at (3.2,0){\tiny $5$};
				\node at (2.6,1.9){\tiny $6$};
				\node at (1,3.1){\tiny $7$};
				
				\draw[blue, thick] (-1, 2.8) to[in=180,out=270] (0,2);
				\draw[blue, thick] (0, 2) to[in=270,out=0] (1,2.8);
				\draw[blue, thick] (-1, -2.8) to[in=180,out=90] (0,-2);
				\draw[blue, thick] (0, -2) to[in=90,out=0] (1,-2.8);
				\draw[blue, thick] (-0.5,0.75) to[in=90,out=155] (-1.4,0);
				\draw[blue, thick] (-0.5,-0.75) to[in=270,out=205] (-1.4,0);

			\end{tikzpicture}
			\caption{Annulus of type $\widetilde{\mathbb A}_{2,4}$.}
			\label{fig:annulus-2-4}
		\end{figure}
		
		We can see that the category $\mathcal{W}(\Sym^3(\Sigma), \Lambda^{(3)})$ is generated by the $\binom{6}{3}$ objects $L_{i}\times L_{j}\times L_{k}$ with $1\leq i < j<k\leq 6$. To keep notation light, we omit the subscripts of $\Hom$. The six Lagrangians $L_{1},\dots ,L_{6}$ have in total seven endpoints (labeled as $1, 2, \dots, 7$) lying on the outer boundary; choosing the stop $S$ as a basepoint equips these endpoints with a linear order (via the boundary orientation), and they become the seven vertices of the strand diagrams in Figures~\ref{fig:left_diagram} and~\ref{fig:right_diagram}. Consider the two morphisms  
		\begin{align*}
			\alpha_2 \alpha_3 \alpha_4 \otimes \Id_{L_3} \otimes \Id_{L_4} & \in \Hom(L_{2}\times L_{3}\times L_{4},L_{3}\times L_{4}\times L_{5})\\ \alpha_5\otimes \Id_{L_4}\otimes \alpha_3 \alpha_4 & \in \Hom(L_{3}\times L_{4}\times L_{5},L_{1}\times L_{4}\times L_{5}).
		\end{align*} 
		The strand diagrams of these two morphisms are displayed in Figure~\ref{fig:left_diagram} and  \ref{fig:right_diagram}, respectively. In Figure~\ref{fig:left_diagram} the orange line corresponds to the boundary segment  $\alpha_2\alpha_3\alpha_4$ (with endpoints $3$ and $6$), the red line to $\Id_{L_3}$, and the green line to $\Id_{L_4}$.
		
		\begin{figure}[htbp]
			\centering 
			\begin{minipage}[b]{0.48\textwidth}
				\centering 
				\begin{tikzpicture}[scale=0.8]
					\tikzset{
						with arrows/.style={
							decoration={ 
								markings,
								mark=at position 0.6 with {\arrow{>}}
							}, 
							postaction={decorate}
						},
						vertex/.style = {
							circle, draw, fill=black,  
							minimum size = 4pt, 
							inner sep=1pt
						}
					}
					
					\foreach \i in {1,2,3,4,5,6,7} {
						\node[vertex] (left\i) at (0,\i) {};
						\node[left] at (left\i.west) {$\i$};
					}
					
					\foreach \i in {1,2,3,4,5,6,7} {
						\node[vertex] (right\i) at (4,\i) {};
						\node[right] at (right\i.east) {$\i$};
					}
					
					\draw[thick, orange] (left3) to[out=0, in=180] (right6);
					
					\draw[thick, red] (left4) to[out=0, in=180] (right4);
					
					\draw[thick, green] (left5) to[out=0, in=180] (right5);
				\end{tikzpicture}
				\caption{Strand diagram of $\alpha_2 \alpha_3 \alpha_4 \otimes \Id_{L_3} \otimes \Id_{L_4}$.}
				\label{fig:left_diagram}
			\end{minipage}
			\hfill
			\begin{minipage}[b]{0.48\textwidth}
				\centering
				\begin{tikzpicture}[scale=0.8]
					\tikzset{
						with arrows/.style={
							decoration={ 
								markings,
								mark=at position 0.6 with {\arrow{>}}
							}, 
							postaction={decorate}
						},
						vertex/.style = {
							circle, draw, fill=black,  
							minimum size = 4pt, 
							inner sep=1pt
						}
					}
					
					\foreach \i in {1,2,3,4,5,6,7} {
						\node[vertex] (left\i) at (0,\i) {};
						\node[left] at (left\i.west) {$\i$};
					}
					
					\foreach \i in {1,2,3,4,5,6,7} {
						\node[vertex] (right\i) at (4,\i) {};
						\node[right] at (right\i.east) {$\i$};
					}
					
					\draw[thick, red] (left4) to[out=0, in=180] (right6);
					
					\draw[thick, orange] (left6) to[out=0, in=180] (right7);
					
					\draw[thick, green] (left5) to[out=0, in=180] (right5);
				\end{tikzpicture}
				\caption{Strand diagram of $\alpha_5\otimes \Id_{L_4}\otimes \alpha_3 \alpha_4$.}
				\label{fig:right_diagram}
			\end{minipage}
		\end{figure}
		Now we can see the composition of $\alpha_2 \alpha_3 \alpha_4 \otimes \Id_{L_3} \otimes \Id_{L_4}$ and $\alpha_5\otimes \Id_{L_4}\otimes \alpha_3 \alpha_4$. At the level of strand diagrams, this algebraic composition corresponds to the concatenation and subsequent reduction of the diagrams shown in Figure~\ref{fig:composition_visualization}.  As a result, the composition is  $\alpha_2 \alpha_3 \alpha_4\alpha_5 \otimes \Id_{L_4} \otimes \alpha_3\alpha_4$.
		\begin{figure}[htbp]
			\centering
			\hspace*{-0.5em}
			\begin{minipage}[c]{0.30\textwidth}
				\centering
				\begin{tikzpicture}[scale=1]
					\tikzset{
						with arrows/.style={decoration={markings, mark=at position 0.6 with {\arrow{>}}}, postaction={decorate}},
						vertex/.style={circle, draw, fill=black, minimum size=3pt, inner sep=1pt}
					}
					\foreach \y in {1,...,7} {
						\node[vertex] (gL\y) at (0, \y) {};
						\node[vertex] (gR\y) at (2.5, \y) {};
					}
					\draw[thick, orange] (gL3) to[out=0, in=180] (gR6);
					\draw[thick, red] (gL4) to[out=0, in=180] (gR4);
					\draw[thick, green] (gL5) to[out=0, in=180] (gR5);
				\end{tikzpicture}
				\caption*{\makebox[\linewidth][c]{$\alpha_2 \alpha_3 \alpha_4 \otimes \Id_{L_3} \otimes \Id_{L_4}$}}
				
				\label{fig:map_g2}
			\end{minipage}
			\hspace{-0.2em}
			\raisebox{0.5cm}{$\circ$}
			\hspace{-0.2em}
			\begin{minipage}[c]{0.30\textwidth}
				\centering
				\begin{tikzpicture}[scale=1]
					\tikzset{
						with arrows/.style={decoration={markings, mark=at position 0.6 with {\arrow{>}}}, postaction={decorate}},
						vertex/.style={circle, draw, fill=black, minimum size=3pt, inner sep=1pt}
					}
					\foreach \y in {1,...,7} {
						\node[vertex] (fL\y) at (0, \y) {};
						\node[vertex] (fR\y) at (2.5, \y) {};
					}
					\draw[thick, red] (fL4) to[out=0, in=180] (fR6);
					\draw[thick, orange] (fL6) to[out=0, in=180] (fR7);
					\draw[thick, green] (fL5) to[out=0, in=180] (fR5);
				\end{tikzpicture}
				\caption*{\makebox[\linewidth][c]{$\alpha_5\otimes \Id_{L_4}\otimes \alpha_3 \alpha_4$}}
				
				\label{fig:map_f2}
			\end{minipage}
			\hspace{-0.2em}
			\raisebox{0.5cm}{$=$}
			\hspace{-0.2em}
			\begin{minipage}[c]{0.30\textwidth}
				\centering
				\begin{tikzpicture}[scale=1, baseline=(current bounding box.center)]
					\tikzset{
						with arrows/.style={decoration={markings, mark=at position 0.6 with {\arrow{>}}}, postaction={decorate}},
						vertex/.style={circle, draw, fill=black, minimum size=3pt, inner sep=0.5pt}
					}
					\foreach \y in {1,...,7} {
						\node[vertex] (cL\y) at (0, \y) {};
						\node[vertex] (cR\y) at (2.5, \y) {};
					}
					
					\draw[thick, orange] (cL3) to[out=5, in=180] (cR7);
					
					\draw[thick, red] (cL4) to[out=0, in=180] (1.25,4);
					\draw[thick, red] (1.25,4) to[out=0, in=175] (cR6);
					
					\draw[thick, green] (cL5) to (1.25,5);
					\draw[thick, green] (1.25,5) to (cR5);
				\end{tikzpicture}
				\caption*{\makebox[\linewidth][c]{$\alpha_2 \alpha_3 \alpha_4\alpha_5 \otimes \Id_{L_4} \otimes \alpha_3\alpha_4$}}
				\label{fig:composition_fg}
			\end{minipage}
			\hspace*{-0.5em}
			\caption{Composition of $\alpha_2 \alpha_3 \alpha_4 \otimes \Id_{L_3} \otimes \Id_{L_4}$ and $\alpha_5\otimes \Id_{L_4}\otimes \alpha_3 \alpha_4$.}
			\label{fig:composition_visualization}
		\end{figure}
		
		From Figure \ref{fig: differential}, we obtain that  the differential of $\alpha_2 \alpha_3 \alpha_4\alpha_5 \otimes \Id_{L_4} \otimes \alpha_3\alpha_4$ equals $$\alpha_3\alpha_4\alpha_5 \otimes \Id_{L_4}\otimes \alpha_2\alpha_3\alpha_4+\alpha_2\alpha_3\alpha_4\alpha_5\otimes \alpha_4 \otimes \alpha_3.$$ 
		Note that the last summand in Figure \ref{fig: differential} vanishes since there are two strands intersecting twice.
		\begin{figure}[htbp]
			\centering
			\hspace*{-0.5em}
			\begin{minipage}[c]{0.2\textwidth}
				\centering
				\hspace{-0.3em}
				\raisebox{-0.1cm}{$\partial$}
				\hspace{-0.3em}
				\begin{tikzpicture}[scale=0.7, baseline=(current bounding box.center)]
					\tikzset{
						with arrows/.style={decoration={markings, mark=at position 0.6 with {\arrow{>}}}, postaction={decorate}},
						vertex/.style={circle, draw, fill=black, minimum size=3pt, inner sep=0.5pt}
					}
					\foreach \y in {1,...,7} {
						\node[vertex] (cL\y) at (0, \y) {};
						\node[vertex] (cR\y) at (2.5, \y) {};
					}
					
					\draw[thick, orange] (cL3) to[out=5, in=180] (cR7);
					
					\draw[thick, red] (cL4) to[out=0, in=180] (1.25,4);
					\draw[thick, red] (1.25,4) to[out=0, in=175] (cR6);
					
					\draw[thick, green] (cL5) to (1.25,5);
					\draw[thick, green] (1.25,5) to (cR5);
				\end{tikzpicture}
			\end{minipage}
			\hspace{-0.3em}
			\raisebox{0cm}{$=$}
			\hspace{-0.3em}
			\begin{minipage}[c]{0.2\textwidth}
				\centering
				\begin{tikzpicture}[scale=0.7]
					\tikzset{
						with arrows/.style={decoration={markings, mark=at position 0.6 with {\arrow{>}}}, postaction={decorate}},
						vertex/.style={circle, draw, fill=black, minimum size=3pt, inner sep=1pt}
					}
					\foreach \y in {1,...,7} {
						\node[vertex] (fL\y) at (0, \y) {};
						\node[vertex] (fR\y) at (2.5, \y) {};
					}
					\draw[thick, orange] (fL4) to[out=30, in=210] (fR7);
					\draw[thick, red] (fL3) to[out=30, in=210] (fR6);
					\draw[thick, green] (fL5) to[out=0, in=180] (fR5);
				\end{tikzpicture}
			\end{minipage}
			\hspace{-0.3em}
			\raisebox{0cm}{$+$}
			\hspace{-0.3em}
			\begin{minipage}[c]{0.2\textwidth}
				\centering
				\begin{tikzpicture}[scale=0.7]
					\tikzset{
						with arrows/.style={decoration={markings, mark=at position 0.6 with {\arrow{>}}}, postaction={decorate}},
						vertex/.style={circle, draw, fill=black, minimum size=3pt, inner sep=1pt}
					}
					\foreach \y in {1,...,7} {
						\node[vertex] (fL\y) at (0, \y) {};
						\node[vertex] (fR\y) at (2.5, \y) {};
					}
					\draw[thick, orange] (fL3) to[out=0, in=180] (fR7);
					\draw[thick, red] (fL5) to[out=0, in=180] (fR6);
					\draw[thick, green] (fL4) to[out=0, in=180] (fR5);
				\end{tikzpicture}
			\end{minipage}
			\hspace{-0.3em}
			\raisebox{0cm}{$+$}
			\hspace{-0.3em}
			\begin{minipage}[c]{0.2\textwidth}
				\centering
				\hspace{-0.3em}
				\raisebox{-0.1cm}{$$}
				\hspace{-0.3em}
				\begin{tikzpicture}[scale=0.7, baseline=(current bounding box.center)]
					\tikzset{
						with arrows/.style={decoration={markings, mark=at position 0.6 with {\arrow{>}}}, postaction={decorate}},
						vertex/.style={circle, draw, fill=black, minimum size=3pt, inner sep=0.5pt}
					}
					\foreach \y in {1,...,7} {
						\node[vertex] (cL\y) at (0, \y) {};
						\node[vertex] (cR\y) at (2.5, \y) {};
					}
					
					\draw[thick, orange] (cL5) to[out=5, in=180] (cR7);
					
					\draw[thick, green] (cL3) to[out=0, in=180] (1.25,5);
					\draw[thick, green] (1.25,5) to[out=0, in=175] (cR5);
					
					\draw[thick, red] (cL4) to[out=0, in=180] (1.25,4);
					\draw[thick, red] (1.25,4) to[out=0, in=180] (cR6);
				\end{tikzpicture}
			\end{minipage}
			\hspace*{-0.5em}
			\caption{Differential of $\alpha_2 \alpha_3 \alpha_4 \alpha_5 \otimes \Id_{L_4} \otimes \alpha_3\alpha_4$, where the last term vanishes since  two strands intersect twice.}
			\label{fig: differential}
		\end{figure}
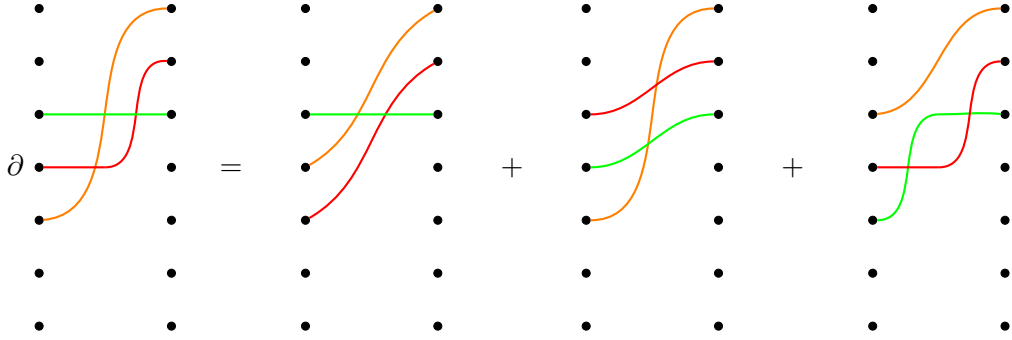
		
	\end{example}
	
	\subsection{Hochschild cohomology of quivers with relations} \label{subsection:reductionsystem} 
	In this subsection, let us recall a small projective bimodule resolution constructed from a reduction system, which can be used to compute Hochschild cohomology.  We refer to \cite{Chouhy-Solotar2015, Barmeier-Wang2020} for further details. 
	
	A \textit{quiver} $Q=(Q_0, Q_1, s, t)$ is a quadruple consisting of two finite sets $Q_0$ (the \textit{vertices}) and $Q_1$ (the \textit{arrows}), together with two maps $s: Q_1 \to Q_0$ (the \textit{source}) and $t: Q_1 \to Q_0$ (the \textit{target}). 
	
	A \textit{path} of length $n \ge 1$ is a sequence of arrows $p = \alpha_1 \dots \alpha_n$ written from left to right, meaning that $t(\alpha_i) = s(\alpha_{i+1})$ for $1 \le i \le n-1$. Each vertex $i \in Q_0$ is regarded as a \textit{trivial path} $e_i$ of length $0$, with $s(e_i) = t(e_i) = i$. For any path $p$, its source $s(p)$ and target $t(p)$ are defined in the natural way.
	
	For each integer $n \ge 0$, let $\Bbbk Q_n$ denote the vector space spanned by all paths of length $n$. The \textit{path algebra} $\Bbbk Q = \bigoplus\limits_{n \ge 0} \Bbbk Q_n$ is the direct sum of these spaces, with multiplication given by concatenation: for any two paths $p$ and $q$, their product $pq$ is the concatenated path if $t(p) = s(q)$, and is $0$ otherwise. This multiplication makes $\Bbbk Q$ an associative algebra where the trivial paths act as orthogonal idempotents, satisfying $e_i e_j = \delta_{i,j} e_i$ and $e_{s(p)} p = p = p e_{t(p)}$.

	Two paths are called \textit{parallel} if they have the same source and the same target. This notion can be extended to an element $\rho=\sum\limits_{k} \lambda_k p_k \in \Bbbk Q_n$ by saying that a path $q$ is parallel to $\rho$ if $q$ is parallel to every path $p_k$ for which the coefficient $\lambda_k$ is nonzero. By convention, any path is considered to be parallel to 0.

	\begin{definition}
		[{\cite[§1]{Bergman1978}, \cite[Definition 3.1]{Barmeier-Wang2020}}]\label{def reduction system}
		A \textit{reduction system} $R$ for the path algebra $\Bbbk Q$ is a set of pairs
		\[
		R = \{ (s, \varphi(s)) \mid s \in S \text{ and } \varphi(s) \in \Bbbk Q \},
		\]
		where
		\begin{itemize}[label=--]
			\item $S$ is a subset of $Q_{\geq 2}$ whose elements are mutually non-subpaths; that is, for any distinct $s, s' \in S$, $s$ is not a subpath of $s'$.
			\item For each $s \in S$, the path $s$ and the element $\varphi(s)$ are parallel and moreover $\varphi(s)$ is \textit{irreducible}; i.e., it is a $\Bbbk$-linear combination of irreducible paths.
		\end{itemize}
		Here, a path is called \textit{irreducible} if it contains no element of $S$ as a subpath. The set of all irreducible paths is denoted by $\Irr_S$.
	\end{definition}
	
	\begin{definition}
		Given a two-sided ideal $I$ of $\Bbbk Q$, a reduction system $R$ is said to \textit{satisfy the Diamond condition for $I$} if  the two-sided ideal $I$ is generated by the set $$\{ s - \varphi(s) \mid (s, \varphi(s)) \in R \}$$ and  every path is \textit{reduction-unique} (see \cite[Definition 3.2]{Barmeier-Wang2020}).
	\end{definition}
	
	\begin{definition}
		For a path $p = qr$, the path $q$ is a \textit{proper left subpath} of $p$ if $q \neq p$.
		
		For $n \geq 0$, a path $p \in Q_\bullet$ is an \textit{$n$-ambiguity}\index{ambiguity!higher!$n$-|textbf} if there exists an arrow $u_0 \in Q_1$ and irreducible paths $u_1, \dotsc, u_{n+1}$ such that $p = u_0 \dotsb u_{n+1}$ and, for all $i$
		\begin{enumerate}
			\item $u_i u_{i+1}$ is reducible;
			\item $u_i d$ is irreducible for every proper left subpath $d$ of $u_{i+1}$.
		\end{enumerate}
	\end{definition}
	
	\begin{theorem}[{\cite[Section 4]{Chouhy-Solotar2015}, \cite[Theorem 4.2]{Barmeier-Wang2020}}]
		\label{theorem:projective}
		\index{resolution!Chouhy--Solotar|textbf}
		Let $A = \Bbbk Q / I$, let $R$ be a reduction system satisfying the Diamond condition for $I$, set $S_m = Q_m$ for $m = 0, 1$, and let $S_{n+2}$ be the set of $n$-ambiguities\index{ambiguity!higher!$n$-} for $n \geq 0$. Then there is a projective $A$-bimodule resolution $P_{\bullet}$ of $A$:
		\[
		\cdots \longrightarrow P_n \xrightarrow{\partial_n} P_{n-1} \longrightarrow \cdots \longrightarrow P_1 \xrightarrow{\partial_1} P_0,
		\]
		where $P_n = A \otimes_{\Bbbk Q_0} \Bbbk S_n \otimes_{\Bbbk Q_0} A$ and the augmentation $\partial_0: P_0 \to A$ is given by the multiplication of $A$.
	\end{theorem}
	Given a reduction system $R = \{ (s, \varphi(s)) \mid s \in S,\ \varphi(s) \in \Bbbk Q \}$ satisfying the diamond condition, the Hochschild cohomology of $A = \Bbbk Q/I$ can be computed via the cochain complex
	\begin{equation}\label{eq:CS-complex}
		\begin{aligned}
			0 \to\;&
			\Hom_{(\Bbbk Q_0)^e}(\Bbbk Q_0, A)
			\xrightarrow{\partial^0}
			\Hom_{(\Bbbk Q_0)^e}(\Bbbk Q_1, A)
			\xrightarrow{\partial^1}
			\Hom_{(\Bbbk Q_0)^e}(\Bbbk S_2, A)
			\xrightarrow{\partial^2} \cdots \\
			&\xrightarrow{\partial^{n-1}}
			\Hom_{(\Bbbk Q_0)^e}(\Bbbk S_n, A)
			\xrightarrow{\partial^n}
			\Hom_{(\Bbbk Q_0)^e}(\Bbbk S_{n+1}, A)
			\to \cdots .
		\end{aligned}
	\end{equation}
	Note that $S_2=S$ in the complex above.
	
	We now explain how the differential $\partial^1$ acts on an element $\varphi \in \Hom_{(\Bbbk Q_0)^e}(\Bbbk Q_1, A)$.  
	Define the splitting map $\mathrm{split}_1: \Bbbk Q \to A \otimes_{\Bbbk Q_0} \Bbbk Q_1 \otimes_{\Bbbk Q_0} A$ (see \cite[Lemma 4.6]{Barmeier-Wang2020}) by
	\begin{align*}
		\mathrm{split}_1(p) = & \; 1 \otimes_{\Bbbk Q_0} x_1 \otimes_{\Bbbk Q_0} \pi(x_2 \cdots x_m) \\
		& + \sum\limits_{1 < i < m} \pi(x_1 \cdots x_{i-1}) \otimes_{\Bbbk Q_0} x_i \otimes_{\Bbbk Q_0} \pi(x_{i+1} \cdots x_m) \\
		& + \pi(x_1 \cdots x_{m-1}) \otimes_{\Bbbk Q_0} x_m \otimes_{\Bbbk Q_0} 1
	\end{align*}
	for a path $p = x_1 \cdots x_m$ with $x_i \in Q_1$, where $\pi: \Bbbk Q \to A$ is the canonical projection.  
	For $\varphi \in \Hom_{(\Bbbk Q_0)^e}(\Bbbk Q_1, A)$, we obtain an extension $\widetilde{\varphi} \in \Hom_{A^e}(A \otimes_{\Bbbk Q_0} \Bbbk Q_1 \otimes_{\Bbbk Q_0} A, A)$ defined by
	\[
	\widetilde{\varphi}(a \otimes_{\Bbbk Q_0} x \otimes_{\Bbbk Q_0} b) = a\varphi(x)b, \qquad a,b \in A,\; x \in Q_1.
	\]
	Then the differential $\partial^1$ acts on $\varphi$ by
	\begin{equation}
		\partial^1(\varphi)(s) = \widetilde{\varphi}\bigl(\mathrm{split}_1(s - \varphi(s))\bigr), \qquad s \in S.
		\label{eq:partial1}
	\end{equation}
	
	For later use in the proofs, we introduce the following convenient notation for basis elements of the spaces $\Hom_{(\Bbbk Q_0)^e}(\Bbbk S_n, A)$.  
	A basis element can be represented by a pair $(\alpha \parallel \beta)$ of parallel paths, where $\alpha \in S_n$ and $\beta \in \Bbbk \operatorname{Irr}_S$, defined by
	\[
	(\alpha \parallel \beta)(\gamma) = 
	\begin{cases}
		\beta, & \text{if } \gamma = \alpha, \\[2pt]
		0, & \text{if } \gamma \in S_n \setminus \{\alpha\}.
	\end{cases}
	\]
	
	The following lemma will be used later, which should be well-known to experts.
	\begin{lemma}
		\label{lemma:HH0-dimension}
		Let $Q$ be a connected and acyclic quiver.  
		Let $I$ be an admissible ideal of the path algebra $\Bbbk Q$, and set $A = \Bbbk Q / I$. Then
		\(
		\dim_{\Bbbk} \operatorname{HH}^0(A) = 1.
		\)
	\end{lemma}
	
	\begin{proof}
		The differential $\partial^0$ acts on basis elements by
		\[
		\partial^0(e_i \parallel e_i)(\alpha) = \delta_{i,t(\alpha)}\,\alpha e_i - \delta_{i,s(\alpha)}\,e_i\alpha,\qquad \alpha\in Q_1,
		\]
		where $\delta$ is the Kronecker delta. 
		For a linear combination $\sum\limits_{i=1}^n k_i (e_i\parallel e_i)$ with $k_i\in\Bbbk$, we obtain
		\[
		\partial^0\left(\sum\limits_{i=1}^n k_i (e_i\parallel e_i)\right)(\alpha) = (k_{t(\alpha)}-k_{s(\alpha)})\,\alpha, \quad \text{for any $\alpha \in Q_1.$}
		\]
		Hence the element is a $0$-cocycle if and only if $k_{t(\alpha)}=k_{s(\alpha)}$ for every arrow $\alpha$. Since $Q$ is connected, all $k_i$ are equal. Therefore $\ker\partial^0$ is one‑dimensional, spanned by $\sum\limits_{i=1}^n (e_i\parallel e_i)$.
	\end{proof}

	\section{Symmetric square dg algebras of the annulus}
	We now suppose the surface $\Sigma$ is an annulus with stops, as illustrated in Figure \ref{annulus}. 
	
	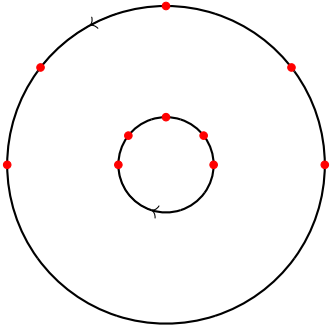
\begin{figure}[htbp]
		\centering
		\begin{tikzpicture}[scale=0.7]
			
			\tikzset{
				with arrows/.style={
					decoration={ 
						markings,
						mark=at position #1 with {\arrow{>}}
					}, 
					postaction={decorate}
				}, 
				with arrows/.default=0.5,
				vertex/.style = {
					circle, draw, fill,  
					minimum size = 2pt, 
					inner sep=1pt
				}
			}
			
			\def \radius {1.5cm}
			\def \numSectors {10}
			
			\foreach \s in {1,...,\numSectors} {
				\draw[thick] ([shift=({360/\numSectors*(\s)}:\radius-0.6cm)]0,0) 
				arc ({360/\numSectors*(\s)}:{360/\numSectors*(\s+1)}:\radius-0.6cm);
				
				
				\draw[thick] ([shift=({360/\numSectors*(\s)}:\radius+1.5cm)]0,0) 
				arc ({360/\numSectors*(\s)}:{360/\numSectors*(\s+1)}:\radius+1.5cm);
			}
			
			\draw[with arrows=0] ([shift=({360/\numSectors*(7)}:\radius-0.6cm)]0,0) 
			arc ({360/\numSectors*(7)}:{360/\numSectors*(5)}:\radius-0.6cm);
			\draw[with arrows=0.3]({360/\numSectors*(3)}:\radius+1.5cm) 
			arc ({360/\numSectors*(3)}:{360/\numSectors*(4)}:\radius+1.5cm);
			
			\node[vertex,red] at ([shift=({360/\numSectors*(2.5)}:\radius+1.5cm)]0,0) {};
			\node[vertex,red] at ([shift=({360/\numSectors*(1.05)}:\radius+1.5cm)]0,0) {};
			\node[vertex,red] at ([shift=({360/\numSectors*(0)}:\radius+1.5cm)]0,0) {};
			\node[vertex,red] at ([shift=({360/\numSectors*(5)}:\radius+1.5cm)]0,0) {};
			\node[vertex,red] at ([shift=({360/\numSectors*(3.95)}:\radius+1.5cm)]0,0) {};
			\node[vertex,red] at ([shift=({360/\numSectors*(2.5)}:\radius-0.6cm)]0,0) {};
			\node[vertex,red] at ([shift=({360/\numSectors*(1.05)}:\radius-0.6cm)]0,0) {};
			\node[vertex,red] at ([shift=({360/\numSectors*(0)}:\radius-0.6cm)]0,0) {};
			\node[vertex,red] at ([shift=({360/\numSectors*(5)}:\radius-0.6cm)]0,0) {};
			\node[vertex,red] at ([shift=({360/\numSectors*(3.95)}:\radius-0.6cm)]0,0) {};

		\end{tikzpicture}
		\caption{An annulus with stops.}
		\label{annulus}
	\end{figure}
	Assume that the number of stops on the outer boundary is $n_1$ and on the inner boundary is $n_2$. Without loss of generality, we assume $n_1 \leq n_2$ throughout this paper. Note that $(\Sigma, \Lambda)$ is of type $\widetilde{\mathbb{A}}_{n_1,n_2}$. Given a full formal arc system on $(\Sigma, \Lambda)$, the associated endomorphism algebra is a gentle one-cycle algebra. By \cite[Proposition 4.3]{Jin-Schroll-Wang2025}, any such algebra is derived equivalent to the gentle algebra in the standard form shown in Figure \ref{standard form}, subject to the single relation $\beta \alpha_{n_1+1}=0$. We denote this standard gentle algebra by $G_{n_1, n_2}$.
	\begin{figure}[htbp]
		\begin{center}
			\begin{tikzcd}[column sep=normal, row sep=normal]
				& L_2 \arrow[r,"\alpha_2"]
				& L_3 \arrow[r, phantom, "\cdots" midway] 
				& L_{n_{1}-1} \arrow[r,"\alpha_{n_1-1}"] 
				& L_{n_{1}} \arrow[dl,"\alpha_{n_1}",swap] \\
				&& L_1\arrow[r,"\beta"]  \arrow[ul,"\alpha_1"] 
				& L_{n_{1}+1}\arrow[r,"\alpha_{n_1+1}"] 
				& \cdots \arrow[r,"\alpha_{n_1+n_2-2}"] 
				& L_{n_1+n_2-1} \arrow[r,"\alpha_{n_1+n_2-1}"] 
				& L_{n_1+n_2}
			\end{tikzcd}
			\caption{Quivers of gentle algebras $G_{n_1,n_2}$.}
			\label{standard form}
		\end{center}
	\end{figure}
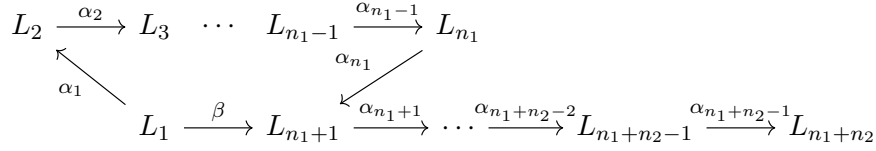

	Let us denote by $\widetilde{\mathcal{A}}_{n_1,n_2}$ the dg endomorphism algebra of the associated generating set $\{L_i \times L_j\mid 1\leq i <j \leq n_1+n_2\}$ in $\mathcal{W}(\operatorname{Sym}^2(\Sigma), \Lambda^{(2)})$, which we call the {\it symmetric square dg algebra} of $\Sigma$. Note that if $n_1 = n_2=1$, the category $\mathcal{W}(\operatorname{Sym}^2(\Sigma), \Lambda^{(2)})\simeq \mathrm{per}(\Bbbk)$ is trivial. We may therefore assume $n_2\geq 2$. In this case the dg algebra $\widetilde{\mathcal{A}}_{n_1,n_2}$ can be described as follows.
	
	\begin{prop}\label{prop:symsquadg}
		Let $(\Sigma,\Lambda)$ be the annulus with stops of type $\widetilde{\mathbb{A}}_{n_1,n_2}$ with $n_2\geq 2$. Then the symmetric square dg algebra $\widetilde{\mathcal{A}}_{n_1,n_2}$ is isomorphic to the dg  path algebra of the quiver $\widetilde Q_{n_1,n_2}$ (depicted in Figure~\ref{dg-quiver A_{n1,n2}} below) modulo a dg ideal $\widetilde I_{n_1,n_2}$, described as follows:
		\begin{enumerate}[label=(\arabic*), leftmargin=*] 
			\item The arrows in $\widetilde Q_{n_1,n_2}$ consist of four types:
			\begin{itemize}[label=--]
				\item \textsc{Horizontal arrows}: 
				\(
				x_{i,j}= \operatorname{id}_{L_i} \otimes \alpha_j 
				\colon L_{i,j} \longrightarrow L_{i,j+1}
				\) for \(1\leq i<j\leq n_1+n_2-1\);
				\item \textsc{Vertical arrows}: 
				\(
				y_{i,j}= \alpha_i \otimes \operatorname{id}_{L_j} 
				\colon  L_{i,j} \longrightarrow L_{i+1,j}
				\) for \(1\leq i\leq n_1+n_2-2\) and \(i+2\leq j\leq n_1+n_2\);
				\item \textsc{Long blue arrows}: 
				\(
				u_{1,i}= \beta \otimes \operatorname{id}_{L_i} 
				\colon L_{1,i} \longrightarrow L_{i,n_1+1}
				\) for \(2\leq i\leq n_1+n_2\) and \(i\neq n_1+1\), where \(L_{i,n_1+1}\) is identified with \(L_{n_1+1,i}\) whenever \(i>n_1+1\). Note that all the long blue arrows start from the first row; see Figure~\ref{dg-quiver A_{n1,n2}};
				\item \textsc{Diagonal arrows}: 
				\(
				z_{i,i+1}= \alpha_i\alpha_{i+1} \otimes \operatorname{id}_{L_{i+1}} 
				\colon L_{i,i+1} \longrightarrow L_{i+1,i+2}
				\) for \(1\leq i\leq n_1+n_2-2\).
			\end{itemize}
			
			\item The ideal of relations $\widetilde I_{n_1,n_2}$ is generated by the following relations. A displayed relation is imposed only when all arrows appearing in it are defined by the index ranges specified above; otherwise, it is omitted.
			\begin{itemize}[label=--]
				\item \textsc{Commutative square relations of type 1}: $x_{i,j} y_{i,j+1} = y_{i,j} x_{i+1,j}$ for $1\leq i<j-1 \leq n_1+n_2-2$; see Figure \ref{fig:commutative-square-relations};
				\item \textsc{Commutative square relations of type 2}: $u_{1,i} x_{n_1+1,i}=x_{1,i} u_{1,i+1}$ for $n_1+2\leq i\leq  n_1+n_2-1$; see Figure \ref{fig:commutative-square-relations};
				\item \textsc{Commutative square relations of type 3}: $u_{1,i} y_{i,n_1+1}=x_{1,i}u_{1,i+1}$ for $1<i<n_1$; see Figure \ref{fig:commutative-square-relations};	
				\item \textsc{Commutative square relations of type 4}: $u_{1,n_1}z_{n_1,n_1+1}
				= x_{1,n_1}x_{1,n_1+1}u_{1,n_1+2}$ for \(n_2\geq 2\);
				\item \textsc{Zig-zag relations}: $z_{i,i+1} x_{i+1,i+2}  y_{i+1,i+3}= x_{i,i+1}  y_{i,i+2}  z_{i+1,i+2}$ for $1\leq i \leq n_1+n_2-3$;
				\item \textsc{Monomial relations}: 
				$u_{1,i} x_{i,n_1+1}=0$ for $1< i\leq n_1$; 
				$u_{1,n_1+i} y_{n_1+1,n_1+i}=0$ for $3\leq i \leq n_2$; 
				$u_{1,n_1+2} z_{n_1+1,n_1+2}=0$ whenever this path is defined; 
				and $z_{i,i+1} z_{i+1,i+2} = 0$ for $1\leq i \leq n_1+n_2-3$.
			\end{itemize}
			\item The differential is uniquely determined by the Leibniz rule and its action on the diagonal arrows:
			\(
			d(z_{i,i+1}) = x_{i,i+1} y_{i,i+2} \text{~for~} 1\leq i\leq n_1+n_2-2.
			\)
			In particular, the paths \(x_{i,i+1}y_{i,i+2}\) become boundaries, while the arrows \(z_{i,i+1}\) themselves do not define cohomology classes.
		\end{enumerate}
	\end{prop}

	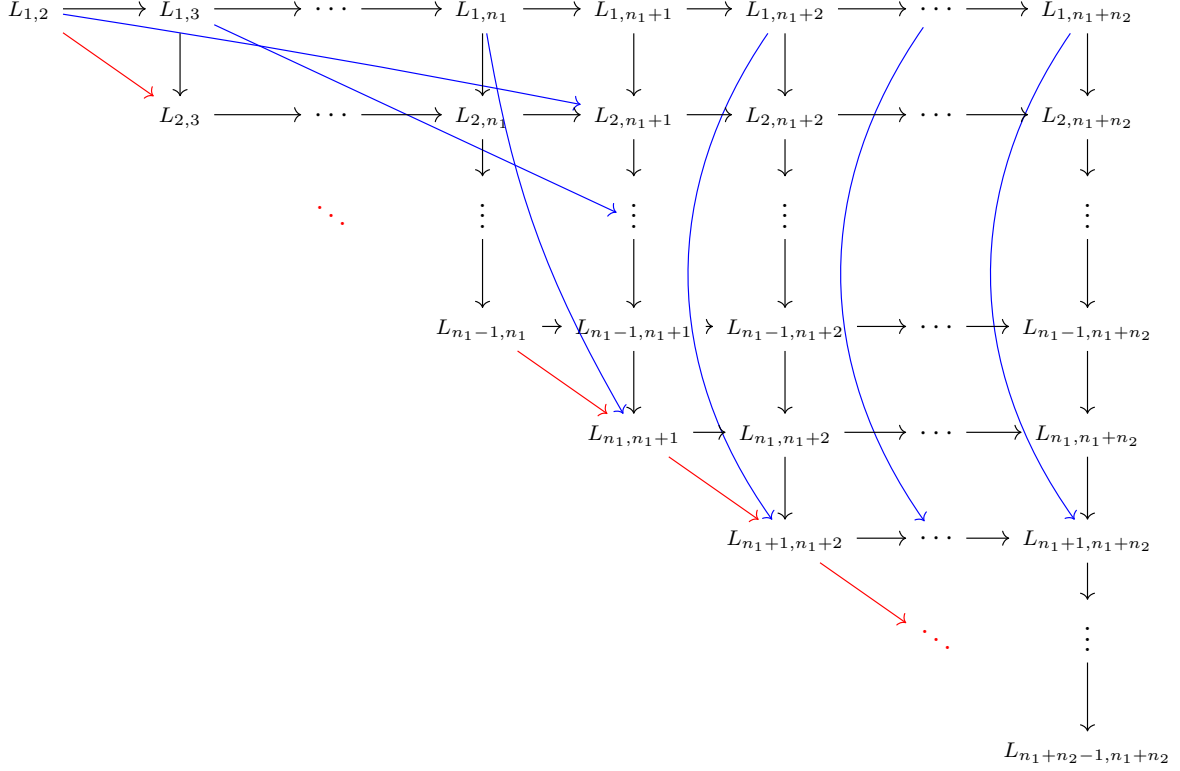
\begin{figure}[htbp]
		\centering
		\begin{tikzcd}[
			column sep={2cm,between origins},
			row sep={1.4cm,between origins},
			every matrix/.append style={name=mycd}
			]
			\text{\scriptsize$L_{1,2}$} \ar[red,dr] \ar[r] \ar[blue,drrrr,bend left=1]& 
			\text{\scriptsize$L_{1,3}$} \ar[r] \ar[d]\ar[blue, ddrrr]  & 
			\cdots \ar[r]& 
			\text{\scriptsize$L_{1,n_1}$} \ar[d] \ar[r] \ar[blue,ddddr,bend right=10] & 
			\text{\scriptsize$L_{1,n_1+1}$} \ar[d]\ar[r] &
			\text{\scriptsize$L_{1,n_1+2}$} \ar[d]\ar[r]\ar[blue,ddddd,bend right=35]&
			\cdots\ar[r]\ar[blue,ddddd,bend right=35]&
			\text{\scriptsize$L_{1,n_1+n_2}$} \ar[d]\ar[blue,ddddd,bend right=35]\\
			& 
			\text{\scriptsize$L_{2,3}$} \ar[r] & 
			\cdots \ar[r] & 
			\text{\scriptsize$L_{2,n_1}$} \ar[d] \ar[r] & 
			\text{\scriptsize$L_{2,n_1+1}$} \ar[d] \ar[r] &
			\text{\scriptsize$L_{2,n_1+2}$} \ar[d]\ar[r]&
			\cdots\ar[r]&
			\text{\scriptsize$L_{2,n_1+n_2}$} \ar[d]\\
			&& \textcolor{red}{\ddots}&
			\vdots \ar[d]  & 
			\vdots \ar[d] &
			\vdots \ar[d]& 
			& 
			\vdots \ar[d]\\
			&&&
			\text{\scriptsize$L_{n_1-1,n_1}$} \ar[r] \ar[red,dr] & 
			\text{\scriptsize$L_{n_1-1,n_1+1}$} \ar[d] \ar[r]&
			\text{\scriptsize$L_{n_1-1,n_1+2}$}\ar[r]\ar[d]&
			\cdots\ar[r]&
			\text{\scriptsize$L_{n_1-1,n_1+n_2}$}\ar[d]\\
			&&&&
			\text{\scriptsize$L_{n_1,n_1+1}$}\ar[r] \ar[red,dr]&
			\text{\scriptsize$L_{n_1,n_1+2}$}\ar[r]\ar[d]&
			\cdots\ar[r]&
			\text{\scriptsize$L_{n_1,n_1+n_2}$}\ar[d]\\
			&&&&&
			\text{\scriptsize$L_{n_1+1,n_1+2}$}\ar[r]\ar[red,dr]&
			\cdots\ar[r]&
			\text{\scriptsize$L_{n_1+1,n_1+n_2}$}\ar[d]\\
			&&&&&&
			\textcolor{red}{\ddots}&
			\vdots \ar[d] \\
			&&&&&&&
			\text{\scriptsize$L_{n_1+n_2-1,n_1+n_2}$}
		\end{tikzcd}
		\captionsetup{justification=centering}
		\caption{Dg quiver $\widetilde Q_{n_1,n_2}$ of $\widetilde{\mathcal{A}}_{n_1,n_2}$, where the horizontal and vertical arrows are $x$ and $y$, the blue ones are $u$ and red ones are $z$.}
		\label{dg-quiver A_{n1,n2}}
	\end{figure}
	
	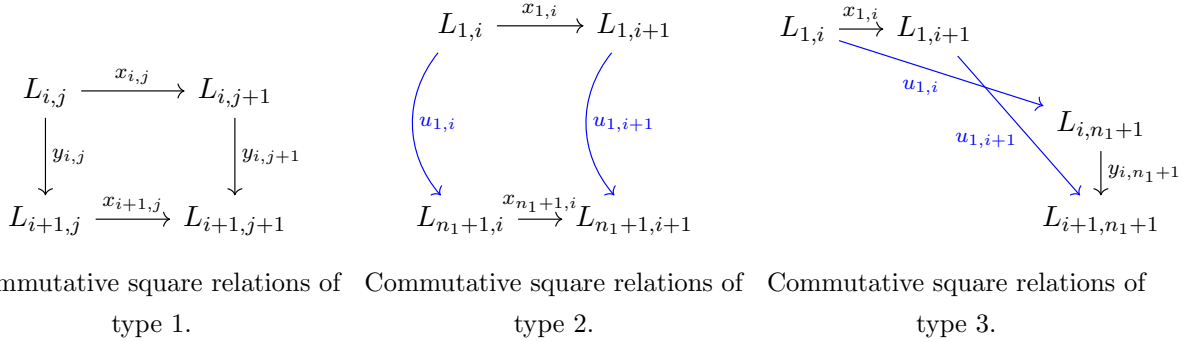
\begin{figure}[htbp]
		\centering
		\begin{subfigure}[b]{0.32\textwidth}
			\centering
			\vspace*{1.8em} 
			\begin{tikzcd}[ampersand replacement=\&, row sep=large, column sep=1cm]
				L_{i,j} \ar[r,"x_{i,j}"] \ar[d,"y_{i,j}"] \& L_{i,j+1} \ar[d,"y_{i,j+1}"] \\
				L_{i+1,j} \ar[r,"x_{i+1,j}"] \& L_{i+1,j+1}
			\end{tikzcd}
			\caption*{\centering Commutative square relations of type $1$.}
			\label{subfig:2x2}
		\end{subfigure}\hfill
		\begin{subfigure}[b]{0.32\textwidth}
			\centering
			\begin{tikzcd}[ampersand replacement=\&, row sep=0.95cm, column sep=0.6cm]
				L_{1,i} \ar[dd, blue, bend right=40, "u_{1,i}"] \ar[r, "x_{1,i}"] \& L_{1,i+1} \ar[dd, blue, bend right=40, "u_{1,i+1}"] \\
				\&  \\
				L_{n_1+1,i} \ar[r, "x_{n_1+1,i}"] \& L_{n_1+1,i+1}
			\end{tikzcd}
			\caption*{\centering Commutative square relations of type $2$.}
			\label{subfig:3x2}
		\end{subfigure}\hfill
		\begin{subfigure}[b]{0.32\textwidth}
			\centering
			\begin{tikzcd}[ampersand replacement=\&, row sep=0.6cm, column sep=0.6cm]
				L_{1,i} \ar[rrd, blue,"u_{1,i}" {xshift=2pt, yshift=0pt}, swap] \ar[r, "x_{1,i}"] \& L_{1,i+1} \ar[ddr, blue,"u_{1,i+1}" {xshift=2pt, yshift=0pt}, swap] \&  \\
				\&  \& L_{i,n_1+1} \ar[d,"y_{i,n_1+1}"] \\
				\&  \& L_{i+1,n_1+1}
			\end{tikzcd}
			\caption*{\centering Commutative square relations of type $3$.}
			\label{subfig:3x3}
		\end{subfigure}
		\caption{Diagrams of commutative square relations.}
		\label{fig:commutative-square-relations}
	\end{figure}
	
	\begin{remark}
		The $\mathbb{Z}$-grading on $\widetilde{\mathcal A}_{n_1,n_2}$ will be discussed later in Subsection \ref{subsec:grading}. Note that the relations are compatible with the differential. For instance, we have
		\[
		d(z_{i,i+1}z_{i+1,i+2}) = x_{i,i+1}y_{i,i+2}z_{i+1,i+2}-z_{i,i+1}x_{i+1,i+2}y_{i+1,i+3}=0.
		\]
		That is, the ideal $\widetilde I_{n_1,n_2}$ is a dg ideal so that $\Bbbk \widetilde Q_{n_1, n_2}/\widetilde I_{n_1,n_2}$ is a dg algebra. 
	\end{remark}
	
	\begin{proof}[Proof of Proposition \ref{prop:symsquadg}]
		Note that there is an algebra homomorphism $\varphi\colon \Bbbk \widetilde Q_{n_1, n_2} \to \widetilde{\mathcal{A}}_{n_1,n_2}$ uniquely determined by 
		\[
		\varphi(x_{i,j}) = \operatorname{id}_{L_i} \otimes \alpha_j,\quad
		\varphi(y_{i,j}) = \alpha_i \otimes \operatorname{id}_{L_j},\quad
		\varphi(u_{1,i}) = \beta \otimes \operatorname{id}_{L_i},\quad
		\varphi(z_{i,i+1}) = \alpha_i\alpha_{i+1} \otimes \operatorname{id}_{L_{i+1}}.
		\]
		Clearly, $\varphi$ is surjective since $\widetilde{\mathcal{A}}_{n_1,n_2}$ is generated by these elements as an algebra. 
		
		Let us verify that $\varphi$ is compatible with the differential. Indeed, by the differential in the strand diagram we have
		\(
		d(\alpha_i\alpha_{i+1} \otimes \operatorname{id}_{L_{i+1}}) = \alpha_i \otimes \alpha_{i+1}.
		\)
		On the other hand,
		\[
		\alpha_i \otimes \alpha_{i+1} = (\operatorname{id}_{L_i} \otimes \alpha_{i+1})(\alpha_i \otimes \operatorname{id}_{L_{i+2}}).
		\]
		This shows that $\varphi d(z_{i,i+1}) = d\varphi(z_{i,i+1})$.
		
		It remains to show that $\varphi$ factors through the ideal $I$ so that it induces an algebra morphism $\widetilde \varphi \colon \Bbbk \widetilde Q_{n_1, n_2}/\widetilde I_{n_1,n_2} \to \widetilde{\mathcal{A}}_{n_1,n_2}$. This can be checked case by case.
		
		\noindent{Commutative squares of type 1:}
		We have $\varphi(x_{i,j} y_{i,j+1} - y_{i,j} x_{i+1,j}) = 0$ since
		\begin{align*}
			\varphi(x_{i,j} y_{i,j+1} - y_{i,j} x_{i+1,j})
			&= \varphi(x_{i,j})\varphi(y_{i,j+1}) - \varphi(y_{i,j})\varphi(x_{i+1,j})\\
			&= (\operatorname{id}_{L_i} \otimes \alpha_j)(\alpha_i \otimes \operatorname{id}_{L_{j+1}}) - (\alpha_i \otimes \operatorname{id}_{L_j})(\operatorname{id}_{L_{i+1}} \otimes \alpha_j)\\
			&= \alpha_i \otimes \alpha_j - \alpha_i \otimes \alpha_j = 0.
		\end{align*}
		
		\noindent{Commutative squares of type 2:}
		We have $\varphi(u_{1,i} x_{n_1+1,i} - x_{1,i} u_{1,i+1}) = 0$ since
		\begin{align*}
			\varphi(u_{1,i} x_{n_1+1,i} - x_{1,i} u_{1,i+1})
			&= (\beta \otimes \operatorname{id}_{L_i})(\operatorname{id}_{L_{n_1+1}} \otimes \alpha_i) - (\operatorname{id}_{L_1} \otimes \alpha_i)(\beta \otimes \operatorname{id}_{L_{i+1}})\\
			&= \beta \otimes \alpha_i - \beta \otimes \alpha_i = 0.
		\end{align*}
		
		\noindent{Commutative squares of type 3:}
		We have $\varphi(u_{1,i} y_{i,n_1+1} - x_{1,i} u_{1,i+1}) = 0$ since
		\begin{align*}
			\varphi(u_{1,i} y_{i,n_1+1} - x_{1,i} u_{1,i+1})
			&= (\beta \otimes \operatorname{id}_{L_i})(\alpha_i \otimes \operatorname{id}_{L_{n_1+1}}) - (\operatorname{id}_{L_1} \otimes \alpha_i)(\beta \otimes \operatorname{id}_{L_{i+1}})\\
			&= \beta \otimes \alpha_i - \beta \otimes \alpha_i = 0.
		\end{align*}
		
		\noindent{Commutative squares of type 4:} We have $\varphi(u_{1,n_1}z_{n_1,n_1+1}-x_{1,n_1}x_{1,n_1+1}u_{1,n_1+2}) = 0$ since
		\begin{align*}
			\varphi(u_{1,n_1}z_{n_1,n_1+1}-x_{1,n_1}x_{1,n_1+1}u_{1,n_1+2})&=(\beta \otimes \operatorname{id}_{L_{n_1}})(\alpha_{n_1}\alpha_{n_1+1} \otimes \operatorname{id}_{L_{n_1+1}})\\
			&\quad-(\operatorname{id}_{L_1} \otimes \alpha_{n_1})(\operatorname{id}_{L_1}\otimes \alpha_{n_1+1})(\beta \otimes \operatorname{id}_{L_{n_1+2}})\\
			&=\beta \otimes \alpha_{n_1}\alpha_{n_1+1}-\beta \otimes \alpha_{n_1}\alpha_{n_1+1} = 0.
		\end{align*}
		
		\noindent{Zig-zag relations:}
		We have $\varphi(z_{i,i+1} x_{i+1,i+2} y_{i+1,i+3} - x_{i,i+1} y_{i,i+2} z_{i+1,i+2}) = 0$ since
		\begin{align*}
			\varphi(z_{i,i+1} x_{i+1,i+2} y_{i+1,i+3} - x_{i,i+1} y_{i,i+2} z_{i+1,i+2})
			&= \alpha_i\alpha_{i+1}\alpha_{i+2} \otimes \alpha_{i+1} - \alpha_i\alpha_{i+1}\alpha_{i+2} \otimes \alpha_{i+1} = 0.
		\end{align*}
		
		\noindent{Monomial relations:}
		We have $\varphi(z_{i,i+1} z_{i+1,i+2}) = 0$ because
		\begin{align*}
			\varphi(z_{i,i+1} z_{i+1,i+2})
			&= (\alpha_i\alpha_{i+1} \otimes \operatorname{id}_{L_{i+1}})(\alpha_{i+1}\alpha_{i+2} \otimes \operatorname{id}_{L_{i+2}}) = 0,
		\end{align*}
		since the corresponding diagram has two strands crossing more than once. The remaining monomial relations involving $u_{1,i}$ can be verified similarly. As a result, there is an algebra morphism $\widetilde \varphi\colon \Bbbk \widetilde Q_{n_1, n_2}/\widetilde I_{n_1,n_2} \to \widetilde{\mathcal{A}}_{n_1,n_2}$, which is surjective.
		
		To prove that $\widetilde \varphi$ is injective, we compare the dimensions of the morphism spaces in $\widetilde{\mathcal{A}}_{n_1,n_2}$ with those in $\Bbbk\widetilde Q_{n_1,n_2}/\widetilde I_{n_1,n_2}$. Since $\widetilde\varphi$ is surjective, it suffices to show that the two algebras have the same finite dimension. This is achieved by computing the dimensions of
		\[
		\Hom_{\mathcal{W}(\operatorname{Sym}^2(\Sigma),\Lambda^{(2)})}(L_i\times L_j,\; L_k\times L_\ell)
		\]
		for all vertices $L_{i,j}$ and $L_{k,\ell}$ ($1\le i<j\le n_1+n_2$, $1\le k<\ell\le n_1+n_2$). By definition of the morphism spaces in the symmetric‑product Fukaya category,
		\[
		\Hom(L_i\times L_j,\; L_k\times L_\ell)=\Hom(L_i,L_k)\otimes\Hom(L_j,L_\ell)\;\oplus\;\Hom(L_i,L_\ell)\otimes\Hom(L_j,L_k).
		\]
		By computing the dimensions of the underlying Reeb chords on $(\Sigma, \Lambda)$, a case-by-case analysis yields the dimensions summarized in Table \ref{tab:dimensions}, where we assume $i < j$ and $k < \ell$.
		
		\begin{table}[htbp]
			\centering
			\renewcommand{\arraystretch}{1.3}
			\begin{tabular}{llc}
				\hline
				\textsc{Case} & \textsc{Constraints} & $\dim \operatorname{Hom}$ \\
				\hline
				\multirow{2}{*}{\textup{(a)} $k, \ell \neq n_1+1$} 
				& $i < j \le k < \ell$ & 2 \\
				& $i \le k < j \leq \ell$ & 1 \\
				\hline
				\multirow{2}{*}{\textup{(b)} $k = n_1+1, i \neq 1$} 
				& $i < j \le n_1+1 < \ell$ & 2 \\
				& $i \le n_1+1 < j \leq \ell$ & 1 \\
				\hline
				\multirow{2}{*}{\textup{(c)} $\ell = n_1+1, i \neq 1$} 
				& $i < j \le k < n_1+1$ & 2 \\
				& $i \le k < j \le n_1+1$ & 1 \\
				\hline
				\multirow{2}{*}{\textup{(d)} $i = 1, k = n_1+1$} 
				& $1 < j \leq n_1+1 \le \ell$  & 3 \\
				& $1 < n_1+1 < j \leq \ell$ & 2 \\
				\hline
				\multirow{2}{*}{\textup{(e)} $i = 1, \ell = n_1+1$} 
				& $1 < j \le k < n_1+1$ & 3 \\
				& $1 \le k < j \le n_1+1$ & 1 \\
				\hline
			\end{tabular}
			\vspace{2mm}
			\caption{Dimensions of the morphism spaces between the selected generators in $\mathcal{W}(\operatorname{Sym}^2(\Sigma), \Lambda^{(2)})$. All other cases yield dimension 0.}
			\label{tab:dimensions}
		\end{table}
		All other ordered quadruples give dimension $0$. One checks that the quiver with relations $\Bbbk\widetilde Q_{n_1,n_2}/\widetilde I_{n_1,n_2}$ has the same dimensions for the corresponding Hom‑spaces. This can be done by constructing a reduction system satisfying the diamond condition. Since $\widetilde\varphi$ is surjective and both algebras have the same finite dimension, it must be an isomorphism. Hence $\widetilde{\mathcal{A}}_{n_1,n_2}\cong \Bbbk\widetilde Q_{n_1,n_2}/\widetilde I_{n_1,n_2}$.
	\end{proof}


\subsection{Grading on the dg algebra $\widetilde{\mathcal{A}}_{n_1,n_2}$}
\label{subsec:grading}
In the absence of any additional structures, $\mathcal{W}(\operatorname{Sym}^k(\Sigma), \Lambda^{(k)})$ of the symmetric product of a general surface $(\Sigma, \Lambda)$ is typically equipped only with a $\mathbb{Z}/2\mathbb{Z}$-grading. However, when the genus of $\Sigma$ is zero, the possible $\mathbb{Z}$-gradings on $\mathcal{W}(\operatorname{Sym}^k(\Sigma), \Lambda^{(k)})$ form a torsor over the group $H^1(\operatorname{Sym}^k(\Sigma), \mathbb{Z}) \cong H^1(\Sigma, \mathbb{Z})$ (see \cite{Seidel2000, Dyckerhoff-Jasso-Lekili2021}). In what follows, we show that the dg algebra $\widetilde{\mathcal{A}}_{n_1, n_2}$ can be equipped with a $\mathbb{Z}$-grading that  is completely determined by a single parameter: the degree of the arrow $u_{1,i}$ in the quiver $\widetilde{Q}_{n_1, n_2}$ (up to derived equivalence). This coincides with the above fact that the possible gradings form a torsor over $H^1(\operatorname{Sym}^2(\Sigma), \mathbb{Z}) \cong \mathbb{Z}$. 

We now equip the quiver $\widetilde{Q}_{n_1, n_2}$ with an arbitrary $\mathbb{Z}$-grading by assigning a degree $\deg(\alpha) = |\alpha| \in \mathbb{Z}$ to each arrow $\alpha \in \widetilde{Q}_{n_1, n_2}$, so that $\widetilde{\mathcal A}_{n_1,n_2}$ is differential $\mathbb Z$-graded. The following lemma establishes a fundamental grading identity,  which serves as a key algebraic constraint and a derived invariant for the symmetric square algebra $\widetilde{\mathcal{A}}_{n_1, n_2}$.

\begin{lemma}\label{der invariant} 
	For any well-defined $\mathbb{Z}$-grading on $\widetilde{\mathcal A}_{n_1,n_2}$, the following identity holds for each fixed $2 \leq i \leq n_1$ and $n_1+2 \leq j \leq n_1+n_2$: 
	\[ \sum\limits_{k=i}^{n_1} |x_{1,k}| + \sum\limits_{r=1}^{i-1} |y_{r,n_1+1}| -|u_{1,i}|+1 = \sum\limits_{l=1}^{n_1} |y_{l,j}| - |u_{1,j}|. \] 
	When \(n_1=1\), the first range for $i$ is empty and the equation reduces to the following  $$|y_{1,j}|-|u_{1,j}| = |y_{1,j'}|-|u_{1,j'}|$$ for any $3\leq j, j'\leq n_2+1$.
\end{lemma}

\begin{proof}
	The proof is structured in two parts: verifying the invariance of two arrow-degree sums within their respective index ranges, and subsequently proving a bridge identity connecting them.
	
	First, we claim that the number
	\[
	\sum\limits_{k=i}^{n_1} |x_{1,k}| + \sum\limits_{r=1}^{i-1} |y_{r,n_1+1}| - |u_{1,i}|
	\]
	is independent of the choice of $2 \leq i \leq n_1$. Indeed, it suffices to show that the expression yields the same value for consecutive indices $i$ and $i+1$:
	\[
	\sum\limits_{k=i}^{n_1} |x_{1,k}| + \sum\limits_{r=1}^{i-1} |y_{r,n_1+1}| -|u_{1,i}| = \sum\limits_{k=i+1}^{n_1} |x_{1,k}| + \sum\limits_{r=1}^{i} |y_{r,n_1+1}| -|u_{1,i+1}|.
	\]
	By canceling the common terms on both sides, this equality is equivalent to
	\[
	|x_{1,i}| - |u_{1,i}| = |y_{i,n_1+1}| - |u_{1,i+1}|,
	\]
	which holds due to the \emph{Commutative square relations of type 3} defined in Proposition \ref{prop:symsquadg}, namely  $u_{1,i}y_{i,n_1+1} = x_{1,i}u_{1,i+1}$. This proves the claim. 
	
	Similarly, the fact that the number
	\[
	\sum\limits_{l=1}^{n_1} |y_{l,j}| - |u_{1,j}|
	\]
	remains constant for all $n_1+2 \leq j \leq n_1+n_2$ follows from an analogous inductive argument utilizing the \emph{Commutative square relations of type 2}.
	
	Next we claim that, upon specializing to $i=n_1$ and $j=n_1+2$, these two invariant numbers are related as follows:
	\begin{equation} \label{eq:bridge_logic}
		|x_{1,n_1}| + \sum\limits_{r=1}^{n_1-1} |y_{r,n_1+1}| - |u_{1,n_1}| + 1 = \sum\limits_{l=1}^{n_1} |y_{l,n_1+2}| - |u_{1,n_1+2}|.
	\end{equation}
	
	Indeed, by the \emph{Commutative square relations of type 4} in Proposition \ref{prop:symsquadg}, we have the relation $u_{1,n_1}z_{n_1,n_1+1} = x_{1,n_1}x_{1,n_1+1}u_{1,n_1+2}$. This imposes the following constraint on degrees:
	\[
	|u_{1,n_1}| - |u_{1,n_1+2}| = |x_{1,n_1}| + |x_{1,n_1+1}| - |z_{n_1,n_1+1}|.
	\]
	Recalling the differential relation $d(z_{n_1,n_1+1}) = x_{n_1,n_1+1}y_{n_1,n_1+2}$, we obtain $|z_{n_1,n_1+1}| = |x_{n_1,n_1+1}| + |y_{n_1,n_1+2}| - 1$. Substituting this into the degree relation yields
	\[
	|u_{1,n_1}| - |u_{1,n_1+2}| = |x_{1,n_1}| + |x_{1,n_1+1}| - |x_{n_1,n_1+1}| - |y_{n_1,n_1+2}| + 1.
	\]
	
	Substituting the expression for $|u_{1,n_1}| - |u_{1,n_1+2}|$ back into Equation \eqref{eq:bridge_logic}, the identity holds if and only if
	\[
	\sum\limits_{r=1}^{n_1-1} |y_{r,n_1+1}| + |x_{n_1,n_1+1}| = |x_{1,n_1+1}| + \sum\limits_{l=1}^{n_1-1} |y_{l,n_1+2}|.
	\]
	This final equality is precisely the degree balance required by the path relation 
	\[ 
	x_{1,n_1+1}y_{1,n_1+2}y_{2,n_1+2}\dots y_{n_1-1,n_1+2} = y_{1,n_1+1}y_{2,n_1+1}\dots y_{n_1-1,n_1+1}x_{n_1,n_1+1} 
	\]
	obtained by iteratively applying the Commutative square relations of type 1. This proves the claim. The case of $n_1=1$ is similar. 
\end{proof}

We next show that, up to derived equivalence, the grading on the partially wrapped Fukaya category $\mathcal{W}(\operatorname{Sym}^2(\Sigma), \Lambda^{(2)})$ is completely determined by a single parameter: the degree of the arrow $u_{1,i}$ in the quiver $\widetilde{Q}_{n_1, n_2}$. This reflects the fact that the set of possible gradings forms a torsor over $H^1(\operatorname{Sym}^2(\Sigma), \mathbb{Z}) \cong \mathbb{Z}$. 

To make this precise, let us introduce a specific grading on $\widetilde{\mathcal A}_{n_1,n_2}$ and denote the resulting dg algebra by   $\widetilde{\mathcal{A}}_{n_1,n_2}'$. That is, $\widetilde{\mathcal{A}}_{n_1,n_2}'$ and  $\widetilde{\mathcal{A}}_{n_1,n_2}$  differ in the grading. To distinguish its generators from those of the original algebra, we append a prime to the arrows of $\widetilde{\mathcal{A}}_{n_1,n_2}'$. That is, the arrows in $\widetilde{\mathcal{A}}_{n_1,n_2}'$ are denoted by $x_{i,j}'$, $y_{i,j}'$, and $u_{1,k}'$.

Specifically, we equip $\widetilde{\mathcal{A}}_{n_1,n_2}'$ with a $\mathbb{Z}$-grading defined as follows:
\begin{itemize}[leftmargin=*, label=--]
	\item $|x_{i,j}'| = 0$ and $|y_{i,j}'| = 0$ for all $1 \leq i < j \leq n_1+n_2$;
	\item $|u_{1,k}'| = m+1$ for $2 \leq k \leq n_1$, and $|u_{1,k}'| = m$ for $n_1+2 \leq k \leq n_1+n_2$.
\end{itemize}
Furthermore, the differential relation $d(z_{i,i+1}') = x_{i,i+1}'y_{i,i+2}'$ forces the degree $$|z_{i,i+1}'| = |x_{i,i+1}'| + |y_{i,i+2}'| - 1 = -1\quad \text{for $1\leq i \leq n_1+n_2-2$}.$$

The following proposition establishes that the symmetric square dg algebra $\widetilde{\mathcal{A}}_{n_1,n_2}$ with any $\mathbb{Z}$-grading is derived equivalent to $\widetilde{\mathcal{A}}_{n_1,n_2}'$ for some $m$.

\begin{proposition}\label{prop:grading_derived_equiv}
	For any given $\mathbb{Z}$-grading, the symmetric square dg algebra $\widetilde{\mathcal{A}}_{n_1,n_2}$ is derived equivalent to the given standard model $\widetilde{\mathcal{A}}_{n_1,n_2}'$ for some $m$.  
\end{proposition}

\begin{proof}
	We begin by fixing the integer parameter $m$ for our target standard model $\widetilde{\mathcal{A}}_{n_1,n_2}'$. Based on the arbitrary $\mathbb{Z}$-grading on $\widetilde{\mathcal{A}}_{n_1,n_2}$, we set
	\[
	m = |u_{1,n_1+2}| - \sum\limits_{l=1}^{n_1} |y_{l,n_1+2}|.
	\]
	
	Next, we construct a compact generator that forces the $x$ and $y$ arrows to be of degree zero. 
	For each vertex \(L_{i,j}\) of \(\widetilde{\mathcal A}_{n_1,n_2}\), with \(1\leq i<j\leq n_1+n_2\), let \(P_{i,j}\) denote the corresponding indecomposable projective left module. For vertices and projective modules, we use the convention \(L_{a,b}=L_{b,a}\) and \(P_{a,b}=P_{b,a}\) whenever \(a>b\). We define the complex of projective modules
	\[
	T = \bigoplus_{1\le i<j\le n_1+n_2} P_{i,j}\bigl[\,s_{i,j}\,\bigr],
	\]
	where the degree shift $s_{i,j}$ is specifically chosen as
	\[
	s_{i,j} = -\sum\limits_{k=2}^{j-1}|x_{1,k}| - \sum\limits_{r=1}^{i-1}|y_{r,j}|,
	\]
	with the standard convention that empty sums evaluate to zero. 
	
	A standard verification shows that $T$ is a compact generator of the derived category, i.e., $\operatorname{thick}(T) = \operatorname{per}(\widetilde{\mathcal{A}}_{n_1,n_2})$. By our explicit choice of shifts $s_{i,j}$, the induced morphisms in the endomorphism dg algebra,
	\[
	x_{i,j}'\colon P_{i,j}[s_{i,j}] \longrightarrow P_{i,j+1}[s_{i,j+1}],\qquad
	y_{i,j}'\colon P_{i,j}[s_{i,j}] \longrightarrow P_{i+1,j}[s_{i+1,j}],
	\]
	are forced to have degree $0$. Consequently, the differential relation forces the induced morphisms $z_{i,i+1}' \colon P_{i,i+1}[s_{i,i+1}] \longrightarrow P_{i+1,i+2}[s_{i+1,i+2}]$ to have degree $-1$. 
	
	It remains to evaluate the degrees of the induced morphisms $u_{1,i}'$ for $i\in[n_1+n_2]\setminus\{n_1+1\}$. In the original algebra $\widetilde{\mathcal{A}}_{n_1,n_2}$, the arrow $u_{1,i}$ corresponds to an element
	\[
	u_{1,i} \in \Hom(P_{1,i}, P_{i,n_1+1}[|u_{1,i}|]).
	\]
	Passing to the shifted modules $P_{1,i}[s_{1,i}]$ and $P_{i,n_1+1}[s_{i,n_1+1}]$, the induced morphism $u_{1,i}'$ lives in the shifted Hom-space:
	\[
	u_{1,i}' \in \Hom\bigl(P_{1,i}[s_{1,i}], P_{i,n_1+1}[s_{i,n_1+1}][|u_{1,i}| + s_{i,n_1+1} - s_{1,i}]\bigr).
	\]
	Substituting the definitions of the shifts $s_{i,n_1+1}$ and $s_{1,i}$, the new degree of $u_{1,i}'$ expands to:
	\[
	|u_{1,i}| + s_{i,n_1+1} - s_{1,i} = 
	\begin{cases}
		|u_{1,i}| - \big(\sum\limits_{k=i}^{n_1} |x_{1,k}| + \sum\limits_{r=1}^{i-1} |y_{r,n_1+1}|\big), & \text{if } 2 \leq i \leq n_1, \\[1ex]
		|u_{1,i}| - \sum\limits_{l=1}^{n_1} |y_{l,i}|, & \text{if } n_1+2 \leq i \leq n_1+n_2.
	\end{cases}
	\]
	By the grading identity established in Lemma \ref{der invariant}, these expressions evaluate to the degrees prescribed by the standard model:
	\[
	|u_{1,i}| + s_{i,n_1+1} - s_{1,i} = 
	\begin{cases}
		m+1, & \text{if } 2 \leq i \leq n_1, \\[1ex]
		m, & \text{if } n_1+2 \leq i \leq n_1+n_2.
	\end{cases}
	\]
	Thus, the induced morphisms $u_{1,i}'$ in the endomorphism algebra $\mathcal{E}nd_{\widetilde{\mathcal{A}}_{n_1,n_2}}(T)$ have exactly the degrees of the corresponding generators in the standard model \(\widetilde{\mathcal A}_{n_1,n_2}'\).
	
	Finally, since this construction simply shifts the projective modules, all algebraic relations of $\widetilde{\mathcal{A}}_{n_1,n_2}$ (the commutative square, zig-zag, and monomial relations) are strictly preserved. The images of the irreducible morphisms satisfy exactly the defining relations of $\widetilde{\mathcal{A}}_{n_1,n_2}'$. Consequently, we obtain an isomorphism of dg algebras
	\[
	\mathcal{E}nd_{\widetilde{\mathcal{A}}_{n_1,n_2}}(T) \cong \widetilde{\mathcal{A}}_{n_1,n_2}',
	\]
	which establishes the derived equivalence between $\widetilde{\mathcal{A}}_{n_1,n_2}$ and $\widetilde{\mathcal{A}}_{n_1,n_2}'$.
\end{proof}

The cohomology algebra of the symmetric square dg algebra $\widetilde{\mathcal{A}}_{n_1,n_2}$ is denoted by $\mathcal{A}_{n_1,n_2}$, called the \textit{symmetric square algebra}. Before giving the quiver presentation of \(\mathcal{A}_{n_1,n_2}\), we record
the following elementary observation.  It explains why the diagonal generators
\(z_{i,i+1}\) disappear after passing to cohomology and why the relations
\(x_{i,i+1}y_{i,i+2}=0\) appear in the cohomology algebra.

\begin{lemma}\label{lem:cohomology-generators}
	The algebra \(\mathcal{A}_{n_1,n_2}=H^\ast(\widetilde{\mathcal{A}}_{n_1,n_2})\)
	is generated by the images of the arrows $x_{i,j}$, $y_{i,j}$ and $u_{1,i}$ in cohomology.  Moreover, the differential in
	\(\widetilde{\mathcal{A}}_{n_1,n_2}\) imposes precisely the additional relations $x_{i,i+1}y_{i,i+2}=0$ for $1\leq i\leq n_1+n_2-2$, among these generators.
\end{lemma}

\begin{proof}
	The differential of \(\widetilde{\mathcal{A}}_{n_1,n_2}\) is zero on all
	arrows of the forms \(x_{i,j}\), \(y_{i,j}\), and \(u_{1,i}\).  Hence these
	arrows determine elements of the cohomology algebra
	\(\mathcal{A}_{n_1,n_2}\).
	
	The only arrows on which the differential is nonzero are the diagonal arrows
	\(z_{i,i+1}\).  They satisfy
	\(
	d(z_{i,i+1})=x_{i,i+1}y_{i,i+2}.
	\)
	Therefore the path \(x_{i,i+1}y_{i,i+2}\) is zero in
	\(\mathcal{A}_{n_1,n_2}\).  This gives the relations
	\(
	x_{i,i+1}y_{i,i+2}=0.
	\)
	
	It remains to explain why the diagonal arrows do not contribute additional cohomology generators. In this argument, we only use the following local simplifications. 
	
	First, by the type \(4\) relation, any path containing the subpath \(u_{1,n_1}z_{n_1,n_1+1}\) is rewritten as a path containing the subpath \(x_{1,n_1}x_{1,n_1+1}u_{1,n_1+2}.\) Thus such a path is no longer counted as a path containing a diagonal arrow.
	
	Second, any path containing the subpath \(u_{1,n_1+2}z_{n_1+1,n_1+2}\) is zero by the defining relations.
	
	We now consider paths which still contain diagonal arrows after these two
	simplifications. If such a path contains two or more diagonal arrows, then the
	defining relations allow us to move two of them next to each other. Hence the
	path contains a subpath of the form \(z_{r,r+1}z_{r+1,r+2}\)
	for some \(r\), and is therefore zero by the defining relations. Thus it is
	enough to consider paths containing exactly one diagonal arrow \(z_{i,i+1}\).
	
	For fixed vertices \(L_{i,j}\) and \(L_{k,\ell}\), the space spanned by such paths
	from \(L_{i,j}\) to \(L_{k,\ell}\) is either zero- or one-dimensional. Let \(p=p_1z_{i,i+1}p_2\) be a nonzero such path, where \(p_1\) and \(p_2\) contain no diagonal arrows. Since the differential is zero on all arrows except the diagonal arrows, we have \(d(p)=(-1)^{|p_1|} p_1x_{i,i+1}y_{i,i+2}p_2.\) By the one-dimensionality statement above, this term cannot cancel with the
	differential of another path containing a diagonal arrow. Hence no nonzero
	cohomology class is represented by a path containing a diagonal arrow after the
	two simplifications above. Consequently, the diagonal arrows \(z_{i,i+1}\) do
	not contribute additional cohomology generators; their only effect in
	cohomology is to impose $x_{i,i+1}y_{i,i+2}=0.$
\end{proof}

Lemma \ref{lem:cohomology-generators} immediately gives the following representation of
\(\mathcal A_{n_1,n_2}\) as a path algebra of a quiver modulo an ideal of
relations.

\begin{prop}\label{prop:cohomologysymsqu}
	The symmetric square algebra $\mathcal{A}_{n_1,n_2}$ is isomorphic to the path algebra of the quiver $Q_{n_1,n_2}$ (depicted in Figure~\ref{quiver A_{n1,n2}}) modulo the ideal of relations $I_{n_1,n_2}$, defined as follows:
	\begin{enumerate}[label=(\arabic*), leftmargin=*]
		\item The arrows of $Q_{n_1,n_2}$ consist of
		\begin{itemize}[label=--]
			\item \textsc{Horizontal arrows}: 
			\(
			x_{i,j}= \operatorname{id}_{L_i} \otimes \alpha_j 
			\colon L_{i,j} \longrightarrow L_{i,j+1}
			\) for \(1\leq i<j\leq n_1+n_2-1\);
			\item \textsc{Vertical arrows}: 
			\(
			y_{i,j}= \alpha_i \otimes \operatorname{id}_{L_j} 
			\colon  L_{i,j} \longrightarrow L_{i+1,j}
			\) for \(1\leq i\leq n_1+n_2-2\) and \(i+2\leq j\leq n_1+n_2\);
			\item \textsc{Long blue arrows}: 
			\(
			u_{1,i}= \beta \otimes \operatorname{id}_{L_i} 
			\colon L_{1,i} \longrightarrow L_{i,n_1+1}
			\) for \(2\leq i\leq n_1+n_2\) and \(i\neq n_1+1\), where \(L_{i,n_1+1}\) is identified with \(L_{n_1+1,i}\) whenever \(i>n_1+1\).
		\end{itemize}
		
		\item The ideal $I_{n_1,n_2}$ is generated by the following relations. We use the same convention as above: a displayed relation involving an undefined path is omitted.
		\begin{itemize}[label=--]
			\item \textsc{Commutative square relations of type 1}: $x_{i,j} y_{i,j+1} = y_{i,j} x_{i+1,j}$ for $1\leq i<j-1 \leq n_1+n_2-2$; see Figure \ref{fig:commutative-square-relations};
			\item \textsc{Commutative square relations of type 2}: $u_{1,i} x_{n_1+1,i}=x_{1,i} u_{1,i+1}$ for $n_1+2\leq i\leq  n_1+n_2-1$; see Figure \ref{fig:commutative-square-relations};
			\item \textsc{Commutative square relations of type 3}: $u_{1,i} y_{i,n_1+1}=x_{1,i}u_{1,i+1}$ for $1<i<n_1$; see Figure \ref{fig:commutative-square-relations};
			\item \textsc{Monomial relations}: $u_{1,i} x_{i,n_1+1}=0$ for $1< i\leq n_1$; $u_{1,n_1+i} y_{n_1+1,n_1+i}=0$ for $3\leq i \leq n_2$; and $x_{i,i+1} y_{i,i+2} = 0$ for $1\leq i \leq n_1+n_2-2$; see Figure \ref{fig:monomial-relations}.
		\end{itemize}
	\end{enumerate}

	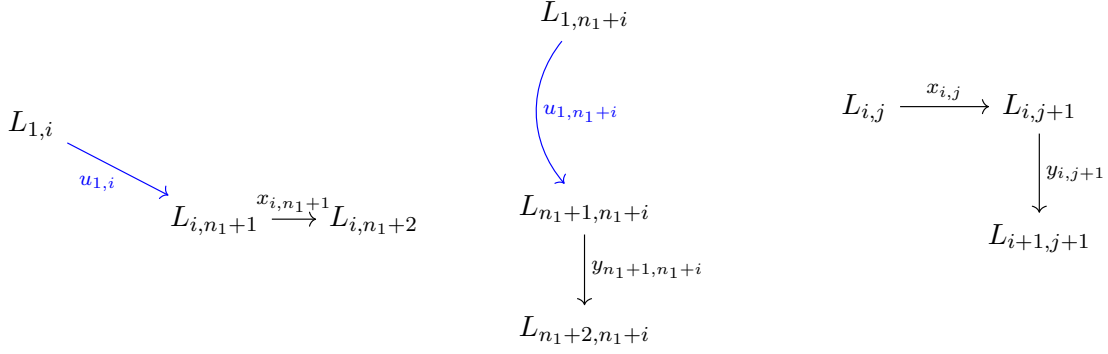
\begin{figure}[htbp]
		\centering
		\begin{subfigure}[b]{0.40\textwidth}
			\centering
			\begin{tikzcd}[ampersand replacement=\&, row sep=0.6cm, column sep=0.6cm]
				L_{1,i} \ar[rrd, blue,"u_{1,i}" {xshift=2pt, yshift=0pt}, swap]  \& \& \& \\
				\&  \& L_{i,n_1+1} \ar[r,"x_{i,n_1+1}" ] \& L_{i,n_1+2} 
			\end{tikzcd}
		\end{subfigure}
		\begin{subfigure}[b]{0.25\textwidth}
			\centering
			\begin{tikzcd}[ampersand replacement=\&, row sep=0.95cm, column sep=0.6cm]
				L_{1,n_1+i} \ar[dd, blue, bend right=40, "u_{1,n_1+i}"]  \\
				\&  \\
				L_{n_1+1,n_1+i} \ar[d, "y_{n_1+1,n_1+i}"] \\
				L_{n_1+2,n_1+i}				
			\end{tikzcd}
		\end{subfigure}\hfill
		\begin{subfigure}[b]{0.32\textwidth}
			\centering
			\begin{tikzcd}[ampersand replacement=\&, row sep=large, column sep=1cm]
				L_{i,j} \ar[r,"x_{i,j}"]  \& L_{i,j+1} \ar[d,"y_{i,j+1}"] \\
				\& L_{i+1,j+1}
			\end{tikzcd}
		\end{subfigure}\hfill
		\caption{Diagrams of monomial relations.}
		\label{fig:monomial-relations}
	\end{figure}
\end{prop}

\begin{figure}[htbp]
	\centering
	\begin{tikzcd}[
		column sep=0.5cm, 
		row sep=0.7cm,
		every matrix/.append style={name=mycd}
		]
		\text{\scriptsize $L_{1,2}$} \ar[r] \ar[blue,drrrr,bend left=1]& 
		\text{\scriptsize $L_{1,3}$} \ar[r] \ar[d]\ar[blue, ddrrr]  & 
		\cdots \ar[r] & 
		\text{\scriptsize $L_{1,n_1}$} \ar[d] \ar[r] \ar[blue,ddddr,bend right=10] & 
		\text{\scriptsize $L_{1,n_1+1}$} \ar[d]\ar[r] &
		\text{\scriptsize $L_{1,n_1+2}$} \ar[d]\ar[r]\ar[blue,ddddd,bend right=40]&\cdots\ar[r]\ar[blue,ddddd,bend right=40]&
		\text{\scriptsize $L_{1,n_1+n_2}$} \ar[d]\ar[blue,ddddd,bend right=40]\\
		& 
		\text{\scriptsize $L_{2,3}$} \ar[r] & 
		\cdots \ar[r] & 
		\text{\scriptsize $L_{2,n_1}$} \ar[d] \ar[r] & 
		\text{\scriptsize $L_{2,n_1+1}$} \ar[d] \ar[r] &
		\text{\scriptsize $L_{2,n_1+2}$} \ar[d]\ar[r]&\cdots\ar[r]&
		\text{\scriptsize $L_{2,n_1+n_2}$} \ar[d]\\
		&&&
		\vdots \ar[d]  & 
		\vdots \ar[d] &\vdots \ar[d]& & \vdots \ar[d]\\
		&&&
		\text{\scriptsize $L_{n_1-1,n_1}$} \ar[r] & 
		\text{\scriptsize $L_{n_1-1,n_1+1}$} \ar[d] \ar[r]&
		\text{\scriptsize $L_{n_1-1,n_1+2}$}\ar[r]\ar[d]&\cdots\ar[r]&
		\text{\scriptsize $L_{n_1-1,n_1+n_2}$}\ar[d]\\
		&&&&
		\text{\scriptsize $L_{n_1,n_1+1}$}\ar[r]&
		\text{\scriptsize $L_{n_1,n_1+2}$}\ar[r]\ar[d]&\cdots\ar[r]&
		\text{\scriptsize $L_{n_1,n_1+n_2}$}\ar[d]\\
		&&&&&
		\text{\scriptsize $L_{n_1+1,n_1+2}$}\ar[r]&\cdots\ar[r]&
		\text{\scriptsize $L_{n_1+1,n_1+n_2}$}\ar[d]\\
		&&&&&&
		&\vdots \ar[d] \\
		&&&&&&&
		\text{\scriptsize $L_{n_1+n_2-1,n_1+n_2}$}
	\end{tikzcd}
	\captionsetup{justification=centering}
	\caption{Quiver $Q_{n_1,n_2}$ of the symmetric square algebra $\mathcal{A}_{n_1,n_2}$, where the horizontal and vertical arrows are $x$ and $y$, the blue and red ones are $u$ and $z$.}
	\label{quiver A_{n1,n2}}
\end{figure}
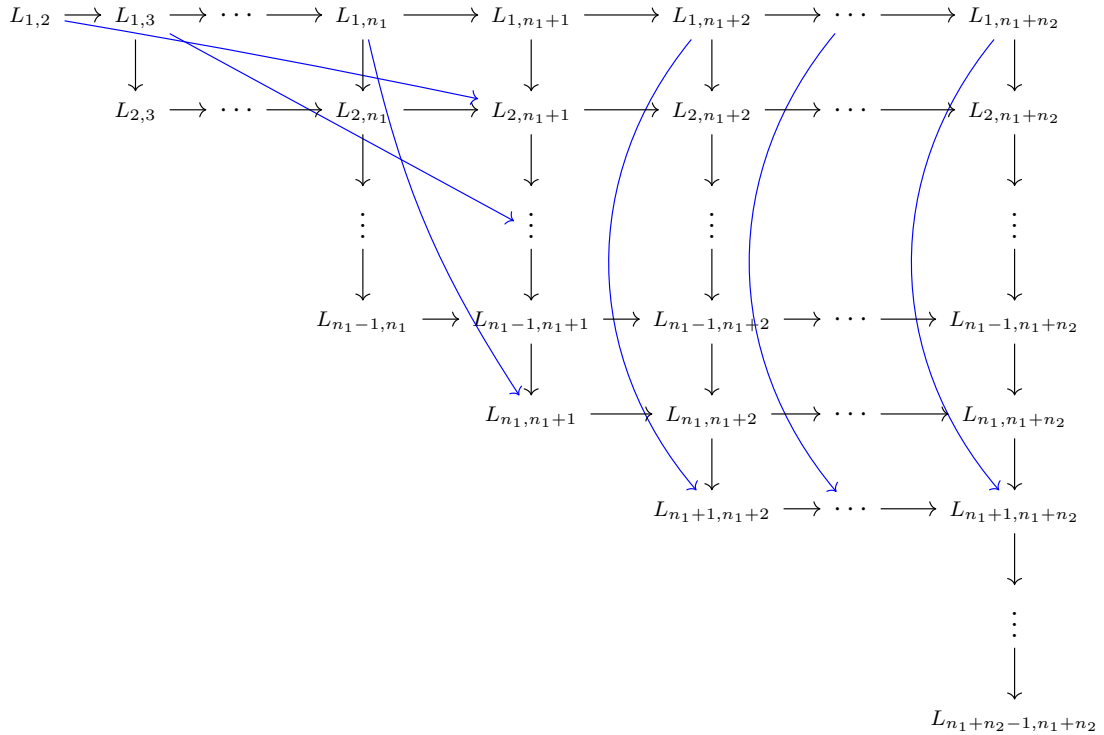

\begin{proof}
	By Lemma~\ref{lem:cohomology-generators}, the cohomology algebra
	\(\mathcal{A}_{n_1,n_2}\) is generated by the images of the horizontal arrows
	\(x_{i,j}\), the vertical arrows \(y_{i,j}\), and the long blue arrows
	\(u_{1,i}\).  These are precisely the arrows of the quiver
	\(Q_{n_1,n_2}\).
	
	The commutative square relations of types \(1\), \(2\), and \(3\), as well as
	the monomial relations involving the long blue arrows \(u_{1,i}\), already hold in
	the dg algebra \(\widetilde{\mathcal A}_{n_1,n_2}\).  Hence they also hold in
	\(\mathcal{A}_{n_1,n_2}\).  In addition, Lemma~\ref{lem:cohomology-generators}
	shows that the differential contributes the relations
	\[
	x_{i,i+1}y_{i,i+2}=0,
	\qquad
	1\leq i\leq n_1+n_2-2.
	\]
	Therefore all generators of \(I_{n_1,n_2}\) hold in
	\(\mathcal{A}_{n_1,n_2}\).  We thus obtain a natural surjective homomorphism
	\[
	\Bbbk Q_{n_1,n_2}/I_{n_1,n_2}
	\longrightarrow
	\mathcal{A}_{n_1,n_2}.
	\]
	
	Conversely, Lemma~\ref{lem:cohomology-generators} says that the only new
	relations produced by passing from
	\(\widetilde{\mathcal A}_{n_1,n_2}\) to cohomology are precisely the relations \(x_{i,i+1}y_{i,i+2}=0.\) All other relations among the horizontal arrows \(x_{i,j}\), the vertical
	arrows \(y_{i,j}\), and the long arrows \(u_{1,i}\) are inherited from the
	defining relations of \(\widetilde{\mathcal A}_{n_1,n_2}\).  Namely, these are
	the commutative square relations of types \(1\), \(2\), and \(3\), together
	with the monomial relations involving the long blue arrows \(u_{1,i}\).  These are
	exactly the generators of \(I_{n_1,n_2}\).  Hence the above homomorphism is
	injective. As a result, we obtain an isomorphism of algebras $\mathcal{A}_{n_1,n_2}\cong \Bbbk Q_{n_1,n_2}/I_{n_1,n_2}.$
\end{proof}

Recall that a dg algebra \(A\) is called \textit{formal} if \(A\) and \(H^\ast(A)\), endowed with the induced product and zero differential, are connected by a zig-zag of dg algebra quasi-isomorphisms. We now use the above quiver descriptions of \(\widetilde{\mathcal A}_{n_1,n_2}\) and \(\mathcal A_{n_1,n_2}\) to study the formality of \(\widetilde{\mathcal A}_{n_1,n_2}\).

\begin{theorem}\label{thm:A_infty_qiso}
	Let $(\Sigma, \Lambda)$ be an annulus with stops of type
	$\widetilde{\mathbb{A}}_{n_1,n_2}$. Then the symmetric square dg algebra
	$\widetilde{\mathcal A}_{n_1,n_2}$ satisfies:
	\begin{enumerate}
		\item If $n_1=1$, then $\widetilde{\mathcal A}_{1,n_2}$ is formal.
		\item If $n_1\ge 2$, then $\widetilde{\mathcal A}_{n_1,n_2}$ admits a
		minimal $A_\infty$-model $(\mathcal A_{n_1,n_2}, m_2, m_3)$ with
		$m_r=0$ for all $r\ge4$ and the ternary operation \(m_3\) is defined on basis paths by
		\[
		m_3(p\,u_{1,n_1}, x_{n_1,n_1+1}, y_{n_1,n_1+2}\,q)
		=(-1)^{|pu_{1,n_1}|+1} p\, x_{1,n_1} x_{1,n_1+1} u_{1,n_1+2}\, q,
		\]
		for all composable paths $p,q$ where the right-hand side is nonzero, and $m_3$ vanishes on all other triples. 
	\end{enumerate}
\end{theorem}

\begin{proof}
	First assume \(n_1=1\). We construct a dg algebra morphism
	\[
	\pi:\widetilde{\mathcal A}_{1,n_2}\longrightarrow
	\mathcal A_{1,n_2}=H^\ast(\widetilde{\mathcal A}_{1,n_2},d)
	\]
	by sending every horizontal arrow \(x_{i,j}\), vertical arrow \(y_{i,j}\), and
	long blue arrow \(u_{1,i}\) to its cohomology class, and by sending every
	diagonal arrow \(z_{i,i+1}\) to \(0\). This is compatible with the differential:
	all horizontal, vertical, and long blue arrows are closed, while
	\[
	\pi(d(z_{i,i+1}))
	=
	\pi(x_{i,i+1}y_{i,i+2})
	=
	0,
	\]
	because \(x_{i,i+1}y_{i,i+2}\) is zero in the cohomology algebra
	\(\mathcal A_{1,n_2}\). The defining relations of
	\(\widetilde{\mathcal A}_{1,n_2}\) are mapped to the defining relations of
	\(\mathcal A_{1,n_2}\). Hence \(\pi\) is a dg algebra morphism. By the
	description of the cohomology algebra, \(\pi\) induces an isomorphism on
	cohomology. Thus \(\widetilde{\mathcal A}_{1,n_2}\) is formal.
	
	Now assume \(n_1\ge 2\). We define a minimal \(A_\infty\)-structure on
	\(\mathcal A_{n_1,n_2}\) by taking \(m_2\) to be the path multiplication,
	\(m_3\) as in the statement, and \(m_r=0$ for $r\ge 4.\)
	The \(A_\infty\)-relations for \((\mathcal A_{n_1,n_2},m_2,m_3)\) reduce to two conditions: the Hochschild cocycle condition \([m_2,m_3]=0\) and the quadratic condition \([m_3,m_3]=0\). The sign in the definition of \(m_3\) is chosen so that the boundary terms in the Hochschild cocycle condition cancel. The relation \([m_3,m_3]=0\) follows from the support condition: the output of a nonzero \(m_3\) cannot be used as an input of another nonzero \(m_3\). Hence \((\mathcal A_{n_1,n_2},m_2,m_3)\) is a well-defined  \(A_\infty\)-algebra.
	
	We now construct an \(A_\infty\)-morphism
	\[
	F=(F_1,F_2,F_3,\dots):(\mathcal A_{n_1,n_2},m_2,m_3)
	\longrightarrow
	(\widetilde{\mathcal A}_{n_1,n_2},d,\mu_2).
	\]
	Let \(F_1\) be the linear section sending each cohomology class represented by a
	path involving only horizontal arrows, vertical arrows, and long blue arrows to
	the same path in \(\widetilde{\mathcal A}_{n_1,n_2}\). 
	Define \(F_2\) on basis paths by
	\[
	F_2(p\,x_{i,i+1},\,y_{i,i+2}\,q)
	=
	(-1)^{|p|+1}p\,z_{i,i+1}\,q
	\]
	for all composable paths \(p,q\), and set \(F_2=0\) on all other pairs of basis
	paths. Here \(|p|\) denotes the \(\mathbb Z\)-degree of the path \(p\). Finally, set $F_r=0$, for $r\ge 3$.
	
	We verify the \(A_\infty\)-morphism equations. 
	
	For arity one, the equation says that \(F_1\) sends the chosen cohomology representatives to closed elements in \(\widetilde{\mathcal A}_{n_1,n_2}\), which holds by construction.
	
	For arity two, the equation is
	\[
	F_1(m_2(a,b))-\mu_2(F_1(a),F_1(b))-dF_2(a,b)=0.
	\]
	The only nontrivial case is when $a=p x_{i,i+1}$ and $b=y_{i,i+2}q$.
	Thus \(m_2(a,b)=0\) in \(\mathcal A_{n_1,n_2}\), while
	\[
	\begin{aligned}
		dF_2(a,b)
		&=
		(-1)^{|p|+1}d(p\,z_{i,i+1}\,q) \\
		&=
		(-1)^{|p|+1}(-1)^{|p|}p\,d(z_{i,i+1})\,q \\
		&=
		-p\,x_{i,i+1}y_{i,i+2}\,q
		=
		-\mu_2(F_1(a),F_1(b)).
	\end{aligned}
	\]
	Thus the arity-two equation holds. In all other cases \(F_2=0\), and the equation reduces to the compatibility of \(F_1\) with path multiplication.
	
	For arity three, the equation is
	\[
	F_1(m_3(a,b,c))-F_2(a,m_2(b,c))+F_2(m_2(a,b),c)
	=
	(-1)^{|a|}\mu_2(F_1(a),F_2(b,c))-\mu_2(F_2(a,b),F_1(c)).
	\]
	The only essential case is
	\[
	(a,b,c)=
	(pu_{1,n_1},\,x_{n_1,n_1+1},\,y_{n_1,n_1+2}q).
	\]
	Here \(m_2(a,b)=m_2(b,c)=0\) and \(F_2(a,b)=0\), so the equation becomes
	\[
	F_1(m_3(a,b,c))=(-1)^{|a|}\mu_2(F_1(a),F_2(b,c)).
	\]
	By the definitions of \(m_3\) and \(F_2\), this equality is exactly
	\[
	(-1)^{|pu_{1,n_1}|+1} p\, x_{1,n_1} x_{1,n_1+1} u_{1,n_1+2}\, q
	=
	(-1)^{|pu_{1,n_1}|+1} p\,u_{1,n_1}z_{n_1,n_1+1}\,q,
	\]
	which is one of the defining relations of
	\(\widetilde{\mathcal A}_{n_1,n_2}\). All other triples are either ordinary associativity cases, or vanish by the support conditions of \(m_3\) and \(F_2\).
	
	It remains to check arity four. After substituting
	\[
	F_r=0\ (r\ge 3),\qquad m_r=0\ (r\ge 4),\qquad \mu_r=0\ (r\ge 3),
	\]
	all terms vanish except those involving \(m_3\), \(F_2\), and the product of
	two \(F_2\)-terms. Thus, for four composable inputs \(a,b,c,d\), it remains to
	verify
	\begin{align}\label{align:f2}
		F_2(m_3(a,b,c),d)+(-1)^{|a|}F_2(a,m_3(b,c,d))
		=
		(-1)^{|a|+|b|}\mu_2(F_2(a,b),F_2(c,d)).
	\end{align}
	
	By definition, \(F_2\) is nonzero only on pairs of the form $(p x_{i,i+1},\, y_{i,i+2}q),$ and \(m_3\) is nonzero only on triples of the form $(pu_{1,n_1},\,x_{n_1,n_1+1},\,y_{n_1,n_1+2}q).$ Thus, in the arity-four equation, all terms are automatically zero unless these input forms occur.

	We first show that the two terms on the left-hand side of \eqref{align:f2} vanish. Suppose that
	\(F_2(m_3(a,b,c),d)\) is potentially nonzero. Then the preceding \(m_3\)-term
	must produce a path of the form
	\[
	p\,x_{1,n_1}x_{1,n_1+1}u_{1,n_1+2}\,q.
	\]
	For this path to be the first input of a supported \(F_2\)-term, the path \(q\)
	must end in a horizontal arrow. Thus we may write \(q=q_1x_{i,i+1},\) and \( d=y_{i,i+2}d_1\)
	for some composable paths \(q_1,d_1\). Then, by the definition of \(F_2\),
	\[
	F_2(m_3(a,b,c),d)
	=
	(-1)^{|p\,x_{1,n_1}x_{1,n_1+1}u_{1,n_1+2}q_1|+|p u_{1,n_1}|}\,p\,x_{1,n_1}x_{1,n_1+1}u_{1,n_1+2}\,q_1z_{i,i+1}d_1.
	\]
	Using the zig-zag relations, the diagonal arrow can be moved to the left along
	the path \(q_1\), so that \(q_1z_{i,i+1}=z_{n_1+1,n_1+2}q_2\) for some composable path \(q_2\). Hence
	\[
	F_2(m_3(a,b,c),d)
	=
	(-1)^{|p\,x_{1,n_1}x_{1,n_1+1}u_{1,n_1+2}q_1|+|p u_{1,n_1}|}\,p\,x_{1,n_1}x_{1,n_1+1}u_{1,n_1+2}z_{n_1+1,n_1+2}q_2d_1
	=0
	\]
	by the monomial relation \(u_{1,n_1+2}z_{n_1+1,n_1+2}=0\). If \(q\) does not
	end in a horizontal arrow, then the term is zero by the support condition of
	\(F_2\).
	
	Similarly, \(F_2(a,m_3(b,c,d))\) can be nonzero only if the path produced by
	\(m_3(b,c,d)\) begins with a vertical arrow. But a nonzero \(m_3(b,c,d)\) has
	the form \(p\,x_{1,n_1}x_{1,n_1+1}u_{1,n_1+2}\,q,\)
	where \(p u_{1,n_1}\) is composable. Indeed, the path \(p\) has target \(L_{1,n_1}\). In the quiver, any nontrivial path ending at a first-row vertex \(L_{1,n_1}\) is composed only of horizontal arrows in the first row. Hence the output of \(m_3(b,c,d)\) cannot begin with a vertical arrow. Therefore \(m_3(b,c,d)\) cannot be the second input of a supported \(F_2\)-term, and so
	\(
	F_2(a,m_3(b,c,d))=0.
	\)
	Thus the left-hand side of the arity-four equation is zero.
	
	On the other hand, the right-hand side of \eqref{align:f2}  can be nonzero only if both \(F_2(a,b)\) and \(F_2(c,d)\) are nonzero. In that case their product contains two diagonal arrows. Using the zig-zag relations, this product reduces to a path containing a consecutive pair \(z_{j,j+1}z_{j+1,j+2},\) which is zero by the monomial relation. Hence the right-hand side is zero. Hence
	the arity four equation holds.
	
	For arity \(r\ge 5\), every term contains either \(F_s\) with \(s\ge 3\),
	\(m_s\) with \(s\ge 4\), or \(\mu_s\) with \(s\ge 3\). All such terms vanish by
	construction. Therefore all higher \(A_\infty\)-morphism equations hold.
	
	Finally, \(F_1\) induces the identity map on \(H^\ast(\widetilde{\mathcal A}_{n_1,n_2},d)=\mathcal A_{n_1,n_2}.\) Therefore \(F\) is an \(A_\infty\)-quasi-isomorphism.
\end{proof}

We use Massey products to detect non-formality. Triple Massey products were introduced by Massey~\cite{Massey1958}; see also May~\cite{May1969} for the dg algebra setting. Let \(A\) be a dg algebra, and let \([a],[b],[c]\in H^*(A)\) be represented by homogeneous cocycles \(a,b,c\). If \([a][b]=[b][c]=0\), choose \(s,t\in A\) such that \(d(s)=ab\) and \(d(t)=bc\). With our convention, the triple Massey product \(\langle [a],[b],[c]\rangle\) is the set of cohomology classes \(\left[(-1)^{|a|}at-sc\right]\) obtained from all such choices of \(s\) and \(t\). Equivalently, it is a coset modulo the indeterminacy 
\[ [a]\cdot H^{|b|+|c|-1}(A)+H^{|a|+|b|-1}(A)\cdot[c].
\]
In particular, to show that \(0\notin \langle [a],[b],[c]\rangle\), it suffices to show that one representative does not lie in the indeterminacy. A defined Massey product which does not contain \(0\) obstructs formality; see \cite{Tralle-Oprea1997} and \cite[Theorem~2.4]{Biswas2016}.

\begin{corollary}\label{Massey product}
	For \(n_1\ge 2\), the dg algebra \(\widetilde{\mathcal A}_{n_1,n_2}\) is not formal. 
	
\end{corollary}

\begin{proof}
	By the definition recalled above, it suffices to compute one representative \(\left[(-1)^{|a|}at-sc\right]\)
	of the triple Massey product and show that its cohomology class does not lie in the corresponding indeterminacy for suitable choices \(s,t\) with \(d(s)=ab\) and \(d(t)=bc\).
	
	Take $a=u_{1,n_1}$, $b=x_{n_1,n_1+1}$ and $c=y_{n_1,n_1+2}.$ Since \(u_{1,n_1}x_{n_1,n_1+1}=0\) and \[
	x_{n_1,n_1+1}y_{n_1,n_1+2}=d(z_{n_1,n_1+1}),
	\]
	the triple Massey product $\langle [u_{1,n_1}], [x_{n_1, n_1+1}], [y_{n_1, n_1+2}] \rangle$ is nonempty. We choose $s=0$ and $t=z_{n_1,n_1+1}$. Hence it contains the class
	\[
	\left[(-1)^{|u_{1,n_1}|}u_{1,n_1}z_{n_1,n_1+1}\right]
	=
	\left[(-1)^{|u_{1,n_1}|}x_{1,n_1}x_{1,n_1+1}u_{1,n_1+2}\right],
	\]
	where the equality follows from the type \(4\) relation. This class is nonzero
	in \(H^\ast(\widetilde{\mathcal A}_{n_1,n_2})\), and by the quiver presentation
	it is not contained in the indeterminacy
	\[
	[u_{1,n_1}]H^\ast(L_{n_1,n_1+1},L_{n_1+1,n_1+2})
	+
	H^\ast(L_{1,n_1},L_{n_1,n_1+2})[y_{n_1,n_1+2}].
	\]
	Indeed, the first summand is zero because \(H^\ast(L_{n_1,n_1+1},L_{n_1+1,n_1+2})\) has no nonzero class represented by a path compatible with the defining relations, while every element of the second summand factors through a path ending in \(y_{n_1,n_1+2}\), whereas \(x_{1,n_1}x_{1,n_1+1}u_{1,n_1+2}\) does not. Thus the Massey product $\langle [u_{1,n_1}], [x_{n_1, n_1+1}], [y_{n_1, n_1+2}] \rangle$ does not contain zero. Therefore
	\(\widetilde{\mathcal A}_{n_1,n_2}\) is not formal.
\end{proof}

\begin{remark}\normalfont
	The non-formality for \(n_1\geq 2\) can also be seen from the Hochschild cohomology computations in Section \ref{section:hhcohomology}. Indeed, if \(\widetilde{\mathcal A}_{n_1,n_2}\) were formal, then it would be $A_\infty$-quasi-isomorphic to its cohomology algebra \((\mathcal A_{n_1,n_2},m_2)\). Since Hochschild cohomology is invariant under $A_\infty$-quasi-isomorphisms, this would imply
	\[
	\operatorname{HH}^*(\widetilde{\mathcal A}_{n_1,n_2})
	\cong
	\operatorname{HH}^*(\mathcal A_{n_1,n_2},m_2).
	\]
	However, comparing Theorem~\ref{thm:graded_HH} with Theorem~\ref{thm:dg_HH_dimensions} shows that these Hochschild cohomology groups are not isomorphic for \(n_1\geq 2\); in particular, the relevant degree-two components have different dimensions. Therefore \(\widetilde{\mathcal A}_{n_1,n_2}\) cannot be formal.
\end{remark}

\section{Hochschild cohomology of the symmetric dg algebras}\label{section:hhcohomology}
Recall from Theorem~\ref{thm:A_infty_qiso} that the symmetric square dg algebra $\widetilde{\mathcal{A}}_{n_1,n_2}$ admits a minimal model $(\mathcal A_{n_1,n_2}, m_2, m_3)$.  
Consequently, its Hochschild cohomology coincides with that of its minimal $A_\infty$-model $(\mathcal A_{n_1,n_2}, m_2, m_3)$.  
Thus, we obtain the following isomorphisms 
\[\operatorname{HH}^*(\mathcal{W}(\operatorname{Sym}^2(\Sigma), \Lambda^{(2)})) \cong \operatorname{HH}^*(\widetilde{\mathcal{A}}_{n_1,n_2}) \cong \operatorname{HH}^*(\mathcal A_{n_1,n_2}, m_2, m_3).\]
The strategy of this section is to separate the computation into two steps. We first compute the Hochschild cohomology of the cohomology algebra \(\mathcal A_{n_1,n_2}\) with only its ordinary multiplication \(m_2\), using the reduction system and the resulting small bimodule resolution. We then incorporate the higher product \(m_3\) by means of the spectral sequence associated with the internal degree filtration on the Hochschild complex; the first nontrivial differential is given by the Gerstenhaber bracket \([m_3,-]\).

\subsection{Hochschild cohomology of $\mathcal{A}_{n_1,n_2}$} 
In this subsection we compute the Hochschild cohomology of the cohomology algebra \(\mathcal A_{n_1,n_2}\), viewed as an ordinary graded algebra with only the multiplication \(m_2\). This computation will later serve as the \(E_1\)-page of the spectral sequence used to recover the Hochschild cohomology of the dg algebra \(\widetilde{\mathcal A}_{n_1,n_2}\).

Throughout this subsection, for simplicity we ignore the internal grading and assume that all arrows have degree \(0\). The graded refinement will be discussed afterwards. We begin by introducing a reduction system for $\mathcal{A}_{n_1,n_2}$, which will serve as the main tool for the computations that follow.

\begin{proposition}\label{reduction system}
	The symmetric square algebra $\mathcal{A}_{n_1,n_2}$ admits a reduction system $R$ satisfying the diamond condition, given by
	\allowdisplaybreaks
	\begin{align*}
		R &=\bigl\{(x_{i,j}y_{i,j+1},\; y_{i,j}x_{i+1,j}) \mid 
		1\leq i<j-1 \leq n_1+n_2-2 \bigr\} \\
		&\cup\; \bigl\{(u_{1,i}x_{n_1+1, i},\; x_{1,i}u_{1,i+1}) \mid 
		n_1+2 \le i \le n_1+n_2-1 \bigr\} \\
		&\cup\; \bigl\{(u_{1,i}y_{i, n_1+1},\; x_{1,i}u_{1,i+1}) \mid 
		1 < i < n_1 \bigr\} \\
		&\cup\; \bigl\{(u_{1,i}x_{i, n_1+1},\; 0) \mid 
		1< i \le n_1 \bigr\} \\
		&\cup\; \bigl\{(u_{1,i}y_{n_1+1, i},\; 0) \mid 
		n_1+3 \le i \le n_1+n_2 \bigr\}\\
		&\cup\; \bigl\{(x_{i,i+1}y_{i,i+2},\; 0) \mid 
		1 \le i \le n_1+n_2-2 \bigr\},
	\end{align*}
	where $S = S^{(1)} \cup S^{(2)} \cup S^{(3)} \cup S^{(4)} \cup S^{(5)} \cup S^{(6)}$ (see Definition \ref{def reduction system}) consists of the following sets:
	\allowdisplaybreaks
	\begin{align*}
		S^{(1)} &= \{x_{i,j}y_{i,j+1} \mid 1\leq i<j-1 \leq n_1+n_2-2 \}, \\
		S^{(2)} &= \{u_{1,i}x_{n_1+1, i} \mid n_1+2 \le i \le n_1+n_2-1 \}, \\
		S^{(3)} &= \{u_{1,i}y_{i, n_1+1} \mid 1 < i < n_1 \}, \\
		S^{(4)} &= \{u_{1,i}x_{i, n_1+1} \mid 1 < i \le n_1 \}, \\
		S^{(5)} &= \{u_{1,i}y_{n_1+1, i} \mid n_1+3 \le i \le n_1+n_2 \},\\
		S^{(6)} &= \{x_{i,i+1}y_{i,i+2} \mid 1 \le i \le n_1+n_2-2 \}.
	\end{align*}
	Here $S^{(1)}$, $S^{(2)}$, $S^{(3)}$ arise from the commutative square relations (see Figure~\ref{fig:commutative-square-relations}) and $S^{(4)}$, $S^{(5)}$, $S^{(6)}$ arise from the monomial relations (see Figure~\ref{fig:monomial-relations}).
\end{proposition}

\begin{proof}
	According to the algorithm in \cite[Heuristic 3.13]{Barmeier-Wang2020}, this follows directly from the relations described above.  
	The set of $1$-ambiguities is
	\begin{align*}
		S_3 &=\bigl\{ u_{1,i}x_{i,n_1+1}y_{i,n_1+2} \mid 2\leq i \leq n_1 \bigr\}\\
		&\cup \bigl\{  u_{1,i}x_{n_1+1,i}y_{n_1+1,i+1} \mid  n_1+2 \leq i\leq n_1+n_2-1  \bigr\}.
	\end{align*}
	Moreover, there are no higher ambiguities. Indeed, every path in \(S=S_2\subseteq Q_2\) has length two and none of them starts with a vertical arrow \(y_{i,j}\), whereas every path in \(S_3\) listed above ends with such a vertical arrow. Hence no element of \(S\) can overlap with the right end of a path in \(S_3\), and therefore no higher overlap ambiguity can occur.
	
	Then we verify that every element of $S_3$ is resolvable.
	
	Case 1: $2\leq i \le n_1-1$.  
	On one side, $u_{1,i}x_{i,n_1+1}y_{i,n_1+2}$ reduces to $0$ because $u_{1,i}x_{i,n_1+1}=0$.  
	On the other side, we have the reduction (where $\to$ denotes the reduction step)
	\[
	u_{1,i}x_{i,n_1+1}y_{i,n_1+2} \;\to\; u_{1,i}y_{i,n_1+1}x_{i+1,n_1+1} \;\to\; x_{1,i}u_{1,i+1}x_{i+1,n_1+1} \;\to\; 0.
	\]
	
	Case 2: $n_1+3 \le i \le n_1+n_2-1$.  
	The two reduction paths are illustrated below:
	\[
	\renewcommand\arraystretch{0.4}
	\setlength\arraycolsep{3pt}
	\begin{array}{c c c c c}
		& & x_{1,i}u_{1,i+1}y_{n_1+1,i+1} & \to & 0 \\
		& \rotatebox{30}{$\to$} & & & \\
		u_{1,i}x_{n_1+1,i}y_{n_1+1,i+1} & & & & \\
		& \rotatebox{-30}{$\to$} & & & \\
		& & u_{1,i}y_{n_1+1,i}x_{n_1+2,i} & \to & 0
	\end{array}.
	\]
	
	Remaining cases: The element $u_{1,n_1}x_{n_1,n_1+1}y_{n_1,n_1+2}$ reduces to $0$ immediately. Indeed, it contains both the zero subpath $u_{1,n_1}x_{n_1,n_1+1} = 0$ and the subpath $x_{n_1,n_1+1}y_{n_1,n_1+2} = 0$. For $u_{1,n_1+2}x_{n_1+1,n_1+2}y_{n_1+1,n_1+3}$, one reduction path replaces $u_{1,n_1+2}x_{n_1+1,n_1+2}$ by $x_{1,n_1+2}u_{1,n_1+3}$, yielding $u_{1,n_1+3}y_{n_1+1,n_1+3} = 0$; meanwhile, the alternative reduction directly applies the zero relation $x_{n_1+1,n_1+2}y_{n_1+1,n_1+3} = 0$. Thus all ambiguities are resolvable.
\end{proof}

\begin{remark}\label{rem:reduction system}
	Since every path in $S_3$ ends with a vertical arrow ($y_{i,n_1+2}$ or $y_{n_1+1,i+1}$) 
	and no element of $S$ starts with such an arrow, no higher ambiguities can arise. Thus \(S_i=0\) for \(i\ge4\), and the projective resolution stops at \(P_3\). As a result, it follows from Theorem \ref{theorem:projective} that $\mathcal A_{n_1,n_2}$ admits a bimodule projective resolution 
	\[
	0\to P_3 \xrightarrow{\partial_3} P_2 \xrightarrow{\partial_2} P_1 \xrightarrow{\partial_1} P_0 \to 0
	\]
	where $P_i = \mathcal A_{n_1,n_2} \otimes_{\Bbbk Q_0} \Bbbk S_i \otimes_{\Bbbk Q_0} \mathcal A_{n_1,n_2}$.
\end{remark}
Using this reduction system, $\operatorname{HH}^*(\mathcal{A}_{n_1,n_2})$ can be computed via the complex concentrated in degrees $0$, $1$, $2$, and $3$ (for brevity we omit the subscript of $\operatorname{Hom}$):
\[
0 \to \operatorname{Hom}(\Bbbk Q_0, \mathcal{A}_{n_1,n_2}) \xrightarrow{\partial^0} 
\operatorname{Hom}(\Bbbk Q_1, \mathcal{A}_{n_1,n_2}) \xrightarrow{\partial^1} 
\operatorname{Hom}(\Bbbk S_2, \mathcal{A}_{n_1,n_2}) \xrightarrow{\partial^2} 
\operatorname{Hom}(\Bbbk S_3, \mathcal{A}_{n_1,n_2}) \to 0.
\]
Let us consider the following two subsets of $\operatorname{Hom}(\Bbbk S_2, \mathcal{A}_{n_1,n_2}) $:
\begin{align}
	\begin{aligned}
		T_1 &= \{(x_{i,j}y_{i,j+1} \parallel y_{i,j}x_{i+1,j}) \mid 1 \le i<j-1 \le n_1+n_2-2\}, \label{eq:T1}\\
		T_2 &= \{(u_{1,j}x_{n_1+1,j} \parallel x_{1,j}u_{1,j+1}) \mid n_1+2 \le j \le n_1+n_2-1\}. 
	\end{aligned}
\end{align}
The set $T_1$ consists of the basis elements corresponding to the commutative square relations of type $1$, while $T_2$ corresponds to the commutative square relations of type $2$ (see Proposition~\ref{prop:cohomologysymsqu} and Figure~\ref{fig:commutative-square-relations}). 

We will extend $T_1 \cup T_2$ into a basis of $\operatorname{Hom}(\Bbbk S_2, \mathcal{A}_{n_1,n_2})$ in the proofs of Lemmas \ref{lemma: HH of tilde A_{1,n_2}}, \ref{lemma:HH-A2-n2}, and \ref{lemma: HH of tilde A_{n1,n2}}. The following lemma will be highly useful for later computations, as it establishes that $\partial^1$ maps onto the span of $T_1\cup T_2$.

\begin{lemma}\label{lemma:T1T2_in_image}
	Let $T_1$ and $T_2$ be the subsets of $\operatorname{Hom}(\Bbbk S_2, \mathcal{A}_{n_1,n_2})$ defined in \eqref{eq:T1}. Then $T_1\cup T_2 \subseteq \operatorname{Im}(\partial^1)$.
\end{lemma}

\begin{proof} 
	We first prove that $T_1\subseteq \operatorname{Im}(\partial^1)$. The set $T_1$ consists of the basis elements corresponding to the commutative square relations of type $1$. For $1\leq i<j-1 \leq n_1+n_2-2$, we denote 
	\[
	\sigma_{i,j} := (x_{i,j}y_{i,j+1} \parallel y_{i,j}x_{i+1,j}) \in T_1,
	\]
	which maps the upper-right path $x_{i,j}y_{i,j+1}$ to the lower-left path $y_{i,j}x_{i+1,j}$. To show that $\sigma_{i,j} \in \operatorname{Im}(\partial^1)$, we evaluate the differential $\partial^1$ on the vertical basis elements $(y_{i,k} \parallel y_{i,k})$. Based on the column index $j$, we consider the following two cases separately:
	
	\begin{itemize}[leftmargin=*, label=--]
		\item If $j \ge n_1+1$, using the definition of the differential in \eqref{eq:partial1}, we evaluate $\partial^1$ on the vertical basis elements. For any $k$ with $i+2 \le k \le n_1+n_2$, we have:
		\begin{equation*}
			\partial^1(y_{i,k} \parallel y_{i,k}) = 
			\begin{cases}
				\sigma_{i,k-1} - \sigma_{i, k}, & \text{if } i+2 \le k \le n_1+n_2-1,\\[4pt]
				\sigma_{i,n_1+n_2-1}, & \text{if } k = n_1+n_2.
			\end{cases}
		\end{equation*}  
		Summing these identities from $k=j+1$ to $n_1+n_2$, the intermediate terms cancel out, yielding (see Figure~\ref{first row comm square}):
		\[
		\partial^1 \left(\sum\limits_{k=j+1}^{n_1+n_2} (y_{i,k} \parallel y_{i,k})\right) = (\sigma_{i,j}-\sigma_{i,j+1}) + \dots + \sigma_{i,n_1+n_2-1} = \sigma_{i,j}.
		\]
		This confirms that each basis element $\sigma_{i,j}$ for $j \ge n_1+1$ lies in the image of $\partial^1$.
		
		\begin{figure}[htbp]
			\centering
			\begin{tikzcd}[column sep=1.3cm, row sep=normal]
				L_{i,j} \ar[r, "x_{i,j}"] \ar[d, "y_{i,j}"] &
				L_{i,j+1} \ar[d, "y_{i,j+1}"] \ar[dl, phantom, "\color{blue}\sigma_{i,j}"'] \ar[r, "x_{i,j+1}"] &
				L_{i,j+2} \ar[dl, phantom, "\color{blue}\sigma_{i,j+1}"'] \ar[d, "y_{i,j+2}"] \ar[r] &
				\cdots \ar[r] &
				L_{i,n_1+n_2} \ar[d, "y_{i,n_1+n_2}"] \ar[dl, phantom, "\color{blue}\sigma_{i,n_1+n_2-1}"'] \\
				L_{i+1,j}  \ar[r, "x_{i+1,j}"] &
				L_{i+1,j+1} \ar[r, "x_{i+1,j+1}"] &
				L_{i+1,j+2} \ar[r] &
				\cdots \ar[r] &
				L_{i+1,n_1+n_2}
			\end{tikzcd}
			\caption{Cancellation of intermediate terms for commutative squares in $T_1$.}
			\label{first row comm square}
		\end{figure}
		
		\item If $j \le n_1$, a similar cancellation occurs. By summing over the vertical arrows from $k=i+2$ to $j$, we obtain:
		\[
		\partial^1\left(\sum\limits_{k=i+2}^{j} (y_{i,k} \parallel y_{i,k})\right) = -(x_{i,j}y_{i,j+1} \parallel y_{i,j}x_{i+1,j}) = -\sigma_{i,j}.
		\]
		This implies that $-\sigma_{i,j}$ (and hence $\sigma_{i,j}$) is in the image of $\partial^1$.
	\end{itemize}
	In both cases, we conclude that each basis element $\sigma_{i,j}$ lies in the image of $\partial^1$, which yields $T_1 \subseteq \operatorname{Im}(\partial^1)$.
	
	Next, we show that $T_2 \subseteq \operatorname{Im}(\partial^1)$. For a given basis element $(u_{1,j}x_{n_1+1,j} \parallel x_{1,j}u_{1,j+1}) \in T_2$, we similarly evaluate $\partial^1$ on the blue arrows $u_{1,k}$. Summing these equations, the intermediate terms cancel out, yielding:
	\[
	\partial^1\left(\sum\limits_{k=n_1+2}^{j} (u_{1,k} \parallel u_{1,k})\right) = (u_{1,j}x_{n_1+1,j} \parallel x_{1,j}u_{1,j+1}).
	\]
	This confirms that $T_2 \subseteq \operatorname{Im}(\partial^1)$, completing the proof.
\end{proof}

\subsubsection{Case: $n_1=1$ and $n_2\geq 2$}

Let us first consider the case  $n_1=1$ and $n_2=2$. Recall the gentle algebra $G_{1,2}$ corresponding to the surface shown in Figure~\ref{gentle algebra A(1,2)} is subject to the relation $\beta\alpha_2=0$.  
\begin{figure}[htbp]
	\centering
	\captionsetup{font=small}
	\begin{minipage}[b]{0.53\textwidth}
		\centering
		\begin{tikzcd}[column sep=normal, row sep=normal]
			& L_1 \ar[r,"\beta" swap] \ar[r, bend left=30, "\alpha_1"]
			& L_{2} \ar[r,"\alpha_{2}"] 
			& L_3
		\end{tikzcd}
		\captionof{figure}{Quiver of $G_{1,2}$.}
		\label{gentle algebra A(1,2)}
	\end{minipage}\hfill
	\begin{minipage}[b]{0.47\textwidth}
		\centering
		\begin{tikzcd}[column sep=large, row sep=normal]
			L_{1,2} \ar[r, "x_{1,2}"] & 
			L_{1,3} \ar[d, "y_{1,3}"] \ar[blue, d, bend right=30, "u_{1,3}" swap] &  \\
			& L_{2,3}
		\end{tikzcd}
		\captionof{figure}{Quiver of $A_{1,2}$.}
		\label{Quiver of $A_{1,2}$}
	\end{minipage}
\end{figure}

\begin{example}\label{ex: n_1=1,n_2=2}
	For $n_1=1$, $n_2=2$ the symmetric square algebra $\mathcal A_{1,2}$ is given by the quiver  $Q_{1,2}$ depicted in Figure~\ref{Quiver of $A_{1,2}$} with the relation $ x_{1,2} y_{1,3}=0.$ In particular, $\mathcal A_{1,2}$ is a gentle algebra.  Then by \cite{Chaparro-Schroll-SuarezAlvarez-Solotar2023} we have that $\operatorname{HH}^*(\mathcal{A}_{1,2})$ is concentrated in degrees $0$ and $1$, with
	\[
	\dim_{\Bbbk} \operatorname{HH}^0(\mathcal{A}_{1,2}) = 1,\qquad 
	\dim_{\Bbbk} \operatorname{HH}^1(\mathcal{A}_{1,2}) = 2.
	\]
	This can be computed by the following complex
	\begin{equation*}
		0 \to \Hom(\Bbbk Q_0, \mathcal{A}_{1,2}) \xrightarrow{\partial^0} \Hom(\Bbbk Q_1, \mathcal{A}_{1,2}) \xrightarrow{\partial^1} \Hom(\Bbbk S, \mathcal{A}_{1,2}) \to 0
	\end{equation*}
	concentrated in degrees $0$, $1$, and $2$. 
	Explicit cocycles representing the nonzero classes are:
	\begin{itemize}[leftmargin=*, label=--]
		\item In $\operatorname{HH}^0(\mathcal{A}_{1, 2})$: $(e_{1,2}\parallel e_{1,2})+(e_{1,3}\parallel e_{1,3})+(e_{2,3}\parallel e_{2,3})$,
		\item In $\operatorname{HH}^1(\mathcal{A}_{1, 2})$: $(u_{1,3} \parallel u_{1,3})$ and $(u_{1, 3} \parallel y_{1, 3})$.
	\end{itemize}
\end{example}

When $n_2>2$ the algebra becomes more involved, but its Hochschild cohomology admits a clean description, as we show next.

\begin{lemma}\label{lemma: HH of tilde A_{1,n_2}}
	Let $n_1 = 1$ and $n_2 > 2$. Then $\operatorname{HH}^*(\mathcal{A}_{1,n_2})$ is concentrated in degrees $0$ and $1$, with
	\[
	\dim_{\Bbbk} \operatorname{HH}^0(\mathcal{A}_{1, n_2}) = 1,\qquad 
	\dim_{\Bbbk} \operatorname{HH}^1(\mathcal{A}_{1, n_2}) = 1.
	\]
\end{lemma}

\begin{proof}
	The gentle algebra $G_{1, n_2}$  is presented by the quiver in Figure~\ref{gentle algebra A(n_1,1)} with the relation $\beta\alpha_2=0$.
	\begin{figure}[htbp]
		\centering
		\begin{tikzcd}[column sep=normal, row sep=normal]
			& L_1\arrow[r,"\beta" swap ] \ar[ r, bend left=30, "\alpha_1" ]
			& L_{2}\arrow[r,"\alpha_{2}"] 
			& \cdots \arrow[r,"\alpha_{n_2-1}"] 
			& L_{n_2} \arrow[r,"\alpha_{n_2}"] 
			& L_{n_2+1}
		\end{tikzcd}
		\caption{Quiver of gentle algebra $G_{1, n_2}$}
		\label{gentle algebra A(n_1,1)}
	\end{figure}
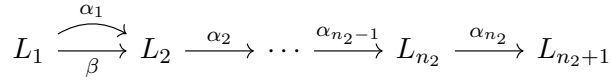
	
	The quiver $Q_{1,n_2}$ of the symmetric square algebra $\mathcal A_{1, n_2}$ is depicted in Figure~\ref{Quiver of $A_{1,n_2}$}. 
	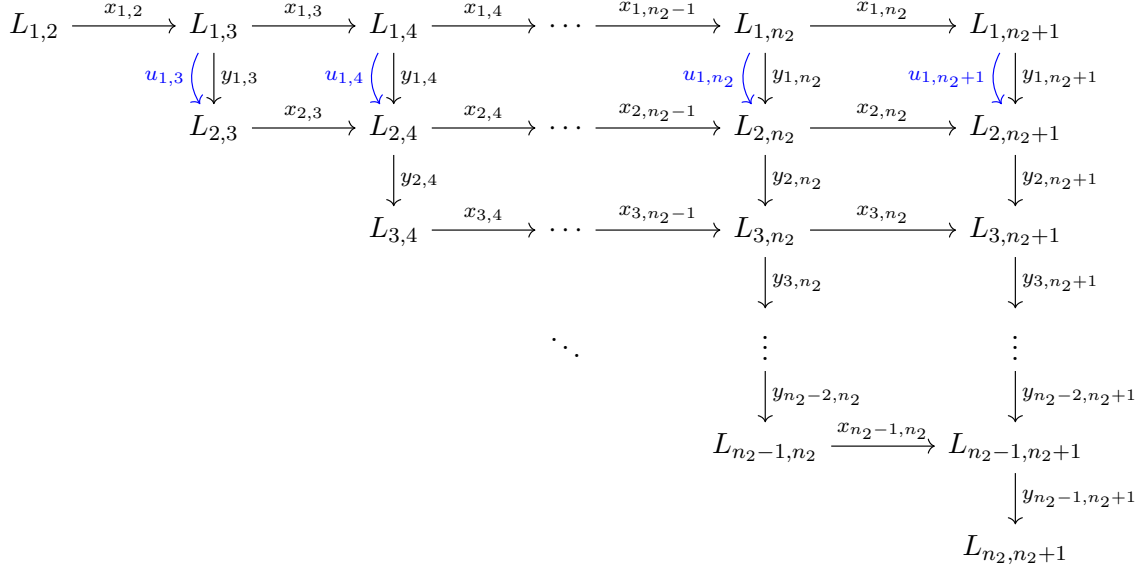
\begin{figure}[htbp]
		\centering
		\begin{tikzcd}[column sep=large, row sep=normal]
			L_{1,2} \ar[r, "x_{1,2}"] & 
			L_{1,3} \ar[r, "x_{1,3}"] \ar[d, "y_{1,3}"] \ar[blue, d, bend right=30, "u_{1,3}" swap] & 
			L_{1,4} \ar[r, "x_{1,4}"] \ar[d, "y_{1,4}"] \ar[blue, d, bend right=30, "u_{1,4}" swap] & 
			\cdots \ar[r, "x_{1,n_2-1}"] & 
			L_{1,n_2} \ar[r, "x_{1,n_2}"] \ar[d, "y_{1,n_2}"] \ar[blue, d, bend right=30, "u_{1,n_2}" swap] & 
			L_{1,n_2+1} \ar[d, "y_{1,n_2+1}"] \ar[blue, d, bend right=30, "u_{1,n_2+1}" swap] \\
			& 
			L_{2,3} \ar[r, "x_{2,3}"] & 
			L_{2,4} \ar[r, "x_{2,4}"] \ar[d, "y_{2,4}"] & 
			\cdots \ar[r, "x_{2,n_2-1}"] & 
			L_{2,n_2} \ar[r, "x_{2,n_2}"] \ar[d, "y_{2,n_2}"] & 
			L_{2,n_2+1} \ar[d, "y_{2,n_2+1}"] \\
			&& 
			L_{3,4} \ar[r, "x_{3,4}"] & 
			\cdots \ar[r, "x_{3,n_2-1}"] & 
			L_{3,n_2} \ar[r, "x_{3,n_2}"] \ar[d, "y_{3,n_2}"] & 
			L_{3,n_2+1} \ar[d, "y_{3,n_2+1}"] \\
			&&& 
			\ddots & 
			\vdots \ar[d, "y_{n_2-2,n_2}"] & 
			\vdots \ar[d, "y_{n_2-2,n_2+1}"] \\
			&&&& 
			L_{n_2-1,n_2} \ar[r, "x_{n_2-1,n_2}"] & 
			L_{n_2-1,n_2+1} \ar[d, "y_{n_2-1,n_2+1}"] \\
			&&&&& 
			L_{n_2,n_2+1}
		\end{tikzcd}
		\caption{Quiver of $\mathcal{A}_{1, n_2}$.}
		\label{Quiver of $A_{1,n_2}$}
	\end{figure}
	By Remark~\ref{rem:reduction system}, the Hochschild cohomology of $\mathcal{A}_{1, n_2}$ can be computed via the following complex,
	\[
	0 \to \operatorname{Hom}(\Bbbk Q_0, \mathcal{A}_{1, n_2}) \xrightarrow{\partial^0} 
	\operatorname{Hom}(\Bbbk Q_1, \mathcal{A}_{1, n_2}) \xrightarrow{\partial^1} 
	\operatorname{Hom}(\Bbbk S_2, \mathcal{A}_{1, n_2}) \xrightarrow{\partial^2} 
	\operatorname{Hom}(\Bbbk S_3, \mathcal{A}_{1, n_2}) \to 0,
	\]
	which is concentrated in degrees $0$, $1$, $2$, and $3$.
	It follows from Lemma~\ref{lemma:HH0-dimension} that  
	\[
	\dim_{\Bbbk} \operatorname{HH}^0(\mathcal{A}_{1,n_2}) = 1.
	\]
	
	We next give bases for the vector spaces appearing in the above complex.
	\begin{enumerate}[
		label=(\alph*),
		leftmargin=*,
		itemindent=0pt,
		labelindent=0pt,
		labelsep=0.5em
		]
		\item A basis of $\operatorname{Hom}(\Bbbk Q_0,  \mathcal{A}_{1, n_2})$ is given by
		\(
		\{(e_{i,j} \parallel e_{i,j}) \mid 1 \leq i < j \leq n_2+1\},
		\)
		where $e_{i,j}$ denotes the idempotent at the vertex $L_{i,j}$. 
		
		\item A basis of $\operatorname{Hom}(\Bbbk Q_1,  \mathcal{A}_{1, n_2})$ is $B = B_1 \cup B_2 \cup B_3$, where
		\begin{align*}
			B_1 &= \{(\alpha \parallel \alpha) \mid \alpha \in Q_1\},\\
			B_2 &= \{(y_{1,j} \parallel u_{1,j}) \mid 3 \le j \le n_2+1\},\\
			B_3 &= \{(u_{1,j} \parallel y_{1,j}) \mid 3 \le j \le n_2+1\}.
		\end{align*}
		
		\item A basis of $\operatorname{Hom}(\Bbbk S_2,  \mathcal{A}_{1, n_2})$ is $T = T_1 \cup T_2 \cup T_3 \cup T_4 \cup T_5$, where
		\allowdisplaybreaks
		\begin{align*}
			T_1 &= \{(x_{i,j}y_{i,j+1} \parallel y_{i,j}x_{i+1,j}) \mid 1\leq i<j-1 \leq n_2-1\},\\
			T_2 &= \{(u_{1,j}x_{2,j} \parallel x_{1,j}u_{1,j+1}) \mid 3 \le j \le n_2\},\\
			T_3 &= \{(x_{1,j}y_{1,j+1} \parallel x_{1,j}u_{1,j+1}) \mid 2 \le j \le n_2\},\\
			T_4 &= \{(u_{1,j}x_{2,j} \parallel y_{1,j}x_{2,j}) \mid 3 \le j \le n_2\},\\
			T_5 &= \{(u_{1,j}y_{2,j} \parallel y_{1,j}y_{2,j}) \mid 4 \le j \le n_2+1\}.
		\end{align*}
		Note that $T_1$ and $T_2$ coincide with the ones in \eqref{eq:T1}.
		\item A basis of $\operatorname{Hom}(\Bbbk S_3,  \mathcal{A}_{1, n_2})$ is
		\[
		\{(u_{1,j}x_{2,j}y_{2,j+1} \parallel y_{1,j}y_{2,j}x_{3,j+1}) \mid 4 \le j \le n_2\}.
		\]
		Each such basis element corresponds to the unique non‑zero parallel path from $L_{1,j}$ to $L_{3,j+1}$. Note that this set is empty when $n_2=3$.
	\end{enumerate}
	Consequently, we have
	
	\allowdisplaybreaks
	\begin{align*}
		\dim_{\Bbbk} \operatorname{Hom}(\Bbbk Q_0,  \mathcal{A}_{1, n_2}) &= \frac{(n_2+1)n_2}{2},\\
		\dim_{\Bbbk} \operatorname{Hom}(\Bbbk Q_1,  \mathcal{A}_{1, n_2}) &= n_2^2+2n_2-3,\\
		\dim_{\Bbbk} \operatorname{Hom}(\Bbbk S_2,  \mathcal{A}_{1, n_2})  &= \frac{n_2^2 + 5n_2 - 12}{2},\\
		\dim_{\Bbbk} \operatorname{Hom}(\Bbbk S_3, \mathcal{A}_{1,n_2}) &= n_2-3.
	\end{align*}
	
	We first show that $\partial^2$ is surjective. For this, consider the image of $T_4$ under $\partial^2$. 
	Note that for $j=3$, the basis element $(u_{1,3}x_{2,3} \parallel y_{1,3}x_{2,3}) \in T_4$ is mapped to $0$ by $\partial^2$, since $S_3$ contains no path starting at $L_{1,3}$.
	For each $4 \le j \le n_2$, we have
	\[
	\partial^2(u_{1,j}x_{2,j} \parallel y_{1,j}x_{2,j})(u_{1,k}x_{2,k}y_{2,k+1}) = \begin{cases}
		y_{1,j}x_{2,j}y_{2,j+1} = y_{1,j}y_{2,j}x_{3,j+1}, & \text{if $k=j$,}\\
		0, & \text{if $k\neq j$}.
	\end{cases}
	\]
	This yields 
	\[
	\partial^2(u_{1,j}x_{2,j} \parallel y_{1,j}x_{2,j}) = (u_{1,j}x_{2,j}y_{2,j+1} \parallel y_{1,j}y_{2,j}x_{3,j+1}) \qquad (4 \le j \le n_2),
	\]
	proving the surjectivity of $\partial^2$ and thus  $$\dim_\Bbbk \ker(\partial^2) = \frac{n_2^2 + 5n_2 - 12}{2} - (n_2-3) = \frac{n_2^2 + 3n_2 - 6}{2}. $$
	
	Next we claim that
	\begin{align*}
		\dim_{\Bbbk} \operatorname{Im}(\partial^1)  = \frac{n_2^2 + 3n_2 - 6}{2}\quad \text{and} \quad
		\dim_{\Bbbk} \ker(\partial^1)  = \frac{n_2(n_2+1)}{2}.
	\end{align*} 
	
	To justify this, we examine the differential $\partial^1$.
	\begin{enumerate}[leftmargin=*, itemindent=0pt, labelindent=0pt, labelsep=0.5em]
		\item\label{item:first} 
		By Lemma~\ref{lemma:T1T2_in_image}, it is immediate that $T_1\cup T_2 \subseteq \operatorname{Im}(\partial^1)$.
		
		\item\label{item:second} We show that $T_3 \subseteq \operatorname{Im}(\partial^1)$. To do this, we evaluate $\partial^1$ on the basis $B_2$ and show that $\partial^1(B_2)$ spans $T_3$.
		Recall that
		\[
		T_3 = \{(x_{1,j}y_{1,j+1} \parallel x_{1,j}u_{1,j+1}) \mid 2 \le j \le n_2\}.
		\]
		A direct computation yields the following identities: 
		\[
		\partial^1(y_{1,j} \parallel u_{1,j}) = 
		\begin{cases} 
			(x_{1,j-1}y_{1,j} \parallel x_{1,j-1}u_{1,j}) - (x_{1,j}y_{1,j+1} \parallel x_{1,j}u_{1,j+1}), &\text{if $3 \le j \le n_2$},\\[4pt]
			(x_{1,n_2}y_{1,n_2+1} \parallel x_{1,n_2}u_{1,n_2+1}), & \text{if $j=n_2+1$}.
		\end{cases}
		\]
		By taking a sum of these equalities from $k=j+1$ to $n_2+1$, we obtain for each $2 \le j \le n_2$:
		\[
		\partial^1\left(\sum\limits_{k=j+1}^{n_2+1} (y_{1,k} \parallel u_{1,k})\right) = (x_{1,j}y_{1,j+1} \parallel x_{1,j}u_{1,j+1}).
		\]
		Hence $T_3 \subseteq \operatorname{Im}(\partial^1)$.
		
		\item\label{item:third} Finally, we consider $T_4\cup T_5$. We focus on the image of $B_3$.
		Recall $B_3 = \{(u_{1,j} \parallel y_{1,j}) \mid 3 \le j \le n_2+1\}$. A straightforward calculation shows that
		\[
		\partial^1(u_{1,3} \parallel y_{1,3}) = (u_{1,3}x_{2,3} \parallel y_{1,3}x_{2,3}),
		\]
		and for $4 \le j \le n_2$,
		\[
		\partial^1(u_{1,j} \parallel y_{1,j}) = (u_{1,j}y_{2,j} \parallel y_{1,j}y_{2,j}) - (u_{1,j-1}x_{2,j-1} \parallel y_{1,j-1}x_{2,j-1}) + (u_{1,j}x_{2,j} \parallel y_{1,j}x_{2,j}),
		\]
		while for $j = n_2+1$,
		\[
		\partial^1(u_{1,n_2+1} \parallel y_{1,n_2+1}) = (u_{1,n_2+1}y_{2,n_2+1} \parallel y_{1,n_2+1}y_{2,n_2+1}) - (u_{1,n_2}x_{2,n_2} \parallel y_{1,n_2}x_{2,n_2}).
		\]
		From the above computation, it follows that all these images lie in $\operatorname{Span}_\Bbbk(T_4 \cup T_5)$. We also observe  that the $n_2-1$ vectors $\partial^1(u_{1,j} \parallel y_{1,j})$ ($3 \le j \le n_2+1$) are linearly independent.  
		Indeed, for each $j \ge 3$, the basis element $(u_{1,j}y_{2,j} \parallel y_{1,j}y_{2,j})$ in $T_5$ appears in $\partial^1(u_{1,j} \parallel y_{1,j})$ as a summand  with coefficient $1$, while it does not appear in $\partial^1(u_{1,k} \parallel y_{1,k})$ for any $3\leq k < j$.  Hence the vectors $\{\partial^1(u_{1,j} \parallel y_{1,j}) \mid  3 \le j \le n_2+1 \}$ are linearly independent so that
		\(
		\dim_\Bbbk \partial^1(B_3) = n_2 - 1.
		\)
	\end{enumerate}
	
	From the computations in (\ref{item:first}), (\ref{item:second}), and (\ref{item:third}), we see that $\partial^1$ maps $B_1 \cup B_2$ onto $\operatorname{Span}_\Bbbk(T_1 \cup T_2 \cup T_3)$ and the two blocks do not mix.  Consequently, $\partial^1$ has a block diagonal structure with respect to the decompositions:
	\[
	\begin{array}{c|cc}
		\partial^1 & T_1 \cup T_2 \cup T_3 & T_4 \cup T_5 \\
		\hline
		B_1 \cup B_2 & * & 0 \\
		B_3          & 0 & *
	\end{array}
	\]
	
	We conclude that
	\[
	\dim_\Bbbk \operatorname{Im}(\partial^1) = |T_1 \cup T_2 \cup T_3| + (n_2 - 1)= \frac{n_2^2 + 3n_2 - 6}{2},
	\]
	and consequently
	\[
	\dim_{\Bbbk} \ker(\partial^1) = (n_2^2+2n_2-3) - \frac{n_2^2 + 3n_2 - 6}{2} = \frac{n_2(n_2+1)}{2}.
	\]
	
	From the connectedness of $Q_{1,n_2}$ we have $\dim_{\Bbbk} \operatorname{Im}(\partial^0) = \frac{n_2(n_2+1)}{2} - 1$. Therefore,
	\[
	\begin{aligned}
		\dim_{\Bbbk} \operatorname{HH}^0( \mathcal{A}_{1, n_2}) &= 1,\\
		\dim_{\Bbbk} \operatorname{HH}^1( \mathcal{A}_{1, n_2}) &= \dim_{\Bbbk} \ker(\partial^1) - \dim_{\Bbbk} \operatorname{Im}(\partial^0) = 1,\\
		\dim_{\Bbbk} \operatorname{HH}^2( \mathcal{A}_{1, n_2}) &= \dim_{\Bbbk} \ker(\partial^2) - \dim_{\Bbbk} \operatorname{Im}(\partial^1) = 0,\\
		\dim_{\Bbbk} \operatorname{HH}^3( \mathcal{A}_{1, n_2}) &= 0.
	\end{aligned}
	\]
	A nontrivial cocycle in $\operatorname{HH}^1( \mathcal{A}_{1, n_2})$ can be taken as $\sum\limits_{i=3}^{n_2+1}(u_{1,i} \parallel u_{1,i})$.
\end{proof}

\subsubsection{Case: $n_1=2$ and $n_2\geq 2$}
Let us consider the case where $n_1=2$ and thus $n_2\geq 2.$

\begin{lemma}\label{lemma:HH-A2-n2}
	Let $n_1 = 2$ and $n_2 \geq 2$. Then the Hochschild cohomology of the symmetric square algebra $\mathcal{A}_{2, n_2}$ is concentrated in degrees $0$, $1$, $2$, $3$, and its dimensions are given as follows:
	\begin{itemize}[leftmargin=*, label=--]
		\item If $n_2 = 2$, then
		\[
		\dim_{\Bbbk} \operatorname{HH}^0 = 1,\quad
		\dim_{\Bbbk} \operatorname{HH}^1 = 3,\quad
		\dim_{\Bbbk} \operatorname{HH}^2 = 1,\quad
		\dim_{\Bbbk} \operatorname{HH}^3 = 1.
		\]
		\item If $n_2\geq 3$, then
		\[
		\dim_{\Bbbk} \operatorname{HH}^0 = 1,\quad
		\dim_{\Bbbk} \operatorname{HH}^1 = 2,\quad
		\dim_{\Bbbk} \operatorname{HH}^2 = 1,\quad
		\dim_{\Bbbk} \operatorname{HH}^3 = 1.
		\]
	\end{itemize}
\end{lemma}

\begin{proof}
	In the case $n_1 = 2$ and $n_2 \geq 2$, the gentle algebra corresponding to the surface shown in Figure~\ref{fig:gentle-algebra-A2-n2} is given by the quiver
	\begin{figure}[htbp]
		\centering
		\begin{tikzcd}[column sep=normal, row sep=normal]
			& 
			& L_{2} \arrow[dr,"\alpha_{2}"] 
			&  \\
			& L_1\arrow[rr,"\beta"] \arrow[ur,"\alpha_1"]& 
			& L_{3}\arrow[r,"\alpha_{3}"] 
			& \cdots \arrow[r,"\alpha_{n_2}"] 
			& L_{n_2+1} \arrow[r,"\alpha_{n_2+1}"] 
			& L_{n_2+2}
		\end{tikzcd}
		\caption{Quiver of gentle algebra $G_{2, n_2}$.}
		\label{fig:gentle-algebra-A2-n2}
	\end{figure}
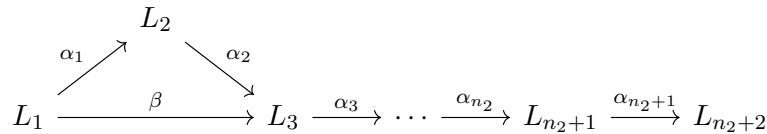
	subject to the relation $\beta\alpha_3=0$. 
	
	Assume first that $n_2\geq 3$, then the quiver $Q_{2,n_2}$ of the symmetric square algebra $\mathcal A_{2, n_2}$ is illustrated in Figure~\ref{fig:quiver-A2-n2}. 
	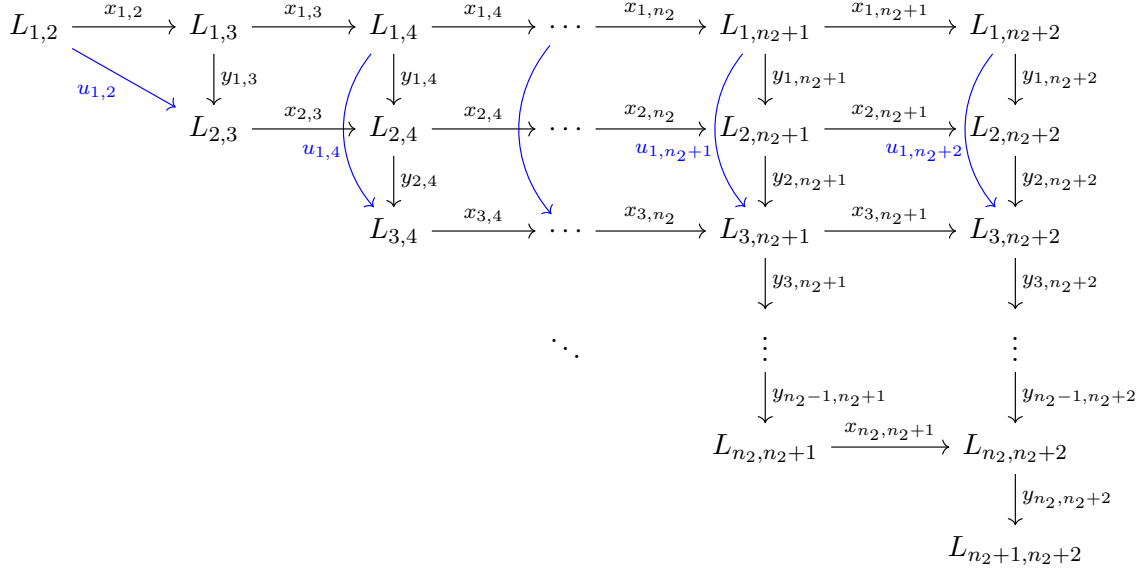
\begin{figure}[htbp]
		\centering
		\begin{tikzcd}[column sep=large, row sep=normal]
			L_{1,2} \ar[r, "x_{1,2}"]  \ar[blue, dr, "u_{1,2}", swap]& 
			L_{1,3} \ar[r, "x_{1,3}"] \ar[d, "y_{1,3}"]& 
			L_{1,4} \ar[r, "x_{1,4}"] \ar[d, "y_{1,4}"] \ar[blue, dd, bend right=40, "u_{1,4}" {xshift=2pt, yshift=-8pt}, swap] & 
			\cdots \ar[r, "x_{1,n_2}"] 
			\ar[blue, dd, bend right=40, swap] & 
			L_{1,n_2+1} \ar[r, "x_{1,n_2+1}"] \ar[d, "y_{1,n_2+1}"]  \ar[blue, dd, bend right=40, "u_{1,n_2+1}" {xshift=2pt, yshift=-8pt}, swap]& 
			L_{1,n_2+2} \ar[d, "y_{1,n_2+2}"]  \ar[blue, dd, bend right=40, "u_{1,n_2+2}" {xshift=2pt, yshift=-8pt}, swap]\\
			& 
			L_{2,3} \ar[r, "x_{2,3}"] & 
			L_{2,4} \ar[r, "x_{2,4}"] \ar[d, "y_{2,4}"] & 
			\cdots \ar[r, "x_{2,n_2}"] & 
			L_{2,n_2+1} \ar[r, "x_{2,n_2+1}"] \ar[d, "y_{2,n_2+1}"] & 
			L_{2,n_2+2} \ar[d, "y_{2,n_2+2}"] \\
			&& 
			L_{3,4} \ar[r, "x_{3,4}"] & 
			\cdots \ar[r, "x_{3,n_2}"] & 
			L_{3,n_2+1} \ar[r, "x_{3,n_2+1}"] \ar[d, "y_{3,n_2+1}"] & 
			L_{3,n_2+2} \ar[d, "y_{3,n_2+2}"] \\
			&&& 
			\ddots & 
			\vdots \ar[d, "y_{n_2-1,n_2+1}"] & 
			\vdots \ar[d, "y_{n_2-1,n_2+2}"] \\
			&&&& 
			L_{n_2,n_2+1} \ar[r, "x_{n_2,n_2+1}"] & 
			L_{n_2,n_2+2} \ar[d, "y_{n_2,n_2+2}"] \\
			&&&&& 
			L_{n_2+1,n_2+2}
		\end{tikzcd}
		\caption{Quiver of $\mathcal{A}_{2, n_2}$.}
		\label{fig:quiver-A2-n2}
	\end{figure}
	
	By Remark~\ref{rem:reduction system}, the Hochschild cohomology of $\mathcal A_{2, n_2}$ is computed by the complex concentrated in degrees $0$, $1$, $2$, and $3$:
	\[
	\scalebox{0.95}{$
		0 \to \operatorname{Hom}(\Bbbk Q_0, \mathcal A_{2, n_2}) \xrightarrow{\partial^0} 
		\operatorname{Hom}(\Bbbk Q_1, \mathcal A_{2, n_2}) \xrightarrow{\partial^1} 
		\operatorname{Hom}(\Bbbk S_2, \mathcal A_{2, n_2}) \xrightarrow{\partial^2} 
		\operatorname{Hom}(\Bbbk S_3,\mathcal A_{2, n_2}) \to 0.
		$}
	\]
	
	Assuming $n_2 \ge 3$, we now determine the bases and dimensions for each term in the complex above:
	\begin{enumerate}[label=(\alph*), leftmargin=*, itemindent=0pt, labelindent=0pt, labelsep=0.5em]
		\item A basis of $\operatorname{Hom}(\Bbbk Q_0, \mathcal{A}_{2, n_2})$ is
		\(
		\{(e_{i,j} \parallel e_{i,j}) \mid 1 \le i < j \le n_2+2\}.
		\)
		
		\item A basis of $\operatorname{Hom}(\Bbbk Q_1, \mathcal{A}_{2, n_2})$ is $B_1\cup B_2$, where
		\begin{align*}
			B_1&=\{(\alpha \parallel \alpha) \mid \alpha \in Q_1\},\\
			B_2&=\{(u_{1,j} \parallel y_{1,j}y_{2,j}) \mid 4 \le j \le n_2+2\}.
		\end{align*}
		
		\item A basis of $\operatorname{Hom}(\Bbbk S_2, \mathcal{A}_{2, n_2})$ is $T = T_1 \cup T_2 \cup T_3 \cup T_4 \cup T_5$, where
		\allowdisplaybreaks
		\begin{align*}
			T_1 &= \{(x_{i,j}y_{i,j+1} \parallel y_{i,j}x_{i+1,j}) \mid 1 \le i<j-1 \le n_2\},\\
			T_2 &= \{(u_{1,j}x_{3,j} \parallel x_{1,j}u_{1,j+1}) \mid 4 \le j \le n_2+1\},\\
			T_3 &= \{(u_{1,j}x_{3,j} \parallel y_{1,j}y_{2,j}x_{3,j}) \mid 4 \le j \le n_2+1\},\\
			T_4 &= \{(u_{1,j}y_{3,j} \parallel y_{1,j}y_{2,j}y_{3,j}) \mid 5 \le j \le n_2+2\},\\
			T_5 &= \{(x_{1,2}y_{1,3} \parallel u_{1,2})\}.
		\end{align*}
		
		\item A basis of $\operatorname{Hom}(\Bbbk S_3, \mathcal{A}_{2, n_2})$ is $N = N_1 \cup N_2$, where
		\begin{align*}
			N_1 &= \{(u_{1,j}x_{3,j}y_{3,j+1} \parallel y_{1,j}y_{2,j}y_{3,j}x_{4,j+1}) \mid 5 \le j \le n_2+1\},\\
			N_2 &= \{(u_{1,2}x_{2,3}y_{2,4} \parallel x_{1,2}x_{1,3}u_{1,4})\}.
		\end{align*}
	\end{enumerate}
	
	Consequently, we have
	\allowdisplaybreaks
	\begin{align*}
		\dim_{\Bbbk} \operatorname{Hom}(\Bbbk Q_0,  \mathcal{A}_{2, n_2}) &= \frac{(n_2+1)(n_2+2)}{2},\\
		\dim_{\Bbbk} \operatorname{Hom}(\Bbbk Q_1,  \mathcal{A}_{2, n_2}) &= n_2^2+3n_2-1,\\
		\dim_{\Bbbk} \operatorname{Hom}(\Bbbk S_2,  \mathcal{A}_{2, n_2}) &= \frac{n_2^2+5n_2-10}{2},\\
		\dim_{\Bbbk} \operatorname{Hom}(\Bbbk S_3, \mathcal{A}_{2,n_2}) &= n_2-2.
	\end{align*}
	
	We first show that $\dim_{\Bbbk} \operatorname{Im}(\partial^2) = n_2-3$. For $5 \le j \le n_2+1$, we have
	\[
	\partial^2(u_{1,j}x_{3,j} \parallel y_{1,j}y_{2,j}x_{3,j})(u_{1,j}x_{3,j}y_{3,j+1}) = y_{1,j}y_{2,j}y_{3,j}x_{4,j+1},
	\]
	and $\partial^2(u_{1,j}x_{3,j} \parallel y_{1,j}y_{2,j}x_{3,j})(u_{1,j'}x_{3,j'}y_{3,j'+1}) = 0$ for $j' \neq j$. 
	Thus
	\[
	\partial^2(u_{1,j}x_{3,j} \parallel y_{1,j}y_{2,j}x_{3,j}) = (u_{1,j}x_{3,j}y_{3,j+1} \parallel y_{1,j}y_{2,j}y_{3,j}x_{4,j+1}) \quad (5 \le j \le n_2+1),
	\]
	while $(u_{1,2}x_{2,3}y_{2,4} \parallel x_{1,2}x_{1,3}u_{1,4})$ is not in the image of $\partial^2$. Hence $\dim_{\Bbbk} \operatorname{Im}(\partial^2) = n_2-3$.
	
	Then we examine $\partial^1$ as follows:
	\begin{enumerate}[leftmargin=*, itemindent=0pt, labelindent=0pt, labelsep=0.5em]
		\item By Lemma~\ref{lemma:T1T2_in_image}, it is immediate that $T_1\cup T_2 \subseteq \operatorname{Im}(\partial^1)$. More precisely, the sum construction demonstrates that $\partial^1(B_1)$ surjects onto $\operatorname{Span}_\Bbbk(T_1 \cup T_2)$.
		
		\item Next, we evaluate the differential $\partial^1$ on the basis elements $(u_{1,j} \parallel y_{1,j}y_{2,j}) \in B_2$. We consider the following three cases based on the index $j$:
		
		\begin{itemize}[label=--]
			\item When $j = 4$, we have
			\[
			\partial^1(u_{1,4} \parallel y_{1,4}y_{2,4}) = (u_{1,4}x_{3,4} \parallel y_{1,4}y_{2,4}x_{3,4}).
			\]
			
			\item  When $5 \le j \le n_2+1$, the image consists of three terms:
			\[
			\begin{aligned}
				\partial^1(u_{1,j} \parallel y_{1,j}y_{2,j}) &= - (u_{1,j-1}x_{3,j-1} \parallel y_{1,j-1}y_{2,j-1}x_{3,j-1}) \\
				&\quad + (u_{1,j}y_{3,j} \parallel y_{1,j}y_{2,j}y_{3,j}) \\
				&\quad + (u_{1,j}x_{3,j} \parallel y_{1,j}y_{2,j}x_{3,j}).
			\end{aligned}
			\]
			
			\item  When $j = n_2+2$, we have
			\[
			\begin{aligned}
				\partial^1(u_{1,n_2+2} \parallel y_{1,n_2+2}y_{2,n_2+2}) &= - (u_{1,n_2+1}x_{3,n_2+1} \parallel y_{1,n_2+1}y_{2,n_2+1}x_{3,n_2+1}) \\
				&\quad + (u_{1,n_2+2}y_{3,n_2+2} \parallel y_{1,n_2+2}y_{2,n_2+2}y_{3,n_2+2}).
			\end{aligned}
			\]
		\end{itemize}
		
		Clearly, all these images lie in $\operatorname{Span}_\Bbbk(T_3 \cup T_4)$. We now show that the $n_2-1$ vectors $\partial^1(u_{1,j} \parallel y_{1,j}y_{2,j})$ for $4 \le j \le n_2+2$ are linearly independent. 
		
		Indeed, for $j=4$, $\partial^1(u_{1,4} \parallel y_{1,4}y_{2,4})$ equals the basis vector $(u_{1,4}x_{3,4} \parallel y_{1,4}y_{2,4}x_{3,4})$ in $T_3$. For each $5 \le j \le n_2+2$, the image $\partial^1(u_{1,j} \parallel y_{1,j}y_{2,j})$ contains the basis element $(u_{1,j}y_{3,j} \parallel y_{1,j}y_{2,j}y_{3,j}) \in T_4$ with coefficient $1$. Crucially, this specific basis element does not appear in the image $\partial^1(u_{1,k} \parallel y_{1,k}y_{2,k})$ for any $k < j$. Due to this triangular structure, no vector can be expressed as a linear combination of the preceding ones. Consequently, these $n_2-1$ vectors $\partial^1(u_{1,j} \parallel y_{1,j}y_{2,j}) $ for $4\leq j \leq n_2+2$ are linearly independent, and we obtain \(\dim_\Bbbk \partial^1(B_2) = n_2 - 1.\)
	\end{enumerate}
	
	Since we have shown that $\partial^1(B_1)$ strictly covers $\operatorname{Span}_\Bbbk(T_1 \cup T_2)$ and $\partial^1(B_2)$ lies entirely in $\operatorname{Span}_\Bbbk(T_3 \cup T_4)$, the basis element in $T_5 = \{(x_{1,2}y_{1,3} \parallel u_{1,2})\}$ cannot be in the image of $\partial^1$. Thus, we can represent $\partial^1$ by a block diagonal structure:
	\[
	\begin{array}{c|ccc}
		\partial^1 & T_1 \cup T_2 & T_3 \cup T_4 & T_5 \\
		\hline
		B_1 & * & 0 & 0 \\
		B_2 & 0 & * & 0
	\end{array}
	\]
	
	From this structure, the rank of $\partial^1$ is simply the sum of $|T_1 \cup T_2|$ and $\dim_\Bbbk \partial^1(B_2)$:
	\[
	\dim_\Bbbk \operatorname{Im}(\partial^1) = |T_1 \cup T_2| + (n_2 - 1) = \frac{n_2^2 + 3n_2 - 6}{2}.
	\]
	
	Using the fact that $\dim_\Bbbk \operatorname{Im}(\partial^0) = \dim_\Bbbk \operatorname{Hom}(\Bbbk Q_0, \mathcal{A}_{2, n_2}) - 1$ (due to the connectedness of the quiver $Q_{2,n_2}$), a straightforward calculation yields the following dimensions for $n_2 \ge 3$:
	\begin{align*}
		\dim_\Bbbk \operatorname{HH}^i(\mathcal{A}_{2, n_2}) = \begin{cases}
			1, & \text{if $i=0, 2, 3$,}\\
			2, & \text{if $i=1$.}
		\end{cases}
	\end{align*}
	
	Explicit cocycles representing the nonzero classes are:
	\begin{itemize}[leftmargin=*, label=--]
		\item In $\operatorname{HH}^0(\mathcal{A}_{2, n_2})$: $\sum\limits_{1\le i<j\le n_2+2} (e_{i,j} \parallel e_{i,j})$,
		\item In $\operatorname{HH}^1(\mathcal{A}_{2, n_2})$: $(u_{1,2} \parallel u_{1,2})$ and $\sum\limits_{i=4}^{n_2+2}(u_{1,i} \parallel u_{1,i})$,
		\item In $\operatorname{HH}^2(\mathcal{A}_{2, n_2})$: $(x_{1,2}y_{1,3} \parallel u_{1,2})$,
		\item In $\operatorname{HH}^3(\mathcal{A}_{2, n_2})$: $(u_{1,2}x_{2,3}y_{2,4} \parallel x_{1,2}x_{1,3}u_{1,4})$.
	\end{itemize}
	
	For the special case $n_1 = 2$ and $n_2 = 2$, a similar computation yields
	\[
	\dim_{\Bbbk} \operatorname{HH}^i(\mathcal{A}_{2,2}) =
	\begin{cases}
		1,& \text{if } i = 0,2,3,\\
		3,& \text{if } i = 1,\\
		0,& \text{otherwise}.
	\end{cases}
	\]
	Explicit cocycles representing the nonzero classes are:
	\begin{itemize}[leftmargin=*, label=--]
		\item In $\operatorname{HH}^0(\mathcal{A}_{2, 2})$: $\sum\limits_{1\le i<j\le 4} (e_{i,j} \parallel e_{i,j})$,
		\item In $\operatorname{HH}^1(\mathcal{A}_{2, 2})$: $(u_{1,2} \parallel u_{1,2})$, $(u_{1,4} \parallel u_{1,4})$ and $(u_{1,4} \parallel y_{1,4}y_{2,4})$,
		\item In $\operatorname{HH}^2(\mathcal{A}_{2, 2})$: $(x_{1,2}y_{1,3} \parallel u_{1,2})$,
		\item In $\operatorname{HH}^3(\mathcal{A}_{2, 2})$: $(u_{1,2}x_{2,3}y_{2,4} \parallel x_{1,2}x_{1,3}u_{1,4})$.
	\end{itemize}
\end{proof}
\begin{remark}
	We remark that $n_1 = 2$ is the unique case yielding a non-vanishing second Hochschild cohomology group (i.e., $\operatorname{HH}^2(\mathcal{A}_{2,n_2}) \neq 0$). This exceptional behavior arises because the specific quiver structure for $n_1 = 2$ admits a nontrivial cocycle $(x_{1,2}y_{1,3} \parallel u_{1,2})$ which cannot be resolved by the differential $\partial^1$. 
	
\end{remark}

\subsubsection{Case: $n_2\geq n_1\geq 3$}
Next, we consider the general case, namely  $n_2\geq n_1\geq 3$.
\begin{lemma}\label{lemma: HH of tilde A_{n1,n2}}
	Let $n_2\geq n_1\geq 3$. Then the Hochschild cohomology of the symmetric square algebra $\mathcal{A}_{n_1,n_2}$ is concentrated in degrees $0$, $1$, $3$, and we have
	\[
	\dim_{\Bbbk} \operatorname{HH}^0(\mathcal{A}_{n_1,n_2}) = 1,\quad
	\dim_{\Bbbk} \operatorname{HH}^1(\mathcal{A}_{n_1,n_2}) = 2,\quad
	\dim_{\Bbbk} \operatorname{HH}^3(\mathcal{A}_{n_1,n_2}) = 1,
	\]
	while $\operatorname{HH}^2(\mathcal{A}_{n_1,n_2}) = 0$.
\end{lemma}

\begin{proof}
	In this case, the gentle algebra $G_{n_1,n_2}$ arises from the surface shown in Figure \ref{standard form} with the single relation $\beta\alpha_{n_1+1}=0$. 
	Let $Q_{n_1,n_2}$ be the quiver of $\mathcal{A}_{n_1,n_2}$ (see Figure \ref{quiver A_{n1,n2}}). 
	By Remark~\ref{rem:reduction system}, its Hochschild cohomology can be computed via the complex concentrated in degrees $0$, $1$, $2$, and $3$:
	\[
	\scalebox{0.95}{$
		0 \to \operatorname{Hom}(\Bbbk Q_0, \mathcal{A}_{n_1,n_2}) \xrightarrow{\partial^0} 
		\operatorname{Hom}(\Bbbk Q_1, \mathcal{A}_{n_1,n_2}) \xrightarrow{\partial^1} 
		\operatorname{Hom}(\Bbbk S_2, \mathcal{A}_{n_1,n_2}) \xrightarrow{\partial^2} 
		\operatorname{Hom}(\Bbbk S_3, \mathcal{A}_{n_1,n_2}) \to 0.
		$}
	\]
	
	We first determine bases and dimensions for each term in the complex above.
	\allowdisplaybreaks
	\begin{enumerate}[label=(\alph*), leftmargin=*, itemindent=0pt, labelindent=0pt, labelsep=0.5em]
		\item A basis of $\operatorname{Hom}(\Bbbk Q_0, \mathcal{A}_{n_1,n_2})$ is
		\(
		\{(e_{i,j} \parallel e_{i,j}) \mid 1 \leq i < j \leq n_1+n_2\}.
		\)
		
		\item A basis of $\operatorname{Hom}(\Bbbk Q_1, \mathcal{A}_{n_1,n_2})$ is $B_1 \cup B_2$, where
		\begin{align*}
			B_1 &= \{(\alpha \parallel \alpha) \mid \alpha \in Q_1\}, \\
			B_2 &= \{(u_{1,j} \parallel y_{1,j}y_{2,j}\cdots y_{n_1,j}) \mid n_1+2 \le j \le n_1+n_2\}.
		\end{align*}
		
		\allowdisplaybreaks
		\item A basis of $\operatorname{Hom}(\Bbbk S_2, \mathcal{A}_{n_1,n_2})$ is $T = T_1 \cup T_2 \cup T_3 \cup T_4 \cup T_5$, where
		\begin{align*}
			T_1 &= \{(x_{i,j}y_{i,j+1} \parallel y_{i,j}x_{i+1,j}) \mid 1 \le i<j-1 \le n_1+n_2-2\},\\
			T_2 &= \{(u_{1,j}x_{n_1+1,j} \parallel x_{1,j}u_{1,j+1}) \mid n_1+2 \le j \le n_1+n_2-1\},\\
			T_3 &= \{(u_{1,j}y_{j,n_1+1} \parallel x_{1,j}u_{1,j+1}) \mid 2 \le j \le n_1-1\},\\
			T_4 &= \{(u_{1,j}x_{n_1+1,j} \parallel y_{1,j}y_{2,j}\cdots y_{n_1,j}x_{n_1+1,j}) \mid n_1+2 \le j \le n_1+n_2-1\},\\
			T_5 &= \{(u_{1,j}y_{n_1+1,j} \parallel y_{1,j}y_{2,j}\cdots y_{n_1,j}y_{n_1+1,j}) \mid n_1+3 \le j \le n_1+n_2\}.
		\end{align*}
		\item A basis of $\operatorname{Hom}(\Bbbk S_3, \mathcal{A}_{n_1,n_2})$ is $N = N_1 \cup N_2$, where
		\begin{align*}
			N_1 &= \{(u_{1,j}x_{n_1+1,j}y_{n_1+1,j+1} \parallel y_{1,j}y_{2,j}\cdots y_{n_1+1,j}x_{n_1+2,j}) \mid n_1+3 \le j \le n_1+n_2-1\},\\
			N_2 &= \{(u_{1,n_1}x_{n_1,n_1+1}y_{n_1,n_1+2} \parallel x_{1,n_1}x_{1,n_1+1}u_{1,n_1+2})\}.
		\end{align*}
	\end{enumerate}
	
	Consequently, we have
	\allowdisplaybreaks
	\begin{align*}
		\dim_{\Bbbk} \operatorname{Hom}(\Bbbk Q_0, \mathcal{A}_{n_1,n_2}) &= \frac{(n_1+n_2)(n_1+n_2-1)}{2}, \\
		\dim_{\Bbbk} \operatorname{Hom}(\Bbbk Q_1, \mathcal{A}_{n_1,n_2}) &= (n_1+n_2)(n_1+n_2-2) + (n_2-1), \\
		\dim_{\Bbbk} \operatorname{Hom}(\Bbbk S_2, \mathcal{A}_{n_1,n_2}) &= \frac{(n_1+n_2-3)(n_1+n_2-2)}{2} + (n_1-2) + 3(n_2-2), \\
		\dim_{\Bbbk} \operatorname{Hom}(\Bbbk S_3, \mathcal{A}_{n_1,n_2}) &= n_2-2.
	\end{align*}
	
	We first show that $\dim_{\Bbbk}\operatorname{Im}(\partial^2) = n_2-3$. For this, note that $\partial^2(T_i) =0$ for $i\neq 4$. Consider the image of $T_4$ under $\partial^2$. For each $n_1+3 \le j \le n_1+n_2-1$, we have
	\[
	\scalebox{0.9}{$
		\partial^2(u_{1,j}x_{n_1+1,j} \parallel y_{1,j}y_{2,j}\cdots y_{n_1,j}x_{n_1+1,j})(u_{1,k}x_{n_1+1,k}y_{n_1+1,k+1}) = 
		\begin{cases}
			y_{1,j}\cdots y_{n_1+1,j}x_{n_1+2,j}, & \text{if } k=j,\\
			0, & \text{if } k\neq j.
		\end{cases}
		$}
	\]
	Furthermore, a direct verification shows that for any element in $T$, its image under $\partial^2$ does not contain the basis element $(u_{1,n_1}x_{n_1,n_1+1}y_{n_1,n_1+2} \parallel x_{1,n_1}x_{1,n_1+1}u_{1,n_1+2}) \in N_2$ as a summand.
	
	This shows that for any $j$ with $n_1+3 \le j \le n_1+n_2-1$, we have
	\[
	\partial^2(u_{1,j}x_{n_1+1,j} \parallel y_{1,j}y_{2,j}\cdots y_{n_1,j}x_{n_1+1,j}) = (u_{1,j}x_{n_1+1,j}y_{n_1+1,j+1} \parallel y_{1,j}y_{2,j}\cdots y_{n_1+1,j}x_{n_1+2,j}).
	\]
	Consequently, $\partial^2$ maps $\operatorname{Span}_\Bbbk(T_4)$ surjectively onto $\operatorname{Span}_\Bbbk(N_1)$, while $\operatorname{Span}_\Bbbk(N_2)$ does not lie in the image of $\partial^2$.  Hence we have 
	\[
	\begin{aligned}
		\dim_{\Bbbk}\operatorname{Im}(\partial^2)&= |N_1| = n_2-3,\\
		\dim_{\Bbbk}\ker(\partial^2) &=  \dim_{\Bbbk}\operatorname{Hom}(\Bbbk S_2, \mathcal{A}_{n_1,n_2}) - (n_2-3)\\
		&= \frac{n_1^2 + 2n_1n_2 + n_2^2 - 3n_1 - n_2 - 4}{2}.
	\end{aligned}
	\]
	
	Then we examine $\partial^1$ as follows:
	\begin{enumerate}[leftmargin=*, itemindent=0pt, labelindent=0pt, labelsep=0.5em]
		\item By Lemma~\ref{lemma:T1T2_in_image}, it is immediate that $T_1\cup T_2 \subseteq \operatorname{Im}(\partial^1)$. More precisely, the summation construction demonstrates that $\partial^1(B_1)$ surjects onto $\operatorname{Span}_\Bbbk(T_1 \cup T_2)$ when restricted to the corresponding basis elements.
		\item Next we prove that $T_3 \subseteq \operatorname{Im}(\partial^1)$. For a basis element $(u_{1,j}y_{j,n_1+1} \parallel x_{1,j}u_{1,j+1})\in T_3$, we have 
		\[
		\partial^1\left(\sum\limits_{k=2}^{j} (u_{1,k} \parallel u_{1,k})\right) =(u_{1,j}y_{j,n_1+1} \parallel x_{1,j}u_{1,j+1}).
		\]
		Hence we have $T_3 \subseteq \operatorname{Im}(\partial^1)$.
		\item Finally, similar to the proof of (\ref{item:first}) in Lemma~\ref{lemma: HH of tilde A_{1,n_2}}, we can also show that $\partial^1(B_2)$ lies entirely in $\operatorname{Span}_\Bbbk(T_4 \cup T_5)$, and that all the vectors in the set 
		\[
		\{\partial^1(u_{1,j} \parallel y_{1,j}y_{2,j}\cdots y_{n_1,j}) \mid n_1+2 \leq j \leq n_1+n_2\}
		\]
		are linearly independent. Explicitly, the action of $\partial^1$ on $B_2$ is given as follows:
		
		\begin{itemize}[label=--]
			\item When $j = n_1+2$, we have
			\[
			\begin{aligned}
				&\partial^1(u_{1,n_1+2} \parallel y_{1,n_1+2}y_{2,n_1+2}\cdots y_{n_1,n_1+2})\\
				=& (u_{1,n_1+2}x_{n_1+1,n_1+2} \parallel y_{1,n_1+2}y_{2,n_1+2}\cdots y_{n_1,n_1+2}x_{n_1+1,n_1+2}).
			\end{aligned}
			\]

			\item When $n_1+3 \le j \le n_1+n_2-1$, the image consists of three terms:
			\[
			\begin{aligned}
				\partial^1(u_{1,j} \parallel y_{1,j}y_{2,j}\cdots y_{n_1,j}) &= - (u_{1,j-1}x_{n_1+1,j-1} \parallel y_{1,j-1}y_{2,j-1}\cdots y_{n_1,j-1}x_{n_1+1,j-1}) \\
				&\quad + (u_{1,j}y_{n_1+1,j} \parallel y_{1,j}y_{2,j}\cdots y_{n_1,j}y_{n_1+1,j}) \\
				&\quad + (u_{1,j}x_{n_1+1,j} \parallel y_{1,j}y_{2,j}\cdots y_{n_1,j}x_{n_1+1,j}).
			\end{aligned}
			\]
			
			\item When $j = n_1+n_2$, we have
			\[
			\begin{aligned}
				&\partial^1 (u_{1,n_1+n_2} \parallel y_{1,n_1+n_2}y_{2,n_1+n_2}\cdots y_{n_1,n_1+n_2}) \\
				=& - (u_{1,n_1+n_2-1}x_{n_1+1,n_1+n_2-1} \parallel y_{1,n_1+n_2-1}y_{2,n_1+n_2-1}\cdots y_{n_1,n_1+n_2-1}x_{n_1+1,n_1+n_2-1})\\
				&+ (u_{1,n_1+n_2}y_{n_1+1,n_1+n_2} \parallel y_{1,n_1+n_2}y_{2,n_1+n_2}\cdots y_{n_1,n_1+n_2} y_{n_1+1,n_1+n_2}).
			\end{aligned}
			\]
		\end{itemize}
		
		Since $B_2$ contains exactly $n_2 - 1$ elements and their images under $\partial^1$ are linearly independent, we conclude that \(\dim_\Bbbk \partial^1(B_2) = n_2 - 1.\)
	\end{enumerate}    
	
	As we have established that $\partial^1(B_1)$ covers $\operatorname{Span}_\Bbbk(T_1 \cup T_2 \cup T_3)$, and observed that the images of $B_2$ lie entirely in $\operatorname{Span}_\Bbbk(T_4 \cup T_5)$, we can represent $\partial^1$ by the following block diagonal structure:
	\[
	\begin{array}{c|cc}
		\partial^1 & T_1 \cup T_2 \cup T_3 & T_4 \cup T_5 \\
		\hline
		B_1 & * & 0 \\
		B_2 & 0 & *
	\end{array}
	\]
	From this structure, the rank of $\partial^1$ is simply the sum of $|T_1 \cup T_2\cup T_3|$ and $\dim_\Bbbk \partial^1(B_2)$:
	\[
	\dim_\Bbbk \operatorname{Im}(\partial^1) = |T_1 \cup T_2 \cup T_3| + \dim_\Bbbk \partial^1(B_2)=\frac{n_1^2 + 2n_1n_2 + n_2^2 - 3n_1 - n_2 - 4}{2}.
	\]
	Hence by comparing dimensions we have $\dim_\Bbbk \operatorname{Im}(\partial^1)=\dim_{\Bbbk}\ker(\partial^2)$.
	
	Finally, we obtain the Hochschild cohomology dimension
	\[
	\dim_{\Bbbk}\operatorname{HH}^i(\mathcal{A}_{n_1,n_2}) = \begin{cases}
		1, & \text{if $i=0,3,$}\\
		2, & \text{if $i=1,$}\\
		0, & \text{otherwise.}
	\end{cases}
	\]
	Explicit cocycles representing the nonzero classes are:
	\begin{itemize}[leftmargin=*, label=--]
		\item In $\operatorname{HH}^0(\mathcal{A}_{n_1, n_2})$: $\sum\limits_{1\le i<j\le n_1+n_2} (e_{i,j} \parallel e_{i,j})$,
		\item In $\operatorname{HH}^1(\mathcal{A}_{n_1, n_2})$: $\sum\limits_{i=2}^{n_1}(u_{1,i} \parallel u_{1, i})$ and $\sum\limits_{i=2}^{n_2}(u_{1,n_1+i} \parallel u_{1,n_1+i})$,
		\item In $\operatorname{HH}^3(\mathcal{A}_{n_1, n_2})$: $(u_{1,n_1}x_{n_1,n_1+1}y_{n_1,n_1+2} \parallel x_{1,n_1}x_{1,n_1+1}u_{1,n_1+2})$.
	\end{itemize}
\end{proof}

To summarize, we have the following result. 
\begin{theorem}\label{thm: main HH result}
	Let $n_2 \geq n_1 \geq 1$. The Hochschild cohomology of the symmetric square algebra \(\mathcal A_{n_1,n_2}\), viewed as an ungraded algebra with only \(m_2\), is concentrated in degrees \(0,1,2,3\). The dimensions $\dim_{\Bbbk} \operatorname{HH}^i(\mathcal{A}_{n_1,n_2})$  are completely classified in the following table:
	\begin{center}
		\renewcommand{\arraystretch}{1.3}
		\begin{tabular}{c|cccc}
			\hline
			$\mathcal{A}_{n_1,n_2}$ & $\operatorname{HH}^0$ & $\operatorname{HH}^1$ & $\operatorname{HH}^2$ & $\operatorname{HH}^3$ \\
			\hline
			$n_1=1,\; n_2=2$ & $1$ & $2$ & $0$ & $0$ \\
			$n_1=1,\; n_2\ge 3$ & $1$ & $1$ & $0$ & $0$ \\
			$n_1=2,\; n_2=2$ & $1$ & $3$ & $1$ & $1$ \\
			$n_1=2,\; n_2\ge 3$ & $1$ & $2$ & $1$ & $1$ \\
			$n_2\ge n_1 \ge 3$ & $1$ & $2$ & $0$ & $1$ \\
			\hline
		\end{tabular}
	\end{center}
\end{theorem}

\begin{proof}
	The explicit dimensions then follow directly by synthesizing the independent cases proven previously: the boundary case $n_1 = 1$ is established in Example \ref{ex: n_1=1,n_2=2} and Lemma~\ref{lemma: HH of tilde A_{1,n_2}}; the intermediate case $n_1 = 2$ is exactly Lemma~\ref{lemma:HH-A2-n2}; and the generic case $n_2 \ge n_1 > 2$ is covered by Lemma~\ref{lemma: HH of tilde A_{n1,n2}}.
\end{proof}

\subsection{Hochschild cohomology of $\widetilde{\mathcal A}_{n_1,n_2}$}\label{grading HH}

In the preceding sections, we computed the dimensions of the Hochschild cohomology groups for $\mathcal{A}_{n_1,n_2}$ by treating it as an ungraded algebra. However, as discussed in Subsection \ref{subsec:grading}, the symmetric square algebra naturally admits an internal $\mathbb{Z}$-grading. Consequently, the Hochschild cohomology is naturally refined into a \textit{bigraded} structure, denoted $\operatorname{HH}^{i,j}(\mathcal{A}_{n_1,n_2})$. Here, the first index $i$ represents the cohomological degree (arising from the length of the projective resolution), and the second index $j$ represents the internal degree (arising from the grading on the quiver). In this subsection, we investigate how the previously identified cocycles are distributed across these bidegrees. 

By Proposition \ref{prop:grading_derived_equiv}, any $\mathbb{Z}$-grading on the algebra is represented, up to graded derived equivalence, by one of the standard graded models $\widetilde{\mathcal A}_{n_1,n_2}'$. We therefore compute the bigraded Hochschild cohomology for this standard representative. For notational simplicity, we will drop the primes from the generators from now on. Recall that under this standard grading, all $x$ and $y$ arrows reside in internal degree zero ($|x_{i,j}| = |y_{i,j}| = 0$). Meanwhile, the arrows $u_{1,k}$ carry an internal degree of $m+1$ for $2 \leq k \leq n_1$, and an internal degree of $m$ for $n_1+2 \leq k \leq n_1+n_2$. The following result records how the Hochschild cohomology classes computed above are distributed with respect to this internal grading.

\begin{theorem}\label{thm:graded_HH}
	Let $\mathcal{A}_{n_1,n_2}$ be the symmetric square algebra equipped with the standard internal grading defined above, justified by Proposition \ref{prop:grading_derived_equiv}. The dimensions of the graded Hochschild cohomology groups $\operatorname{HH}^{i,j}(\mathcal{A}_{n_1,n_2})$ are classified in the following table:
	\begin{center}
		\renewcommand{\arraystretch}{1.3}
		\begin{tabular}{c|ccccc}
			\hline
			$\mathcal{A}_{n_1,n_2}$& $\operatorname{HH}^{0,0}$ & $\operatorname{HH}^{1,0}$ & $\operatorname{HH}^{3,-1}$ & $\operatorname{HH}^{1,-m}$ & $\operatorname{HH}^{2,m+1}$ \\
			\hline
			$n_1=1, n_2=2$ & $1$ & $1$ & $0$ & $1$ & $0$  \\
			$n_1=1, n_2\ge 3$ & $1$ & $1$ & $0$ & $0$ & $0$ \\
			$n_1=2, n_2=2$ & $1$ & $2$ & $1$ & $1$ & $1$ \\
			$n_1=2, n_2\ge 3$ & $1$ & $2$ & $1$ & $0$ & $1$ \\
			$n_2\ge n_1 \ge 3$ & $1$ & $2$ & $1$ & $0$ & $0$ \\
			\hline
		\end{tabular}
	\end{center}
\end{theorem}

\begin{proof}
	The table is obtained by tracking the bidegrees (the cohomological and internal degrees) of the explicit generators identified in our previous analysis.
	
	For instance, consider the symmetric square algebra $\mathcal{A}_{2,n_2}$ with $n_2 \ge 2$. As computed previously, the single generator of $\operatorname{HH}^2(\mathcal{A}_{2,n_2})$ is represented by the cocycle $(x_{1,2}y_{1,3} \parallel u_{1,2})$. The input path $x_{1,2}y_{1,3}$ consists only of $x$ and $y$ arrows, thus having an internal degree of $0$. The output arrow $u_{1,2}$, however, has an internal degree of $m+1$. Therefore, the internal degree of this homomorphism is $(m+1) - 0 = m+1$. Pairing this with its cohomological degree $2$, we deduce that this cocycle contributes exactly to the bidegree component $\operatorname{HH}^{2,m+1}(\mathcal{A}_{2,n_2})$. 
	
	Similarly, consider the algebra $\mathcal{A}_{2,2}$. According to Lemma \ref{lemma:HH-A2-n2}, one of the generators of $\operatorname{HH}^1(\mathcal{A}_{2,2})$ is explicitly given by the cocycle $(u_{1,4} \parallel y_{1,4}y_{2,4})$. In this case, the input arrow $u_{1,4}$ carries an internal degree of $m$, whereas the output path $y_{1,4}y_{2,4}$ has an internal degree of $0$. The internal degree of this map is thus $0 - m = -m$. Combined with its cohomological degree $1$, this generator resides in the bidegree $(1,-m)$, which yields the nontrivial entry for $\operatorname{HH}^{1,-m}(\mathcal{A}_{2,2})$.
	
	Furthermore, let us examine the nonzero class in \(\operatorname{HH}^3(\mathcal A_{n_1,n_2})\) that exists whenever \(n_1\geq 2\) and \(n_2\geq 2\). It is represented by the explicit cocycle
	\[
	(u_{1,n_1}x_{n_1,n_1+1}y_{n_1,n_1+2}
	\parallel
	x_{1,n_1}x_{1,n_1+1}u_{1,n_1+2}).
	\]
	For the input path, the arrow \(u_{1,n_1}\) has internal degree \(m+1\), while the arrows \(x_{n_1,n_1+1}\) and \(y_{n_1,n_1+2}\) both have internal degree \(0\). Hence the input has internal degree \(m+1\). On the other hand, the output path contains the arrow \(u_{1,n_1+2}\), which has internal degree \(m\), while the arrows \(x_{1,n_1}\) and \(x_{1,n_1+1}\) both have internal degree \(0\). Thus the output has internal degree \(m\). Therefore the internal degree of this cocycle is $m-(m+1)=-1.$
	Together with its Hochschild degree \(3\), this shows that the above cocycle represents the unique nonzero class in $\operatorname{HH}^{3,-1}(\mathcal A_{n_1,n_2}).$

	The bidegrees of the remaining generators can be verified in exactly the same way, yielding the full table.
\end{proof}

We now compute the Hochschild cohomology of the symmetric square dg algebra via
the spectral sequence described below; see \cite[Chapter~5]{Weibel1994} for the
general theory. We first recall the general construction
for computing the Hochschild cohomology of a dg algebra from its minimal
\(A_\infty\)-model. Let \(\widetilde{\mathcal A}=(\Bbbk Q/I, d)\) be a dg algebra over a field \(\Bbbk\). We denote its cohomology algebra by $\mathcal A=H^\ast(\widetilde{\mathcal A}).$ By Kadeishvili's minimal model theorem
\cite{Kadeishvili1980} 
(see also \cite[Section~3.3]{Keller2001} and
\cite[Theorem~1]{Petersen2020}), \(\mathcal A\) carries a minimal
\(A_\infty\)-structure \((\mathcal A,m_2,m_3,\ldots)\)
such that we have an \(A_\infty\)-quasi-isomorphism
\[
F:(\mathcal A,m_2,m_3,\ldots)\longrightarrow (\widetilde{\mathcal A},d,\mu_2).
\] Consequently, their Hochschild cohomologies are isomorphic:
\[
\operatorname{HH}^\ast(\widetilde{\mathcal{A}})
\cong
\operatorname{HH}^\ast(\mathcal A,m_2,m_3,\ldots).
\]
Assume that the (ungraded) algebra $\Bbbk Q/I$ admits a reduction system $$R =\{(s, \varphi(s)) \mid s \in S, \varphi(s) \in \Bbbk Q\}$$ satisfying the Diamond condition. We write
\(
C^{i,j}:=\Hom_{(\Bbbk Q_0)^e}^{j}(\Bbbk S_i,\mathcal A)
\)
for the cochains of Hochschild degree \(i\) and internal degree \(j\) appearing in \eqref{eq:CS-complex}. That is, an element in \(C^{i,j}\) sends a homogeneous element in \(\Bbbk S_i\) of  degree \(l\) to an element in \(\mathcal A\) of  degree \(l+j\).

The differential in the small Hochschild complex of the minimal $A_\infty$-model is given by
\[
D=[m_2+m_3+m_4+\cdots,-].
\]
We decompose it as
\(
D=\delta_0+\delta_1+\delta_2+\cdots
\), where \(\delta_r=[m_{r+2},-].
\) The operation \(m_{r+2}\) is homogeneous of internal degree \(-r\), then
\(
\delta_r:C^{i,j}\longrightarrow C^{i+r+1,j-r}.
\)
In particular,
\[
[m_2,-]:C^{i,j}\to C^{i+1,j},\qquad
[m_3,-]:C^{i,j}\to C^{i+2,j-1},\qquad
[m_4,-]:C^{i,j}\to C^{i+3,j-2}.
\]

We filter the total Hochschild cochain complex by internal degree. Let
\[
C^N=\bigoplus_{i+j=N}C^{i,j}
\]
and define a decreasing filtration
\[
F^pC^N=\bigoplus_{\substack{i+j=N\\ -j\ge p}}C^{i,j}.
\]
Equivalently, \(F^pC^N\) consists of cochains whose internal degree is at most \(-p\). Since
\[
\delta_0(F^pC^\bullet)\subseteq F^pC^\bullet,\qquad
\delta_r(F^pC^\bullet)\subseteq F^{p+r}C^\bullet\quad (r\ge 1),
\]
the total differential \(D\) preserves the filtration. Hence this filtration gives a spectral sequence. In the computational bidegree \((i,j)\), its \(E_0\)-page is \(E_0^{i,j}=C^{i,j},\) and the \(E_0\)-differential is induced by \(m_2\): \(d_0=[m_2,-]:E_0^{i,j}\longrightarrow E_0^{i+1,j}.\)
Hence the \(E_1\)-page is
\[
E_1^{i,j}
=
H^i(C^{\bullet,j},[m_2,-])
=
\operatorname{HH}^{i,j}(\mathcal A,m_2).
\]
Thus the \(E_1\)-page is precisely the bigraded Hochschild cohomology of the
underlying graded algebra \((\mathcal A,m_2)\).

The next differential is induced by the higher multiplication \(m_3\). If a class
\([\varphi]\in E_1^{i,j}\) is represented by \(\varphi\in C^{i,j}\), then
\(d_1([\varphi])\) is represented by \([m_3,\varphi]\). Hence \(d_1:E_1^{i,j}\longrightarrow E_1^{i+2,j-1}.\)

More generally, the \(r\)-th differential of the $E_r$-page is the first obstruction to lifting a
representative to a cocycle for the total Hochschild differential \(D\). More
precisely, let \([\varphi_0]\in E_r^{i,j}\) be represented by
\(\varphi_0\in C^{i,j}\). If one can choose cochains \(\varphi_\ell\in C^{i+\ell,j-\ell}\) for \( 1\leq \ell\leq r-1,\)
such that \(\sum\limits_{a=0}^{s}\delta_a\varphi_{s-a}=0\) for \(0\leq s\leq r-1,\)
then \(d_r([\varphi_0])\) is represented by \(\sum\limits_{a=1}^{r}\delta_a\varphi_{r-a}.\)
Thus \(d_r:E_r^{i,j}\longrightarrow E_r^{i+r+1,j-r}.\)

\begin{remark}
	We use the usual spectral sequence associated with a filtered cochain complex,
	but reindex the bidegrees according to the Hochschild bidegree used in our
	computation. If \((p,q)\) denotes the standard bidegree of this spectral
	sequence, as in \cite[Definition 5.2.3]{Weibel1994}, and \((i,j)\) denotes our
	Hochschild bidegree, then \(p=-j\) and \(q=i+2j.\)
	Under this change of coordinates, the standard differential becomes
	\(
	d_r:E_r^{i,j}\longrightarrow E_r^{i+r+1,j-r}.
	\)
	In particular, \(d_0:E_0^{i,j}\to E_0^{i+1,j}\).
\end{remark}

Under the usual boundedness or finite-dimensionality assumptions, the spectral
sequence associated with the internal degree filtration abuts to the Hochschild
cohomology of the minimal \(A_\infty\)-model. Since \((\mathcal A,m_2,m_3,\ldots)\) is a minimal \(A_\infty\)-model of the
original dg algebra \(\widetilde{\mathcal A}\), we have
\(
\operatorname{HH}^\ast(\widetilde{\mathcal A})
\cong
\operatorname{HH}^\ast(\mathcal A,m_2,m_3,\ldots).
\)
Therefore, at the level of dimensions, the stable page gives
\[
\dim_{\Bbbk} \operatorname{HH}^n(\widetilde{\mathcal A})
=
\sum\limits_{i+j=n}\dim_{\Bbbk}E_\infty^{i,j}.
\]

In the case considered below, \(m_k=0\) for all \(k\geq 4\). Hence the possible
higher differentials are obtained by iterating the contribution of
\([m_3,-]\). For example, if \([\varphi]\in E_2^{i,j}\) is represented by \(\varphi\in C^{i,j}\) and \([m_3,\varphi]=-[m_2,\psi_1]\)
for some \(\psi_1\in C^{i+1,j-1}\), then \(d_2([\varphi])\) is represented by \([m_3,\psi_1],\)
so \(d_2:E_2^{i,j}\longrightarrow E_2^{i+3,j-2}.\)
Since in our application the cochain complex in Remark \ref{rem:reduction system} is concentrated
in Hochschild degrees \(0,1,2,3\), we have \(d_r=0\) for all \(r\geq 3\). Thus \(E_3=E_\infty.\)

We now specialize the preceding discussion to
\(\widetilde{\mathcal A}_{n_1,n_2}\). The resulting computation is summarized in
the following theorem.
\begin{theorem}\label{thm:dg_HH_dimensions}
	Let \(\widetilde{\mathcal A}_{n_1,n_2}\) be the symmetric square dg algebra. Then the dimensions of the Hochschild cohomology of \(\widetilde{\mathcal A}_{n_1,n_2}\) are given in the following table:
	\[
	\begin{array}{c|cccc}
		\hline
		\widetilde{\mathcal A}_{n_1,n_2}
		&
		\operatorname{HH}^{0}
		&
		\operatorname{HH}^{1}
		&
		\operatorname{HH}^{1-m}
		&
		\operatorname{HH}^{m+3}
		\\
		\hline
		n_1=1,\ n_2=2
		&
		1 & 1 & 1 & 0
		\\
		n_1=1,\ n_2\ge 3
		&
		1 & 1 & 0 & 0
		\\
		n_1=2,\ n_2=2
		&
		1 & 1 & 1 & 1
		\\
		n_1=2,\ n_2\ge 3
		&
		1 & 1 & 0 & 1
		\\
		n_2\ge n_1\ge 3
		&
		1 & 1 & 0 & 0
		\\
		\hline
	\end{array}
	\]
	All Hochschild cohomology groups not listed in the table are zero. If two of
	the displayed degrees coincide for a special value of \(m\), then the
	corresponding dimensions should be added.
	
	More explicitly, in the non-formal case \(n_1\ge 2\), the stable page is
	represented by the following surviving classes:
	\[
	\sum\limits_{1\le i<j\le n_1+n_2}(e_{i,j}\parallel e_{i,j}) \quad \text{in } \operatorname{HH}^0,
	\]
	and
	\[
	\sum\limits_{i=2}^{n_1}(u_{1,i}\parallel u_{1,i})
	+
	\sum\limits_{i=2}^{n_2}(u_{1,n_1+i}\parallel u_{1,n_1+i}) \quad \text{in } \operatorname{HH}^1,
	\]
	together with the following additional classes when they exist:
	\[
	(u_{1,4}\parallel y_{1,4}y_{2,4}) \quad \text{in } \operatorname{HH}^{1-m}
	\quad
	\text{if }(n_1,n_2)=(2,2),
	\]
	and
	\[
	(x_{1,2}y_{1,3}\parallel u_{1,2}) \quad \text{in } \operatorname{HH}^{m+3}
	\quad
	\text{if }n_1=2.
	\]
\end{theorem}

\begin{proof}
	If \(n_1=1\), then by Theorem~\ref{thm:A_infty_qiso},
	\(\widetilde{\mathcal A}_{1,n_2}\) is formal. Hence
	\(
	\operatorname{HH}^\ast(\widetilde{\mathcal A}_{1,n_2})
	\cong
	\operatorname{HH}^\ast(\mathcal A_{1,n_2},m_2).
	\)
	Therefore the desired dimensions are obtained directly from the bigraded
	Hochschild cohomology of the graded algebra \((\mathcal A_{1,n_2},m_2)\) in
	Theorem~\ref{thm:graded_HH}, by passing from bidegree \((i,j)\) to total degree
	\(i+j\). This proves the first two rows of the table.
	
	Now suppose that \(n_1\ge 2\). By Theorem~\ref{thm:A_infty_qiso},
	\(\widetilde{\mathcal A}_{n_1,n_2}\) admits a minimal \(A_\infty\)-model \[(\mathcal A_{n_1,n_2},m_2,m_3),\] with \(
	m_r=0 \) for \(r\ge 4.\)
	Thus
	\(
	\operatorname{HH}^\ast(\widetilde{\mathcal A}_{n_1,n_2})
	\cong
	\operatorname{HH}^\ast(\mathcal A_{n_1,n_2},m_2,m_3).
	\)
	
	We use the spectral sequence associated with the internal degree filtration on
	the Hochschild cochain complex of this minimal model. Its \(E_1\)-page is
	\(
	E_1^{i,j}
	=
	\operatorname{HH}^{i,j}(\mathcal A_{n_1,n_2},m_2),
	\)
	which is computed in Theorem~\ref{thm:graded_HH}. By Remark~\ref{rem:reduction system},
	the complex is concentrated in Hochschild degrees \(i=0,1,2,3\). Hence the only
	possible nonzero \(d_1\)-differentials are
	\[
	d_1:E_1^{0,j}\longrightarrow E_1^{2,j-1},
	\qquad
	d_1:E_1^{1,j}\longrightarrow E_1^{3,j-1}.
	\]
	
	The only nontrivial contribution comes from the ternary operation
	\[
	m_3(u_{1,n_1},x_{n_1,n_1+1},y_{n_1,n_1+2})
	=
	(-1)^m x_{1,n_1}x_{1,n_1+1}u_{1,n_1+2}.
	\]
	Set
	\(
	\theta_1=
	\sum\limits_{i=2}^{n_1}(u_{1,i}\parallel u_{1,i})\), \(\theta_2=
	\sum\limits_{i=2}^{n_2}(u_{1,n_1+i}\parallel u_{1,n_1+i}),
	\)
	and
	\[
	\omega=
	(u_{1,n_1}x_{n_1,n_1+1}y_{n_1,n_1+2}
	\parallel
	x_{1,n_1}x_{1,n_1+1}u_{1,n_1+2}).
	\]
	
	A direct computation of the bracket with \(m_3\) gives \(d_1(\theta_1)=(-1)^m\omega\) and \(d_1(\theta_2)=-(-1)^m\omega.\) Hence \(d_1(\theta_1+\theta_2)=0,\) while the class \(\omega\) is killed on the \(E_2\)-page.
	
	All other classes listed in Theorem~\ref{thm:graded_HH} either have no possible
	target for degree reasons or have zero bracket with \(m_3\). Therefore the
	surviving class in degree \(1\) is \(\theta_1+\theta_2.\)
	Moreover, the class \((u_{1,4}\parallel y_{1,4}y_{2,4})\)
	survives in the exceptional case \((n_1,n_2)=(2,2)\), contributing to
	\(\operatorname{HH}^{1-m}\), and the class \((x_{1,2}y_{1,3}\parallel u_{1,2})\)
	survives whenever \(n_1=2\), contributing to \(\operatorname{HH}^{m+3}\).
	
	The next possible differential is \(d_2:E_2^{i,j}\longrightarrow E_2^{i+3,j-2}.\)
	The only class in Hochschild degree \(3\) on the \(E_1\)-page is \(\omega\), and
	\(\omega\) lies in the image of \(d_1\). Hence \(E_2^{3,*}=0.\) Therefore \(d_2\) has no nonzero target, and so \(d_2=0\). All higher
	differentials vanish for degree reasons, since the complex is concentrated in
	Hochschild degrees \(0,1,2,3\). Thus \(E_2=E_3=E_\infty.\)
	
	Finally, the dimensions of the Hochschild cohomology of
	\(\widetilde{\mathcal A}_{n_1,n_2}\) are obtained by summing the dimensions of
	the stable page along total degree:
	\[
	\dim_{\Bbbk} \operatorname{HH}^n(\widetilde{\mathcal A}_{n_1,n_2})
	=\sum\limits_{i+j=n}\dim_{\Bbbk} E_\infty^{i,j}.
	\]
	Using the surviving classes described above gives exactly the stated table.
\end{proof}

\begin{remark}
	It was shown in \cite{schrollsolotarwen} that, for every symmetric product of a disk with stops, the Hochschild cohomology of the associated partially wrapped Fukaya category is one-dimensional and concentrated in degree zero. Consequently, the corresponding higher Auslander algebra is rigid, in the sense that it admits no nontrivial deformations. In contrast, Theorem~\ref{thm:dg_HH_dimensions} shows that, for the symmetric square of an annulus with stops, the Hochschild cohomology can be nonzero in several degrees.
	
	This naturally raises the question of the geometric meaning of these additional Hochschild cohomology classes. For partially wrapped Fukaya categories of surfaces, certain generators of Hochschild cohomology admit a geometric interpretation in terms of the boundary data: contributions arise both from boundary components carrying exactly one stop and from boundary components carrying no stops; see \cite{Chaparro-Schroll-Solotar-SuarezAlvarez2026} for the ungraded case and \cite{Bian-Schroll-Solotar-Wang-Wen2026} for the graded setting. It would be interesting to determine whether the additional classes arising from the symmetric square of an annulus admit a similar geometric interpretation. 
\end{remark}

\subsection{Dg deformations of the symmetric square dg algebras}\label{deformed-alg}

Using our computation of Hochschild cohomology, we now describe the nontrivial degree-two deformation directions of the symmetric square dg algebras \(\widetilde{\mathcal A}_{n_1,n_2}\). Throughout this subsection, we use the grading convention from Subsection~\ref{subsec:grading}:
\(|x_{i,j}|=|y_{i,j}|=0\), \(|z_{i,i+1}|=-1,\)
and
\[
|u_{1,k}|= \begin{cases}
	m+1 & \text{for $2\leq k\leq n_1$},\\ 
	m & \text{for $n_1+2\leq k\leq n_1+n_2$}.
\end{cases}
\]
The degree-two Hochschild cohomology is nonzero only in the following cases:
\[
(n_1,n_2,m)=(1,2,-1),\qquad (2,2,-1),\qquad (2,n_2,-1)\quad (n_2\geq 3).
\]
Equivalently, the corresponding nontrivial deformation directions are represented by the cocycles listed below. In the case \((n_1,n_2,m)=(2,2,-1)\), two independent degree-two directions occur.

\begin{enumerate}[label=(\arabic*), leftmargin=*]
	\item \((n_1,n_2,m)=(1,2,-1)\). The cocycle $(u_{1,3}\parallel y_{1,3})$ induces a dg deformation \(\widetilde{\mathcal A}_{1,2}^{\mathrm{def}}\)  of \(\widetilde{\mathcal A}_{1,2}\). Namely, \(\widetilde{\mathcal A}_{1,2}^{\mathrm{def}} =\widetilde{\mathcal A}_{1,2}\) as graded algebras and the differential agrees with the original one on all generators except that \(d_{\mathrm{def}}(u_{1,3})=y_{1,3}.\)
	See Figure~\ref{fig:deformed-quiver-1-2}.
	
	\item \((n_1,n_2,m)=(2,2,-1)\). The cocycle $(u_{1,4}\parallel y_{1,4}y_{2,4})$ also induces a dg deformation  \(\widetilde{\mathcal A}_{2,2}^{\mathrm{def}}\) of  \(\widetilde{\mathcal A}_{2,2}\) so that the differential agrees with the original one except that \(d_{\mathrm{def}}(u_{1,4})=y_{1,4}y_{2,4}.\)
	See Figure~\ref{fig:deformed-quiver-2-2}.

	\item \((n_1,n_2,m)=(2,n_2,-1)\) with \(n_2\geq 2\). The cocycle $(x_{1,2}y_{1,3}\parallel u_{1,2})$ gives an associative deformation of the multiplication on the $A_\infty$-algebra $\mathcal A_{2, n_2}$. Namely, in the deformed algebra $\mathcal A_{2, n_2}^{\mathrm{def}}$ we have a new non-admissible relation $x_{1,2} y_{1,3}=u_{1,2}$ and all the other relations are unchanged.  Note that this deformation can be induced by a dg deformation   \(\widetilde{\mathcal A}_{2,n_2}^{\mathrm{def}}\) over the dg algebra \(\widetilde{\mathcal A}_{2,n_2}\) given by \(d_{\mathrm{def}}(z_{1,2}) = x_{1,2} y_{1,3} - u_{1,2}.\) 
\end{enumerate}

\begin{figure}[H] 
	\captionsetup{width=1.2\linewidth} \centering 
	
	\begin{minipage}[b]{0.48\textwidth} 
		\centering 
		\begin{tikzcd}[column sep=large, row sep=normal] 
			L_{1,2} \ar[r, "x_{1,2}"] \ar[dr, red, "z_{1,2}", swap]& L_{1,3} \ar[d, "y_{1,3}"] \ar[blue, d, bend right=30, "u_{1,3}" swap] & \\ & L_{2,3} 
		\end{tikzcd}
		
		\vspace{0.15cm}
		$d_{\mathrm{def}}(u_{1,3})=y_{1,3}.$ 
		
		\caption{Deformation of $\widetilde{\mathcal{A}}_{1,2}$.} 
		\label{fig:deformed-quiver-1-2} 
	\end{minipage} 
	\hfill 
	\begin{minipage}[b]{0.48\textwidth} 
		\centering 
		\begin{tikzcd} 
			L_{1,2} \ar[r, "x_{1,2}"] \ar[blue, dr, "u_{1,2}", swap,bend right=50] \ar[dr,"z_{1,2}",red, swap]& L_{1,3} \ar[r, "x_{1,3}"] \ar[d, "y_{1,3}"]& L_{1,4} \ar[d, "y_{1,4}"] \ar[blue, dd, bend right=40, "u_{1,4}" {xshift=2pt, yshift=-8pt}, swap]\\ & L_{2,3} \ar[dr,"z_{2,3}",red, swap] \ar[r, "x_{2,3}"] & L_{2,4} \ar[d, "y_{2,4}"]\\ && L_{3,4} 
		\end{tikzcd}
		
		\vspace{0.15cm} 
		$d_{\mathrm{def}}(u_{1,4})=y_{1,4}y_{2,4}.$ 
		
		\caption{Deformation of $\widetilde{\mathcal{A}}_{2,2}$.} 
		\label{fig:deformed-quiver-2-2} 
	\end{minipage} 
	
\end{figure}

\subsection*{Acknowledgments} The authors wish to express their sincere gratitude to Nanqing Ding for his guidance and support. The authors are also very grateful to Yuan Gao, Gustavo Jasso, Sibylle Schroll, Can Wen and Tianyu Yuan for many useful comments and discussions. This work was supported by the National Key R\&D Program of China (2024YFA1013803) and the NSFC (Nos.~12271243, 12371043).

\end{document}